\documentclass[a4paper, 10pt, parskip=half]{scrartcl}

\usepackage[utf8]{inputenc}
\usepackage[T1]{fontenc}
\usepackage{lmodern}
\usepackage{csquotes}
\usepackage{amsmath}
\usepackage{amssymb}
\usepackage{amsthm}
\usepackage{amsfonts}
\usepackage{dsfont}
\usepackage{mathtools}
\usepackage{marginnote}
\usepackage[dvipsnames]{xcolor}
\usepackage{hyperref}
\hypersetup{%
  colorlinks=true,
  linkcolor=black,
  citecolor=black,
  urlcolor=blue,
  pdftitle={Strong convergence of weighted gradients in parabolic equations and applications to global generalized solvability of cross-diffusive systems},
  pdfauthor={Mario Fuest},
  pdfkeywords={},
  bookmarksopen=true,
}

\allowdisplaybreaks

\usepackage{xcolor}

\usepackage[numbers, sort&compress]{natbib}
\newcommand{\R}{\mathbb{R}}

\newcommand{\N}{\mathbb{N}}

\newcommand{\mc}[1]{\mathcal{#1}}

\newcommand{\ur}[1]{\mathrm{#1}}
\newcommand{\ure}{\ur{e}}

\ifdefined\labelenumi%
  \renewcommand{\labelenumi}{(\roman{enumi})}
  
\fi

\newcommand{\eps}{\varepsilon}

\newcommand{\gt}{>}
\newcommand{\lt}{<}

\DeclareMathOperator{\supp}{supp}

\newcommand{\defs}{\coloneqq}
\newcommand{\sfed}{\eqqcolon}

\newcommand{\ra}{\rightarrow}

\newcommand{\sea}{\searrow}

\newcommand{\rh}{\rightharpoonup}

\newcommand{\ol}{\overline}
\newcommand{\ul}{\underline}

\newcommand{\dx}{\,\mathrm{d}x}
\newcommand{\ds}{\,\mathrm{d}s}
\newcommand{\dt}{\,\mathrm{d}t}

\newcommand{\dtau}{\,\mathrm{d}\tau}
\newcommand{\dsigma}{\,\mathrm{d}\sigma}

\newcommand{\ddt}{\frac{\mathrm{d}}{\mathrm{d}t}}

\DeclareMathOperator{\sign}{sign}

\newcommand{\embed}{\hookrightarrow}

\newcommand{\hp}{\hphantom}
\newcommand{\pe}{\mathrel{\hp{=}}}

\newcommand{\intom}{\int_\Omega}

\newcommand{\intnt}{\int_0^T}
\newcommand{\intnst}{\int_0^t}

\newcommand{\intntom}{\int_0^T \int_\Omega}
\newcommand{\intnstom}{\int_0^t \int_\Omega}
\newcommand{\intninfom}{\int_0^\infty \int_\Omega}

\newcommand{\Ombar}{\ol \Omega}
\newcommand{\Ominf}{\Omega \times (0, \infty)}
\newcommand{\OmT}{\Omega \times (0, T)}
\newcommand{\Ombarinf}{\Ombar \times [0, \infty)}
\newcommand{\OmbarT}{\Ombar \times [0, T]}
\newcommand{\loc}{\mathrm{loc}}

\newcommand{\leb}[2][\Omega]{\ensuremath{L^{#2}(#1)}}

\newcommand{\sob}[3][\Omega]{\ensuremath{W^{#2, #3}(#1)}}

\newcommand{\con}[2][\Ombar]{\ensuremath{C^{#2}(#1)}}

\newcommand{\dual}[1]{\ensuremath{(#1)^\star}}

\newcommand{\ue}{u_\eps}
\newcommand{\uet}{u_{\eps t}}
\newcommand{\une}{u_{0 \eps}}
\newcommand{\ve}{v_\eps}
\newcommand{\vet}{v_{\eps t}}
\newcommand{\vne}{v_{0 \eps}}

\newcommand{\we}{w_\eps}
\newcommand{\wej}{w_{\eps_j}}
\newcommand{\wet}{w_{\eps t}}

\newcommand{\uej}{u_{\eps_j}}

\newcommand{\unej}{u_{0 \eps_j}}
\newcommand{\vej}{v_{\eps_j}}
\newcommand{\vejt}{v_{\eps_j t}}
\newcommand{\vnej}{v_{0 \eps_j}}

\newcommand{\fe}{f_\eps}
\newcommand{\fej}{f_{\eps_j}}

\newcommand{\infv}{M}
\newcommand{\tops}{\texorpdfstring}

\makeatletter
\renewenvironment{proof}[1][\proofname]{\par
  \pushQED{\qed}%
  \normalfont \topsep0\p@\relax
  \trivlist
  \item[\hskip\labelsep\scshape
  #1\@addpunct{.}]\ignorespaces
}{%
  \popQED\endtrivlist\@endpefalse
}
\makeatother

\newtheorem{base}{Base}[section]
\numberwithin{equation}{section}

\newtheorem{theorem}[base]{Theorem} \newtheorem*{theorem*}{Theorem}
\newtheorem{lemma}[base]{Lemma} \newtheorem*{lemma*}{Lemma}
 \newtheorem*{prop*}{Proposition}
 \newtheorem*{cor*}{Corollary}

\theoremstyle{definition}
\newtheorem{remark}[base]{Remark} \newtheorem*{remark*}{Remark}
\newtheorem{definition}[base]{Definition} \newtheorem*{definition*}{Definition}
 \newtheorem*{example*}{Example}
 \newtheorem*{cond*}{Condition}

\newif\ifhyperconst             

\makeatletter
\newcounter{globalconst}[section]
\newcommand{\newgc}[2][]{%
\refstepcounter{globalconst}%
\ltx@label{gc:\thesection:#2}%
\ifhyperconst
  \hyperref[gc:\thesection:#2]{C}_{\ref{gc:\thesection:#2}#1}
\else
  C_{\begin{NoHyper}\ref{gc:\thesection:#2}\end{NoHyper}}
\fi}
\ifhyperconst
  \newcommand{\gc}[2][]{\hyperref[gc:\thesection:#2]{C}_{\ref{gc:\thesection:#2}#1}}
\else
  \newcommand{\gc}[2][]{C_{\begin{NoHyper}\ref{gc:\thesection:#2}\end{NoHyper}}}
\fi

\newcounter{localconst}[base]
\newcommand{\newlc}[2][]{%
\refstepcounter{localconst}%
\ltx@label{lc:\thesection:\arabic{base}:#2}%
\ifhyperconst
  \hyperref[lc:\thesection:\arabic{base}:#2]{c}_{\ref{lc:\thesection:\arabic{base}:#2}#1}
\else
  c_{\begin{NoHyper}\ref{lc:\thesection:\arabic{base}:#2}\end{NoHyper}}
\fi}
\ifhyperconst
  \newcommand{\lc}[2][]{\hyperref[lc:\thesection:\arabic{base}:#2]{c}_{\ref{lc:\thesection:\arabic{base}:#2}#1}}
\else
  \newcommand{\lc}[2][]{c_{\begin{NoHyper}\ref{lc:\thesection:\arabic{base}:#2}\end{NoHyper}}}
\fi
\makeatother

\begin{document}
\setkomafont{title}{\normalfont\Large}
\title{Strong convergence of weighted gradients in parabolic equations and applications to global generalized solvability of cross-diffusive systems}
\author{%
Mario Fuest\footnote{fuest@ifam.uni-hannover.de}\\
{\small Leibniz Universität Hannover, Institut für Angewandte Mathematik} \\
{\small Welfengarten 1, 30167 Hannover, Germany}
}
\date{}

\maketitle

\KOMAoptions{abstract=true}
\begin{abstract}
\noindent
  In the first part of the present paper,
  we show that strong convergence of $(v_{0 \varepsilon})_{\varepsilon \in (0, 1)}$ in $L^1(\Omega)$
  and weak convergence of $(f_{\varepsilon})_{\varepsilon \in (0, 1)}$ in $L_{\textrm{loc}}^1(\overline \Omega \times [0, \infty))$
  not only suffice to conclude that solutions to the initial boundary value problem
  \begin{align*}
    \begin{cases}
      v_{\varepsilon t} = \Delta v_\varepsilon + f_\varepsilon(x, t) & \text{in $\Omega \times (0, \infty)$}, \\
      \partial_\nu v_\varepsilon = 0                                 & \text{on $\partial \Omega \times (0, \infty)$}, \\
      v_\varepsilon(\cdot, 0) = v_{0 \varepsilon}                    & \text{in $\Omega$},
    \end{cases}
  \end{align*}
  which we consider in smooth, bounded domains $\Omega$,
  converge to the unique weak solution of the limit problem,
  but that also certain weighted gradients of $v_\varepsilon$ converge strongly in $L_{\textrm{loc}}^2(\overline \Omega \times [0, \infty))$ along a subsequence.\\[0.5em]
  We then make use of these findings to obtain global generalized solutions to various cross-diffusive systems.
  Inter alia, we establish global generalized solvability of the system
  \begin{align*}
    \begin{cases}
      u_t = \Delta u - \chi \nabla \cdot (\frac{u}{v} \nabla v) + g(u), \\
      v_t = \Delta v - uv,
    \end{cases}
  \end{align*}
  where $\chi > 0$ and $g \in C^1([0, \infty))$ are given,
  merely provided that ($g(0) \ge 0$ and) $-g$ grows superlinearily.
  This result holds in all space dimensions and does neither require any symmetry assumptions nor the smallness of certain parameters.
  Thereby, we expand on a corresponding result for quadratically growing $-g$ proved by Lankeit and Lankeit (Nonlinearity, \textbf{32(5)}:1569–1596, 2019).
  \\[0.5pt]
 \textbf{AMS Classification (2020):} {35A35 (primary); 35D99, 35K05, 35K55, 92C17 (secondary)}
\end{abstract}

\newpage
\tableofcontents
\section{Introduction}\label{sec:intro}
\subsection{Part I: Strong convergence of weighted gradients}\label{sec:intro:part1}
The first part of the present paper consists in analyzing convergence properties of classical solutions $\ve$ to the initial boundary value problem
\begin{align}\label{prob:ve}
  \begin{cases}
    \vet = \Delta \ve - \kappa \ve + \fe(x, t) & \text{in $\Omega \times (0, \infty)$}, \\
    \partial_\nu \ve = 0                       & \text{on $\partial \Omega \times (0, \infty)$}, \\
    \ve(\cdot, 0) = \vne                       & \text{in $\Omega$},
  \end{cases}
\end{align}
where $\Omega$ is a smooth bounded domain, $\kappa \in \R$, $\fe$ and $\vne$ are given sufficiently regular functions and $\eps \in (0, 1)$ is an approximation parameter.

Making use of various a priori estimates (which we collect in Section~\ref{sec:apriori}) and compactness theorems,
we show in Subsection~\ref{sec:conv_l1:weak_conv} that if $\vne \ra v_0$ in $\leb1$ and $\fe \rh f$ in $L_{\loc}^1(\Ombarinf)$ as $\eps \sea 0$,
then $(\ve)_{\eps \in (0, 1)}$ converges (along some null sequence) inter alia pointwise a.e.\ to the unique weak solution $v$ of the limit problem
\begin{align}\label{prob:v_limit}
  \begin{cases}
    v_t = \Delta v - \kappa v + f(x, t) & \text{in $\Omega \times (0, \infty)$}, \\
    \partial_\nu v = 0                  & \text{on $\partial \Omega \times (0, \infty)$}, \\
    v(\cdot, 0) = v_0                   & \text{in $\Omega$}.
  \end{cases}
\end{align}
While these results can obviously be used to guarantee the unique weak solvability of problems of the form \eqref{prob:v_limit},
that alone would be only of limited interest,
as there are already several ways to develop a satisfactory existence theory for the heat equation with integrable data.

The point is that we go beyond convergence properties immediately implied by easily obtained a priori estimates combined with well-known weak compactness theorems and for instance also verify that
\begin{align}\label{eq:intro:strong_grad_conv}
  \mathds 1_{\{|\ve| \le k\}} \nabla \ve &\ra \mathds 1_{\{|v| \le k\}} \nabla v
  \qquad \text{\emph{strongly} in $L_{\loc}^2(\Ombarinf)$ for all $k \in \N$}
\end{align}
as $\eps = \eps_j \sea 0$ for some null sequence $(\eps_j)_{j \in \N} \subset (0, 1)$.
(Here and below, $\mathds 1_A$ denotes the characteristic function on a set $A$ which equals $1$ on $A$ and $0$ elsewhere.)
Of course, this in turn raises the question why convergence properties of this type are interesting.
While similar results have led to global existence and uniqueness results for certain parabolic equations (which we review at the beginning of Subsection~\ref{sec:conv_l1:strong_grad_conv}),
proving the existence of a unique weak solution of \eqref{prob:v_limit} is possible without making use of \eqref{eq:intro:strong_grad_conv}, as we have already noted above.

The usefulness of \eqref{eq:intro:strong_grad_conv} and related statements becomes apparent when \eqref{prob:v_limit} appears as a subproblem in a system of parabolic equations,
for which one desires to obtain certain global generalized solutions.
Indeed, Theorem~\ref{th:strong_grad_conv} is a corner stone in the global existence proofs for three such systems
considered in Sections~\ref{sec:a1}--\ref{sec:a3}.
A brief introduction to these systems and a more detailed explanation how exactly convergence properties such as \eqref{eq:intro:strong_grad_conv} can be made use of is given in Subsection~\ref{sec:intro:part2} below.
Before further discussing its consequences and applications, however, let us state our main theorem.
\begin{theorem}\label{th:strong_grad_conv}
  Let $\Omega \subset \R^n$, $n \in \N$, be a smooth, bounded domain and $\kappa \in \R$.
  For $\eps \in (0, 1)$, let $\vne \in \con0$, $\fe \in C^0(\Ombarinf)$
  and suppose that
  $\ve \in C^0(\Ombarinf) \cap C^{2, 1}(\Ombar \times (0, \infty))$
  is a classical solution of \eqref{prob:ve}.
  Moreover, assume that there are $v_0 \in \leb1$ and $f \in L_{\loc}^1(\Ombarinf)$
  such that
  \begin{alignat}{2}
    \vne & \ra v_0  && \qquad \text{in $\leb1$}, \label{eq:strong_grad_conv:v0_conv} \\
    \fe  & \rh f    && \qquad \text{in $L_{\loc}^1(\Ombarinf)$} \label{eq:strong_grad_conv:f_conv}
  \end{alignat}
  as $\eps = \eps_j' \sea 0$ for some null sequence $(\eps_j')_{j \in \N}$.
  
  Then there exist
  a subsequence $(\eps_j)_{j \in \N}$ of $(\eps_j')_{j \in \N}$
  and a function
  \begin{align}\label{eq:strong_grad_conv_gen:v_reg}
    v \in L_{\loc}^1(\Ombarinf)
    \quad \text{with} \quad
    \begin{cases}
      \nabla v \in L_{\loc}^\lambda(\Ombarinf)                      & \text{for all $\lambda \in [1, \frac{n+2}{n+1})$}, \\
      \mathds 1_{\{|v| \le k\}} \nabla v \in L_{\loc}^2(\Ombarinf)  & \text{for all $k \in \N$}, \\
      (|v|+1)^{-r} \nabla v \in L_{\loc}^2(\Ombarinf)               & \text{for all $r > \frac12$}
    \end{cases}
  \end{align}
  such that
  \begin{alignat}{2}
    \ve &\ra v
    &&\qquad \text{in $L^1(\Ombarinf)$ and a.e.\ in $\Omega \times (0, \infty)$}, \label{eq:strong_grad_conv:v_l1_conv}\\
    \nabla \ve &\ra \nabla v
    &&\qquad \text{in $L_{\loc}^\lambda(\Ombarinf)$ for all $\lambda \in [1, \tfrac{n+2}{n+1})$ and a.e.\ in $\Omega \times (0, \infty)$}, \label{eq:strong_grad_conv:v_pointwise_grad_conv}\\
    \mathds 1_{\{|\ve| \le k\}} \nabla \ve &\ra \mathds 1_{\{|v| \le k\}} \nabla v
    &&\qquad \text{in $L_{\loc}^2(\Ombarinf)$ for all $k \in \N$}, \label{eq:strong_grad_conv:v_grad_conv}\\
    \mathds (|\ve|+1)^{-r} \nabla \ve &\ra (|v|+1)^{-r} \nabla v
    &&\qquad \text{in $L_{\loc}^2(\Ombarinf)$  for all $r > \tfrac12$} \label{eq:strong_grad_conv:v_weighted_grad_conv}
  \end{alignat}
  as $\eps = \eps_j \sea 0$.
  Moreover, $v$ is the unique weak solution with regularity $L_{\loc}^1([0, \infty); \sob11)$ of \eqref{prob:v_limit} in the sense that
  \begin{align}\label{eq:strong_grad_conv:weak_sol}
      - \intninfom v \varphi_t
      - \intom v_0 \varphi(\cdot, 0)
    = - \intninfom \nabla v \cdot \nabla \varphi
      - \kappa \intninfom v \varphi
      + \intninfom f \varphi
  \end{align}
  holds for all $\varphi \in C_c^\infty(\Ombarinf)$.
\end{theorem}

\begin{remark}
  \begin{enumerate}
    \item
      Let us again emphasize that \eqref{eq:strong_grad_conv:v_l1_conv}--\eqref{eq:strong_grad_conv:v_weighted_grad_conv} assert \emph{strong} convergence in the respective spaces.

    \item
      For typical applications of Theorem~\ref{th:strong_grad_conv} (such as those stated in Subsection~\ref{sec:intro:part2}) one has some control in choosing the approximative data $\vne$ and $\fe$,
      thus requiring continuity of these functions and that $\ve$ is a classical solution of \eqref{prob:ve} in Theorem~\ref{th:strong_grad_conv} does not appear to be a huge restriction.
      Let us note, however, that since the proof mainly rests on several testing procedures and thus could be carried out also for sufficiently regular weak solutions,
      one can obtain a similar result also for less regular $\vne$ and $\fe$.

    \item 
      We prove in Lemma~\ref{lm:limit_ve_c0_l1} that if in addition to the assumptions in Theorem~\ref{th:strong_grad_conv} it holds that
      \begin{align*}
        \fe  & \ra f
        \qquad \text{in $L_{\loc}^1(\Ombarinf)$ as $\eps = \eps' \sea 0$},
      \end{align*}
      then $v$ belongs to $C^0([0, \infty); \leb1)$ and 
      \begin{align*}
        \ve &\ra v
        \qquad \text{in $C^0([0, \infty); \leb1)$ as $\eps = \eps' \sea 0$}. 
      \end{align*}
  \end{enumerate}
\end{remark}

Next, in Section~\ref{sec:stronger_grad_conv},
we show that (and how) stronger convergence properties of $(\vne)_{\eps \in (0, 1)}$ and $(\fe)_{\eps \in (0, 1)}$ than those required by \eqref{eq:strong_grad_conv:v0_conv} and \eqref{eq:strong_grad_conv:f_conv}
imply stronger convergence properties of the solutions.
Our interest is again partly motivated by an application to global solvability of cross-diffusive systems.
Indeed, in Section~\ref{sec:a3} we not only crucially rely on Theorem~\ref{th:strong_grad_conv} but also on the following theorem.

The main idea for its proof along with arising challenges and considerations how to overcome them is given at the beginning of Section~\ref{sec:stronger_grad_conv}.
Here, let us just note that to the best of our knowledge such as result (at the very least in this general form) is not only new for \eqref{prob:ve} but also for parabolic problems in general.
\begin{theorem}\label{th:stronger_grad_conv}
  Suppose the hypotheses of Theorem~\ref{th:strong_grad_conv} hold
  and let $(\eps_j)_{j \in \N}$ and $v$ be as given by Theorem~\ref{th:strong_grad_conv}.
  Moreover, for $T > 0$ set $X_T \defs \bigcup_{j \in \N} \vej(\Ombar \times [0, T])$ as well as $X \defs \bigcup_{T > 0} X_T$
  and suppose $0 \le \psi \in C^2(X)$ is convex and such that
  \begin{align}\label{eq:stronger_grad_conv:psi_ra_infty}
    \lim_{X_T \ni s \ra \pm \infty} \psi(s) = \infty
  \end{align}
  and
  \begin{align}\label{eq:stronger_grad_conv:psi''_bdd}
    \sup_{s \in X_T} \psi''(s) \mathds 1_{\{\psi(s) \le C\}} < \infty 
    \qquad \text{for all $C > 0$ and $T > 0$}.
  \end{align}
  If
  \begin{alignat}{2}
    \psi(\vne) &\ra \psi(v_0)
    &&\qquad \text{in $\leb1$}, \label{eq:stronger_grad_conv:conv_v0} \\
    \psi'(\ve) \fe &\ra \psi'(v) f
    &&\qquad \text{in $L_{\loc}^1(\Ombarinf)$} \label{eq:stronger_grad_conv:conv_v_f}
  \end{alignat}
  as $\eps = \eps_j \sea 0$, then $\psi(v) \in \leb1$, $(\psi''(v))^\frac12 \nabla v \in L_{\loc}^2(\Ombarinf)$ and
  \begin{alignat}{2}
    \psi(\ve) &\ra \psi(v)
    &&\qquad \text{in $L_{\loc}^1(\Ombarinf)$}, \label{eq:stronger_grad_conv:conv_psi_v} \\
    (\psi''(\ve))^\frac12 \nabla \ve &\ra (\psi''(v))^\frac12 \nabla v
    &&\qquad \text{in $L_{\loc}^2(\Ombarinf)$} \label{eq:stronger_grad_conv:conv_psi''_nabla_v}
  \end{alignat}
  as $\eps = \eps_j \sea 0$.
\end{theorem}

\begin{remark}
  \begin{enumerate}
    \item
      As an exemplary application of Theorem~\ref{th:stronger_grad_conv}, let us suppose that (in addition to the hypotheses of Theorem~\ref{th:strong_grad_conv}) $\vej \ge 0$ for all $j \in \N$ and fix $p > 1$.
      We may then choose $\psi \colon [0, \infty) \ra [0, \infty), s \mapsto (s+1)^p$ in Theorem~\ref{th:stronger_grad_conv}.
      Thus, if
      \begin{alignat*}{2}
        \vne &\ra v_0
        &&\qquad \text{in $\leb p$}, \\
        \ve^{p-1} \fej &\ra v^{p-1} f
        &&\qquad \text{in $L_{\loc}^1(\Ombarinf)$ }
      \end{alignat*}
      as $\eps = \eps_j \sea 0$, then
      \begin{align*}
        (\vej+1)^\frac{p-2}{2} \nabla \vej &\ra (v+1)^\frac{p-2}{2} \nabla v
        \qquad \text{in $L_{\loc}^2(\Ombarinf)$ as $j \ra \infty$}.
      \end{align*}

    \item
      For $p=2$,
      this has essentially been shown by more direct means in \cite[Lemma~4.5]{WinklerRoleSuperlinearDamping2019} and also by use of Steklov averages in \cite[Lemma~8.2]{WinklerLargedataGlobalGeneralized2015}.
      However, these proofs do not immediately translate to other values of $p$.
  \end{enumerate}
\end{remark}

\subsection{Part II: Applications to global solvability of cross-diffusive systems}\label{sec:intro:part2}
In the second part of the present paper,
we show how Theorem~\ref{th:strong_grad_conv} and Theorem~\ref{th:stronger_grad_conv}
can be used to obtain global generalized solutions to certain cross-diffusive systems of the form
\begin{align}\label{prob:chemotaxis}
  \begin{cases}
    u_t = \Delta u - \nabla \cdot (S(u, v) \nabla v) + g(u, v), \\
    v_t = \Delta v + h(u, v),
  \end{cases}
\end{align}
where $S$, $g$ and $h$ are given functions.
Such models have first been introduced by Keller and Segel in \cite{KellerSegelInitiationSlimeMold1970} to model the aggregation of slime mold
and have since then been used to describe a variety of natural phenomena (cf.\ for instance the surveys \cite{HillenPainterUserGuidePDE2009} and \cite{BellomoEtAlMathematicalTheoryKeller2015}).
The key feature here is that an organism with density $u$ is assumed to respond chemotactically to, that is, to partially orient its movement in response to, a substrate with density $v$
and that the chemical is influenced (for instance, either produced or consumed) by the organism.
We postpone the concrete biological motivations for the systems considered in this paper to the introductions of the Sections~\ref{sec:a1}--\ref{sec:a3}.

A key feature of the so-called minimal Keller--Segel system (\eqref{prob:chemotaxis} with $S(u, v) = u$, $g \equiv 0$ and $h(u, v) = -v + u$)
is that when considered in multi-dimensional balls solutions may blow up in finite time
\cite{HerreroVelazquezBlowupMechanismChemotaxis1997, MizoguchiWinklerBlowupTwodimensionalParabolic, WinklerFinitetimeBlowupHigherdimensional2013};
for an overview of blow-up results and techniques also for related systems, we refer to the survey \cite{LankeitWinklerFacingLowRegularity2019}.

These results indicate that obtaining global classical or weak solutions to concrete chemotaxis systems may be a very hard,
if not an impossible task.
Thus, one often needs to settle for even weaker solution concepts.
A particular idea, going back to \cite{DiPernaLionsCauchyProblemBoltzmann1989} and requiring less regularity of the solution $(u, v)$, is to aim for \emph{renormalized solutions}
and thus to not consider the weak formulations for $u$ and $v$ but instead the corresponding one for $\phi(u, v)$ for, say, all $\phi \in C_c^\infty([0, \infty))$.
In the corresponding solution definition terms such as
\begin{align*}
  \intntom \phi_{uu}(u, v) S(u, v) \varphi \nabla u \cdot \nabla v
\end{align*}
appear and even if $u$ and $v$ are known to be sufficiently regular for this expression to make sense,
whether corresponding approximative terms converge towards this term is entirely unclear as long as only weak convergence of (weighted) gradients for both solution components has been shown.

This is the point where Theorem~\ref{th:strong_grad_conv} (applied to the second solution component)
and in particular \eqref{eq:strong_grad_conv:v_grad_conv} come to rescue:
If one factor converges weakly in $L^2$ and the other one strongly in $L^2$, the product converges weakly in $L^1$.

Still, Theorem~\ref{th:strong_grad_conv} does not allow us to handle terms involving $|\nabla u|^2$.
Thus, we follow \cite{WinklerLargedataGlobalGeneralized2015}
and settle for an even weaker solution concept, which only requires that $\phi(u, v)$ is a weak supersolution (for certain functions $\phi$),
which, when combined with an upper estimate for the mass of the first solution component and a (usual) weak formulation for the second one, turns out to still be a sensible definition;
see the introduction and discussion in Subsection~\ref{sec:a1:sol_concept} (for precisely this concept) and in Subsection~\ref{sec:a2:sol_concept} and Subsection~\ref{sec:a3:sol_concept} (for related ones).

The first two concrete applications of Theorem~\ref{th:strong_grad_conv} consist of global existence results for
the Keller--Segel system with superlinear dampening
(\eqref{prob:chemotaxis} with $S(u, v) = u$, $g(u, v) = \tilde g(u)$ with $\lim_{s \ra \infty} \frac{\tilde g(s)}{s} = - \infty$ and $h(u, v) = -v + u$)
and the Keller--Segel system with logarithmic sensitivity 
(\eqref{prob:chemotaxis} with $S(u, v) = \frac{\chi u}{v}$, $\chi > 0$, $g \equiv 0$ and $h(u, v) = -v + u$),
which we discuss in Section~\ref{sec:a1} and Section~\ref{sec:a2}, respectively.
While for both these systems global generalized solutions have already been constructed
(cf.\ \cite{WinklerSolutionsParabolicKellertoappear} and \cite{LankeitWinklerGeneralizedSolutionConcept2017}, respectively),
there are two main reasons why we still choose to consider these systems here.
First, they serve as excellent examples for showing how Theorem~\ref{th:strong_grad_conv} can drastically simplify (and shorten) global existence proofs.
Second, we propose different solutions concepts which are both slightly stronger and in some sense more natural
than the ones considered in \cite{WinklerSolutionsParabolicKellertoappear} and \cite{LankeitWinklerGeneralizedSolutionConcept2017};
we discuss these differences in detail in Subsection~\ref{sec:a1:sol_concept} and Subsection~\ref{sec:a2:sol_concept}.

As a third and final example, we consider a chemotaxis system with logarithmic sensitivity, signal consumption and superlinear dampening
(\eqref{prob:chemotaxis} with $S(u, v) = \frac{\chi u}{v}$, $\chi > 0$, $g(u, v) = \tilde g(u)$ with $\lim_{s \ra \infty} \frac{\tilde g(s)}{s} = - \infty$ and $h(u, v) = - uv$)
in Section~\ref{sec:a3}.
This application distinguishes itself from the previous two in multiple ways.
In particular, Theorem~\ref{th:a3} appears to be the first result concerning global solvability of this system
and we also make crucial use not only of Theorem~\ref{th:strong_grad_conv} but also of Theorem~\ref{th:stronger_grad_conv}.

\section{A priori estimates for the heat equation with \tops{$L^1$}{L1} data}\label{sec:apriori}
Let us begin our journey by deriving various a priori estimates for the Neumann heat equation with an integrable source term.
None of these are particularly difficult to obtain (and may perhaps already be found in the literature),
but the corresponding proofs are also quite short, so that we choose to include them for the sake of completeness.

Throughout this section, we fix a smooth bounded domain $\Omega \subset \R^n$, $n \in \N$.
We are interested in estimates for classical solutions $z \in C^0(\Ombarinf) \times C^{2, 1}(\Ombar \times (0, \infty))$ of
\begin{align}\label{prob:z}
  \begin{cases}
    z_t = \Delta z + f(x, t) & \text{in $\Omega \times (0, \infty)$}, \\
    \partial_\nu z = 0       & \text{on $\partial \Omega \times (0, \infty)$}, \\
    z(\cdot, 0) = z_0        & \text{in $\Omega$},
  \end{cases}
\end{align}
where
\begin{align}\label{eq:z0_f_reg}
  z_0 \in \con0
  \quad \text{and} \quad
  f \in C^0(\Ombarinf).
\end{align}
We note that in order to be able to later on derive $\eps$-independent estimates for solutions of \eqref{prob:ve} from the lemmata proven in this section,
we need to keep track how the constants depend on the data.

We start with a uniform-in-time $L^1$ bound.
\begin{lemma}\label{lm:l1_bdd}
  Assume \eqref{eq:z0_f_reg} and that $z$ is a classical solution of \eqref{prob:z}.
  Then
  \begin{align*}
    \|z\|_{C^0([0, T]; \leb1)} = \|z_0\|_{\leb1} + \|f\|_{L^1(\OmT)}
    \qquad \text{for all $T \gt 0$}.
  \end{align*}
\end{lemma}
\begin{proof}
  We denote the Neumann heat semigroup in $\Omega$ by $(\ure^{t \Delta})_{t \ge 0}$
  and the positive and negative part of a function $\psi$ by $\psi^+ \defs \max\{\psi, 0\}$ and $\psi^- \defs \max\{-\psi, 0\}$, respectively.
  By making use of the variation-of-constants formula, splitting both $z_0$ and $f$ in their respective positive and negative parts
  and noting that $\ure^{t \Delta} \varphi \ge 0$ as well as $\intom \ure^{t \Delta} \varphi = \intom \varphi$ hold for $0 \le \varphi \in \con0$ and $t \ge 0$,
  we then obtain
  \begin{align*}
          \intom |z(\cdot, t)|
    &=    \intom \left| \ure^{t \Delta} z_0 + \intnst \ure^{(t-s) \Delta} f(\cdot, s) \ds \right| \\
    &=    \intom \ure^{t \Delta} z_0^+ + \intnst \intom \ure^{(t-s) \Delta} f^+(\cdot, s) \ds
          + \intom \ure^{t \Delta} z_0^- + \intnst \intom \ure^{(t-s) \Delta} f^-(\cdot, s) \ds \\
    &=    \intom z_0^+ + \intnstom f^+
          + \intom z_0^- + \intnstom f^-
    =     \intom |z_0| + \intnstom |f|
  \end{align*}
  for all $t \ge 0$.
  Since $z \in C^0(\Ombar \times [0, T]) \subset C^0([0, T]; \leb1)$ for all $T > 0$,
  this already entails the statement.
\end{proof}

Testing the first equation in \eqref{prob:z} with increasing bounded functions of $z$ and making use of the dissipative effects of the Laplacian allows us to obtain the following weighted space-time gradient estimates.
\begin{lemma}\label{lm:nabla_tk_z}
  Assume \eqref{eq:z0_f_reg}, suppose that $z$ is a classical solution of \eqref{prob:z}
  and let $T > 0$.
  Then
  \begin{align*}
    \intntom \mathds 1_{\{|z| \le k\}} |\nabla z|^2 \le 2k \left(\|z_0\|_{\leb1} + \|f\|_{L^1(\OmT)}\right).
  \end{align*}
\end{lemma}
\begin{proof}
  We set $T_k s \defs \max\{\min\{s, k\},-k\}$ for $s \in \R$.
  As
  \begin{align*}
        S_k s
    \defs \int_0^s T_k \sigma \dsigma
    =   \frac12 s^2 \mathds 1_{\{|s| \le k\}} + \left(\frac12 k^2 + k (|s|-k)\right) \mathds 1_{\{|s| \gt k\}}
    \le k |s| \mathds 1_{\{|s| \le k\}} + k |s| \mathds 1_{\{|s| \gt k\}}
    =   k |s|
  \end{align*}
  for all $s \in \R$, testing the first equation in \eqref{prob:z} with $T_k z$ and applying Lemma~\ref{lm:l1_bdd} yields
  \begin{align*}
          \intntom \mathds 1_{\{|z| \le k\}} |\nabla z|^2
    &=    - \intom S_k(z(\cdot, T))
          + \intom S_k(z_0)
          + \intntom f T_k(z) \\
    &\le  k \intom |z(\cdot, T)|
          + k \intom |z_0|
          + k \intntom |f| \\
    &\le  2k \left(\|z_0\|_{\leb1} + \|f\|_{L^1(\OmT)}\right).
    \qedhere
  \end{align*}
\end{proof}

\begin{lemma}\label{lm:alpha_nabla_z}
  Let $\alpha > 0$. Then there is $\newgc{alpha_nabla_z} > 0$ such that
  for all $z_0$, $f$ satisfying \eqref{eq:z0_f_reg} and all classical solutions $z$ of \eqref{prob:z},
  the estimate
  \begin{align}\label{eq:alpha_nabla_z:statement}
    \intntom \frac{|\nabla z|^2}{(|z|+1)^{1+\alpha}} \le \gc{alpha_nabla_z} \left(\|z_0\|_{\leb1} + \|f\|_{L^1(\OmT)}\right)
    \qquad \text{holds for all $T \gt 0$}.
  \end{align}
\end{lemma}
\begin{proof}
  Fixing such $z_0$ and $f$ as well as a solution $z$ of \eqref{prob:z} and abbreviating $M \defs \|z_0\|_{\leb1} + \|f\|_{L^1(\OmT)}$,
  we apply Lemma~\ref{lm:nabla_tk_z} to obtain
  \begin{align*}
          \intntom \frac{|\nabla z|^2}{(|z|+1)^{1+\alpha}}
    &=    \sum_{j=0}^\infty \intntom \mathds 1_{\{2^j-1 \le |z| \lt 2^{j+1}-1\}} \frac{|\nabla z|^2}{(|z|+1)^{1+\alpha}} \\
    &\le  \sum_{j=0}^\infty 2^{-j (1+\alpha)} \intntom \mathds 1_{\{|z| \le 2^{j+1}\}} |\nabla z|^2
     \le  2 M \sum_{j=0}^\infty 2^{-j (1+\alpha)} 2^{j+1}
     =    4 M \sum_{j=0}^\infty (2^{-\alpha})^j.
  \end{align*}
  Thus, \eqref{eq:alpha_nabla_z:statement} holds for $\gc{alpha_nabla_z} \defs 4 \sum_{j=0}^\infty (2^{-\alpha})^j$,
  which is finite because of $\alpha > 0$.
\end{proof}

Next, the Gagliardo--Nirenberg allows us to combine Lemma~\ref{lm:l1_bdd} and Lemma~\ref{lm:alpha_nabla_z} to obtain certain space-time bounds also for the function $z$ itself.
\begin{lemma}\label{lm:z_lq}
  Let $T > 0$ and $q \in [1, \frac{n+2}{n})$.
  Then there is $\newgc{z_lq} > 0$ such that
  \begin{align}\label{eq:z_lq:statement}
    \||z|+1\|_{L^q(\OmT)} \le \gc{z_lq} \left(\|z_0\|_{\leb1} + \|f\|_{L^1(\OmT)} + 1\right)^{1+\frac1q}
  \end{align}
  for all classical solutions $z$ of \eqref{prob:z} and $z_0$, $f$ satisfying \eqref{eq:z0_f_reg}.
\end{lemma}
\begin{proof}
  By Hölder's inequality, we may without loss of generality assume $q > \max\{1, \frac2n\}$.
  Then $\alpha \defs 1 + \frac{2}{n} - q \in (0, 1)$
  and since
  \begin{align*}
      \frac{\frac{1-\alpha}{2} - \frac{1-\alpha}{2q}}{\frac{1-\alpha}{2} + \frac{1}{n} - \frac{1}{2}}
    = \frac{(1-\alpha) \frac{q-1}{q}}{1-\alpha + \frac{2}{n} - 1}
    = \frac{1-\alpha}{q}
    = \frac{q - \frac{2}{n}}{q}
    \in (0, 1),
  \end{align*}
  the Gagliardo--Nirenberg inequality asserts that there is $\newlc{gni} > 0$ such that
  \begin{align*}
        \intom |\varphi|^\frac{2q}{1-\alpha}
    \le \lc{gni} \left( \intom |\nabla \varphi|^2 \right) \left( \intom |\varphi|^\frac{2}{1-\alpha} \right)^{\frac2n}
        + \lc{gni} \left( \intom |\varphi|^\frac{2}{1-\alpha} \right)^q
    \qquad \text{for all $\varphi \in \sob12$}.
  \end{align*}
  Fixing $z_0$ and $f$ as in \eqref{eq:z0_f_reg} as well as a solution $z$ of \eqref{prob:ve}
  and abbreviating $M \defs \|z_0\|_{\leb1} + \|f\|_{L^1(\OmT)}$,
  we make use of Lemma~\ref{lm:l1_bdd} and Lemma~\ref{lm:alpha_nabla_z} to estimate
  \begin{align*}
          \intntom (|z|+1)^q
    &=    \intntom [(|z|+1)^\frac{1-\alpha}{2}]^{\frac{2q}{1-\alpha}} \\
    &\le  \lc{gni} \intntom |\nabla (|z|+1)^\frac{1-\alpha}{2}|^2 \left( \intom (|z|+1) \right)^\frac2n
          + \lc{gni} \int_0^T \left( \intom (|z|+1) \right)^q \\
    &\le  \frac{\lc{gni} (1-\alpha)^2 (M+|\Omega|)^\frac2n}{4} \intntom \frac{|\nabla z|^2}{(|z|+1)^{1+\alpha}}
          + \lc{gni} T (M+|\Omega|)^q \\
    &\le  \frac{\lc{gni} (1-\alpha)^2 (M+|\Omega|+1)^{q+1} \gc{alpha_nabla_z}}{4}
          + \lc{gni} T (M+|\Omega|)^q,
  \end{align*}
  where $\gc{alpha_nabla_z}$ is given by Lemma~\ref{lm:alpha_nabla_z}.
  Taking both the left- and right-hand side herein to the power $\frac1q$ implies \eqref{eq:z_lq:statement} for an appropriately chosen $\gc{z_lq} > 0$.
\end{proof}

By means of another interpolation, we also obtain non-weighted gradient estimates.
\begin{lemma}\label{lm:nabla_z_llambda}
  Let $T > 0$ and $\lambda \in [1, \frac{n+2}{n+1})$.
  Then there is $\newgc{nabla_z_llambda} > 0$ such that
  \begin{align}\label{eq:nabla_z_llambda:statement}
    \|\nabla z\|_{L^\lambda(\OmT)} \le \gc{nabla_z_llambda} \left(\|z_0\|_{\leb1} + \|f\|_{L^1(\OmT)} + 1\right)^2
  \end{align}
  for all classical solutions $z$ of \eqref{prob:z} and $z_0$, $f$ satisfying \eqref{eq:z0_f_reg}.
\end{lemma}
\begin{proof}
  Since $\lambda \in [1, \frac{n+2}{n+1})$ and thus $\frac{\lambda}{2-\lambda} < \frac{n+2}{n}$, we may fix $r \in (\frac 12, 1]$ with $q \defs \frac{2r \lambda}{2-\lambda} \in [1, \frac{n+2}{n})$.
  As then also $\alpha \defs 2r-1 > 0$, Lemma~\ref{lm:alpha_nabla_z} and Lemma~\ref{lm:z_lq}
  assert that there are $\gc{alpha_nabla_z}, \gc{z_lq} > 0$ such that \eqref{eq:alpha_nabla_z:statement} and \eqref{eq:z_lq:statement} hold.
  Therefore, again fixing suitable $z_0$, $f$, and $z$ and abbreviating $M \defs \|z_0\|_{\leb1} + \|f\|_{L^1(\OmT)}$,
  we conclude with Hölder's inequality that
  \begin{align*}
          \intntom |\nabla z|^\lambda
    &=    \intntom \left( \frac{|\nabla z|}{(|z|+1)^r}\right)^\lambda (|z|+1|)^{r \lambda} \\
    &\le  \left(\intntom \frac{|\nabla z|^2}{(|z|+1)^{2r}} \right)^\frac{\lambda}{2} \left( \intntom (|z|+1)^\frac{2r \lambda}{2-\lambda} \right)^\frac{2-\lambda}{2} \\
    &\le  (\gc{alpha_nabla_z} M)^\frac{\lambda}{2} (\gc{z_lq} (M+1))^{(1+\frac1q)r \lambda}
     \le  \gc{alpha_nabla_z}^\frac{\lambda}{2} \gc{z_lq}^{(1+\frac1q)r \lambda} (M+1)^{[\frac12 + (1+\frac1q)r] \lambda}.
  \end{align*}
  Due to $(1+\frac1q)r = (1 + \frac{2-\lambda}{2r\lambda})r = \frac{(2r-1)\lambda+2}{2\lambda} \le \frac{2r+1}{2} \le \frac32$,
  this implies \eqref{eq:nabla_z_llambda:statement} for $\gc{nabla_z_llambda} \defs \gc{alpha_nabla_z}^\frac12 \gc{z_lq}^{(1+\frac1q)r}$.
\end{proof}

Finally, we derive a bound for the time derivative.
\begin{lemma}\label{lm:z_t}
  Let $T > 0$.
  Then there is $\newgc{z_t} > 0$ such that
  \begin{align*}
    \|z_t\|_{L^1((0, T); \dual{\sob{n+1}{2}})} \le \gc{z_t} \left(\|z_0\|_{\leb1} + \|f\|_{L^1(\OmT)} + 1\right)^2
  \end{align*}
  for all classical solutions $z$ of \eqref{prob:z} and $z_0$, $f$ satisfying \eqref{eq:z0_f_reg}.
\end{lemma}
\begin{proof}
  We fix $z_0$ and $f$ as in \eqref{eq:z0_f_reg} and a solution $z$ of \eqref{prob:z}.
  Then
  \begin{align*}
          \left| \intom z_t(\cdot, t) \varphi \right|
    &\le  \intom |\nabla z(\cdot, t) \cdot \nabla \varphi| + \intom |f(\cdot, t) \varphi|
     \le  \|\varphi\|_{\sob1\infty} \left( \|\nabla z(\cdot, t)\|_{\leb1} + \|f(\cdot, t)\|_{\leb1} \right)
  \end{align*}
  holds for all $\varphi \in C^\infty(\Ombar)$ and $t > 0$.
  Since $\sob{n+1}{2} \embed \sob1\infty$,
  the statement follows from Lemma~\ref{lm:nabla_z_llambda} (for $\lambda=1$) upon an integration in time.
\end{proof}

\section{Convergence properties for \tops{$L^1$}{L1} data: proof of Theorem~\ref{th:strong_grad_conv}}\label{sec:conv_l1}
The present section is devoted to the proof of Theorem~\ref{th:strong_grad_conv}.
After making use of the a priori estimates derived in Section~\ref{sec:apriori} to obtain a global weak solution of \eqref{prob:v_limit} as a limit of approximate solutions in Subsection~\ref{sec:conv_l1:weak_conv},
the most challenging part of the proof is then carried out in Subsection~\ref{sec:conv_l1:strong_grad_conv}.
There, we prove the desired strong convergence of weighted gradients by means of intricate testing procedures.
Finally, the proof of Theorem~\ref{th:strong_grad_conv} is given in Subsection~\ref{sec:conv_l1:proof}.

Throughout this section, we again fix a smooth, bounded domain $\Omega \subset \R^n$, $n \in \N$.
Moreover, if $v$ is a classical solution of \eqref{prob:v_limit} for some $v_0 \in \con0$, $f_0 \in C^0(\Ombarinf)$ and $\kappa \in \R$,
then $(x, t) \mapsto \ure^{\kappa t} v(x, t)$ solves \eqref{prob:v_limit} classically for $\kappa = 0$ (and the same $v_0$, $f$).
Thus, as all statements in Theorem~\ref{th:strong_grad_conv} are local in time, we may (and will) henceforth always assume $\kappa = 0$.

\subsection{Convergence to a weak solution}\label{sec:conv_l1:weak_conv}
First, we briefly note that if instead of \eqref{eq:strong_grad_conv:f_conv} we assumed strong convergence of $(f_{\eps_j'})_{j \in \N}$ in $L_{\loc}^1(\Ombarinf)$,
then the assumptions in Theorem~\ref{th:strong_grad_conv} and the conclusion of Lemma~\ref{lm:l1_bdd} rapidly imply convergence of the solutions in $C^0([0, \infty); \leb1)$.
\begin{lemma}\label{lm:limit_ve_c0_l1}
  Suppose that the hypotheses of Theorem~\ref{th:strong_grad_conv} hold and that additionally
  \begin{align}\label{eq:limit_ve_c0_l1:f_l1_strong}
    \fe \ra f \qquad \text{in $L_{\loc}^1(\Ombarinf)$ as $\eps = \eps_j' \sea 0$}. 
  \end{align}
  Then there is $v \in C^0([0, \infty); \leb1)$ such that
  \begin{align}\label{eq:limit_ve_c0_l1:statement}
    \ve &\ra v \qquad \text{in $C^0([0, \infty); \leb1)$ as $\eps = \eps_j' \sea 0$}.
  \end{align}
\end{lemma}
\begin{proof}
  According to Lemma~\ref{lm:l1_bdd}, \eqref{eq:strong_grad_conv:v0_conv} and \eqref{eq:limit_ve_c0_l1:f_l1_strong},
  \begin{align*}
        \lim_{j \ra \infty} \sup_{k \ge j} \|v_{\eps_k'}- v_{\eps_j'}\|_{C^0([0, T]; \leb1)}
    \le \lim_{j \ra \infty} \sup_{k \ge j} \left( \|v_{0 \eps_k'}- v_{0 \eps_j'}\|_{\leb1} + \|f_{\eps_k'}- f_{\eps_j'}\|_{L^1(\OmT)} \right)
    =   0
  \end{align*}
  for all $T \in (0, \infty)$.
  Therefore, there is $v^{(T)} \in C^0([0, T); \leb1)$ such that $v_{\eps_j'} \ra v^{(T)}$ in $C^0([0, T]; \leb1)$ as $j \ra \infty$ for all $T > 0$.
  Since $v^{(T)} = v^{(T')}$ in $\Ombar \times [0, T]$ for $0 \lt T \lt T' \lt \infty$,
  setting $v(x, t) \defs v^{(T)}(x, t)$ for $x \in \Ombar$ and $t \in [0, T]$ defines a function $v \in C^0([0, \infty); \leb1)$ fulfilling \eqref{eq:limit_ve_c0_l1:statement}.
\end{proof}

With the various a priori estimates derived in the previous section at hand, we are also able to obtain a suitable limit function $v$, and hence a solution candidate for \eqref{prob:v_limit},
without requiring \eqref{eq:limit_ve_c0_l1:f_l1_strong}.
In fact, boundedness of $(v_{\eps_j'})_{j \in \N}$ and $(f_{\eps_j'})_{j \in \N}$ in suitable spaces suffices.

\begin{lemma}\label{lm:limit_ve}
  Let $(\eps'_j)_{j \in \N}$ be a null sequence, let
  \begin{align}\label{eq:limit_ve:vne_l1}
    (v_{\eps_j'})_{j \in \N} \subset \con0
    &\quad \text{with} \quad \sup_{j \in \N} \|v_{\eps_j'}\|_{\leb1} \lt \infty
  \intertext{as well as}\label{eq:limit_ve:fe_l1}
    (f_{\eps_j'})_{j \in \N} \subset C^0(\Ombarinf)
    &\quad \text{with} \quad \sup_{j \in \N} \|f_{\eps_j'}\|_{L^1(\OmT)} \lt \infty \text{ for all $T > 0$}
  \end{align}
  and suppose that $\vej \in C^0(\Ombarinf) \cap C^{2, 1}(\Ombar \times (0, \infty))$ is a classical solution of \eqref{prob:ve} (with $\eps$ replaced by $\eps_j$) for all $j \in \N$.
  Then there are
  \begin{align*}
    v \in L_{\loc}^1([0, \infty); \sob11)
    \quad \text{with} \quad
    \mathds 1_{\{|v| \le k\}} \nabla v \in L_{\loc}^2(\Ombarinf) \text{ for all $k \in \N$}
  \end{align*}
  and a subsequence $(\eps_j)_{j \in \N}$ of $(\eps_j')_{j \in \N}$ such that
  \begin{alignat}{2}
    \ve &\ra v
    &&\qquad \text{in $L_{\loc}^1(\Ombarinf)$ and a.e.\ in $\Omega \times (0, \infty)$}, \label{eq:limit_ve:pw} \\
    \nabla \ve &\rh \nabla v
    &&\qquad \text{in $L_{\loc}^1(\Ombarinf)$}, \label{eq:limit_ve:nabla_ve_l1} \\
    \mathds 1_{\{|\ve| \le k\}} \nabla \ve &\rh \mathds 1_{\{|v| \le k\}} \nabla v
    &&\qquad \text{in $L_{\loc}^2(\Ombarinf)$} \label{eq:limit_ve:truncations_weak}
  \end{alignat}
  as $\eps = \eps_j \sea 0$.
\end{lemma}
\begin{proof}
  Thanks to \eqref{eq:limit_ve:vne_l1} and \eqref{eq:limit_ve:fe_l1}, Lemma~\ref{lm:nabla_z_llambda} and Lemma~\ref{lm:z_t} assert that
  the family $(v_{\eps_j'})_{j \in \N}$ is bounded in the space $\{\,\varphi \in L^1((0, T); \sob11) : \varphi_t \in L^1((0, T); \dual{\sob{n+1}{2}})\,\}$ for all $T > 0$,
  so that by making use of the Aubin--Lions lemma and a diagonalization argument,
  we obtain a subsequence $(\eps_j)_{j \in \N}$ of $(\eps_j')_{j \in \N}$ and $v \in L_{\loc}^1([0, \infty); \sob11)$
  such that $\ve \ra v$ in $L_{\loc}^1(\Ombarinf)$ as $\eps = \eps_j \sea 0$.
  Upon switching to a subsequences, this entails \eqref{eq:limit_ve:pw} and (again making use of Lemma~\ref{lm:nabla_z_llambda}) also \eqref{eq:limit_ve:nabla_ve_l1}.
   
  Moreover, for each $k \in \N$, Lemma~\ref{lm:nabla_tk_z} and another diagonalization argument
  assert the existence of a function $z_k \in L_{\loc}^2([0, \infty); \sob12)$ and a subsequence of $(\eps_j)_{j \in \N}$, which we do not relabel,
  such that $T_k \vej \rh z_k$ in $L_{\loc}^2([0, \infty); \sob12) \embed L_{\loc}^2(\Ombarinf)$ as $j \ra \infty$,
  where $T_k s \defs \max\{\min\{s, k\}, -k\}$ for $s \in \R$.
  As $T_k$ is continuous, the pointwise convergence statement in \eqref{eq:limit_ve:pw} implies $z_k = T_k v$ for all $k \in \N$;
  that is, \eqref{eq:limit_ve:truncations_weak} holds.
\end{proof}

Next, we note that the convergence properties asserted by Lemma~\ref{lm:limit_ve} are sufficiently strong to conclude that, if the hypotheses of Theorem~\ref{th:strong_grad_conv} hold,
the function $v$ constructed in Lemma~\ref{lm:limit_ve} is indeed a global weak solution of \eqref{prob:v_limit}.
\begin{lemma}\label{lm:v_weak_sol}
  Suppose the hypotheses of Theorem~\ref{th:strong_grad_conv} hold
  and let $v$ and $(\eps_j)_{j \in \N}$ be as given by Lemma~\ref{lm:limit_ve}.
  Then $v$ is a weak solution of \eqref{prob:v_limit}; that is, \eqref{eq:strong_grad_conv:weak_sol} holds for all $\varphi \in C_c^\infty(\Ombarinf)$.
\end{lemma}
\begin{proof}
  For $j \in \N$, $\vej$ solves \eqref{prob:ve} classically,
  so that testing the first equation in \eqref{prob:ve} with $\varphi \in C_c^\infty(\Ombarinf)$ gives
  \begin{align*}
      - \intninfom \vej \varphi_t
      - \intom \vnej \varphi(\cdot, 0)
    = - \intninfom \nabla \vej \cdot \nabla \varphi
      + \intninfom \fej \varphi
    \qquad \text{for all $j \in \N$}.
  \end{align*}
  Thanks to \eqref{eq:limit_ve:pw}, \eqref{eq:strong_grad_conv:v0_conv}, \eqref{eq:limit_ve:nabla_ve_l1} and \eqref{eq:strong_grad_conv:f_conv},
  the statement follows upon taking the limit $j \ra \infty$.
\end{proof}

Let us next remark that since \eqref{prob:v_limit} is a linear parabolic equation,
a duality argument shows that the solution of \eqref{prob:v_limit} constructed above is unique, even among very weak solutions.
\begin{lemma}\label{lm:unique_very_weak}
  Let $v_0 \in \leb1$ and $f \in L_{\loc}^1(\Ombarinf)$.
  Then \eqref{prob:v_limit} possesses at most one very weak solution.
  That is, there is at most one
  \begin{align*}
    v \in L_{\loc}^1(\Ombarinf)
  \end{align*}
  such that
  \begin{align}\label{eq:unique_very_weak:sol_def}
      - \intninfom v \varphi_t
      - \intom v_0 \varphi(\cdot, 0)
    = \intninfom v \Delta \varphi
      + \intninfom f \varphi
  \end{align}
  for all $\varphi \in C_c^\infty(\Ombarinf)$ with $\partial_\nu \varphi = 0$ on $\partial \Omega \times (0, \infty)$.
\end{lemma}
\begin{proof}
  Let $v_1, v_2$ be two such solutions.
  A straightforward approximation argument allows us to infer from \eqref{eq:unique_very_weak:sol_def} that for a.e.\ $T > 0$
  and all $\varphi \in C^{2, 1}(\Ombar \times [0, T])$ with $\varphi(\cdot, T) = 0$ in $\Omega$ and $\partial_\nu \varphi = 0$ on $\partial \Omega \times (0, \infty)$,
  \begin{align*}
        0
    &=  \intom (v_0-v_0) \varphi(\cdot, 0) + \intntom (f-f) \varphi
     =  - \intntom (v_1 - v_2) \varphi_t
        - \intntom (v_1 - v_2) \Delta \varphi.
  \end{align*}
  Given such $T$ and any $g \in C^\infty(\Ombar \times [0, T])$,
  classic parabolic theory (cf.\ \cite[Theorem~IV~5.3]{LadyzenskajaEtAlLinearQuasilinearEquations1988}) asserts the existence of a function $\varphi \in C^{2, 1}(\Ombar \times [0, T])$ solving
  \begin{align*}
    \begin{cases}
      -\varphi_t = \Delta \varphi + g & \text{in $\Omega \times (0, T)$}, \\
      \partial_\nu \varphi = 0        & \text{on $\partial \Omega \times (0, T)$}, \\
      \varphi(\cdot, T) = 0           & \text{in $\Omega$}
    \end{cases}
  \end{align*}
  classically,
  hence $\intntom (v_1 - v_2) g = 0$ for these $T$ and $g$.
  By the fundamental lemma of the calculus of variations, we conclude $v_1 - v_2 = 0$ a.e.\ in $\Omega \times (0, T)$ for a.e.\ $T > 0$ and thus $v_1 = v_2$ in $L_{\loc}^1(\Ombarinf)$.
\end{proof}

\begin{lemma}\label{lm:unique_weak}
  Let $v_0 \in \leb1$ and $f \in L_{\loc}^1(\Ombarinf)$.
  Then \eqref{prob:v_limit} possesses at most one weak solution, that is, a function $v \in L_{\loc}^1([0, \infty); \sob11)$ fulfilling \eqref{eq:strong_grad_conv:weak_sol} for all $\varphi \in C_c^\infty(\Ombarinf)$.
\end{lemma}
\begin{proof}
  As a straightforward approximation argument shows that weak solutions are very weak solutions,
  the statement immediately follows from Lemma~\ref{lm:unique_very_weak}.
\end{proof}

\subsection{Strong convergence of (weighted) gradients}\label{sec:conv_l1:strong_grad_conv}
We now come to the most challenging part in the proof of Theorem~\ref{th:strong_grad_conv},
namely the verification that under the assumptions made in Theorem~\ref{th:strong_grad_conv}
(weighted) gradients of the approximate solutions convergence \emph{strongly} to the corresponding expressions of the limit function constructed in Lemma~\ref{lm:limit_ve}.

As we will see in Lemma~\ref{lm:v_llambda_strong_conv} and Lemma~\ref{lm:v_r_nabla_v_strong_conv} below, \eqref{eq:strong_grad_conv:v_pointwise_grad_conv} and \eqref{eq:strong_grad_conv:v_weighted_grad_conv}
follow rather straightforwardly from \eqref{eq:strong_grad_conv:v_grad_conv} and the a priori estimates collected in Section~\ref{sec:apriori}.
Thus, our main focus lies on \eqref{eq:strong_grad_conv:v_grad_conv},
that is, (essentially) the strong convergence of the truncations $T_k \vej$ to $T_k v$ in $L_{\loc}^2([0, \infty); \sob12)$ as $j \ra \infty$,
where again $T_k s \defs \max\{\min\{s, k\}, -k\}$ for $s \in \R$.

As already briefly mentioned in the introduction, similar statements have been derived in related contexts,
some of which we would like to briefly review here.
They all differ from the present paper in that they do not motivate the interest in the strong convergence of truncations by applications to global solvability of certain parabolic systems
(as we do in the present paper, see Subsection~\ref{sec:intro:part2} in the introduction and Sections~\ref{sec:a1}--\ref{sec:a3} below).
Instead, they consider settings where weak solutions either are not known to exist (that is, an analogue of Lemma~\ref{lm:v_weak_sol} cannot be proved)
or are no longer unique (which they are in our setting, as we have seen in Lemma~\ref{lm:unique_weak}),
so that one strives for stronger solution concepts, for which strong convergence of truncations form a crucial ingredient in the corresponding global existence proofs.
Early findings in this direction deal with rather general Leray--Lions operators and integrable right-hand sides \cite{BoccardoMuratStronglyNonlinearCauchy1989, BoccardoMuratAlmostEverywhereConvergence1992}.
These results then have been extended to not necessarily strictly monotone operators \cite{BlanchardTruncationsMonotonicityMethods1993},
to right-hand sides involving gradient terms of natural growth \cite{DallaglioOrsinaNonlinearParabolicEquations1996, PorrettaExistenceResultsNonlinear1999}
and to measure initial data \cite{BlanchardPorrettaNonlinearParabolicEquations2001}.
In all these works, the proof relies on an intricate testing procedure, carried out at the level of approximative problems.

An alternative and fundamentally different approach for verifying the desired strong convergence properties
consists in directly testing the weak solution $v$ to the limit problem \eqref{prob:v_limit} with suitably chosen functions.
Indeed, if certain regularity properties of $v$ were known,
we could choose $\varphi = T_k v$ in \eqref{eq:strong_grad_conv:weak_sol}
and argue as in Lemma~\ref{lm:nabla_tk_z} to obtain an identity (and not just an upper estimate) for $\intntom |\nabla T_k v|^2$.
This would then eventually imply convergence of the norms $\|\nabla T_k \vej\|_{L^2(\OmT)}$ to $\|\nabla T_k v\|_{L^2(\OmT)}$ as $j \ra \infty$
and hence also the desired strong convergence of the truncations $(T_k \vej)_{j \in \N}$ in $L^2((0, T); \sob12)$.
Unfortunately, however, $\varphi$ is not an admissible test function.
By using a sequence of approximate and sufficiently regular test functions instead
and verifying that the obtained identities survive the corresponding limit processes,
one can hope to deal with this problem.
For instance, approximations based on so-called Steklov averages may be used to overcome the low time regularity of $v$
and these have been successfully used inter alia in
\cite[Lemma~8.2]{WinklerLargedataGlobalGeneralized2015} and \cite[Lemma~2.9]{WangGlobalLargedataGeneralized2016}.
However, in contrast to the weak solutions considered in both these works,
the space regularity of the weak solution of \eqref{prob:v_limit} given by Theorem~\ref{th:strong_grad_conv} is also rather low,
so that the techniques employed in \cite{WinklerLargedataGlobalGeneralized2015} and \cite{WangGlobalLargedataGeneralized2016}
can not be directly adopted to our situation.

In the present paper, we choose to follow the reasoning in \cite{PorrettaExistenceResultsNonlinear1999} instead,
where the methods in \cite{LeonePorrettaEntropySolutionsNonlinear1998} have been adapted from an elliptic to a parabolic setting.
That is, the proofs of Lemma~\ref{lm:vet_varphi_lower_bdd} and Lemma~\ref{lm:tk_v_strong_conv} below roughly resemble \cite[pages 149--152]{PorrettaExistenceResultsNonlinear1999}.
The main reason the results from \cite{PorrettaExistenceResultsNonlinear1999} are not directly applicable to our situation is that different boundary conditions are employed;
\eqref{prob:v_limit} features Neumann and the systems in \cite{PorrettaExistenceResultsNonlinear1999} Dirichlet boundary conditions.
As it turns out, this difference barely affects the proof.
Moreover, instead of a wide class of elliptic operators, we only consider the Laplacian, which allows us to streamline the proof slightly.

As a preparation, we collect some properties of an approximation introduced by Landes in \cite{LandesExistenceWeakSolutions1981} to deal with weak time regularity.
While very similar statements can already be found in \cite{LandesExistenceWeakSolutions1981},
Neumann boundary conditions are not treated there.
Therefore, we choose to include a short proof here for completeness.
\begin{lemma}\label{lm:landes_approx}
  Let $T \gt 0$, $k \in \N$, $v_0 \in \leb1$, $v \in L^1(\OmT)$ with $T_k v = \max\{\min\{v, k\},-k\} \in L^2((0, T); \sob12)$ and
  \begin{align}\label{eq:landes_approx:def_zeta}
    (\zeta_\ell)_{\ell \in \N} \subset C^\infty(\Ombar)
    \quad \text{with} \quad
    \zeta_\ell \ra T_k v_0 \text{ in $\leb1$ and a.e.\ in $\Omega \times (0, T)$ as $\ell \ra \infty$}.
  \end{align}
  Then, for $\sigma, \ell \in \N$, the function
  \begin{align}\label{eq:landes_approx:def_eta}
          \eta_{\sigma, \ell}(v) \colon \Omega \times (0, T) \ra \R, \quad
    (x, t) \mapsto \ure^{-\sigma t} T_k \zeta_\ell(x) + \sigma \int_0^t \ure^{-\sigma(t-s)} T_k v(x, s) \ds
  \end{align}
  belongs to $L^2((0, T); \sob12) \cap L^\infty(\OmT)$ with $(\eta_{\sigma, \ell}(v))_t \in L^2((0, T); \sob12)$.
  Moreover,
  \begin{align}
    \intntom (\eta_{\sigma, \ell}(v))_t \varphi &= \sigma \intntom (T_k v - \eta_{\sigma, \ell}(v)) \varphi, \label{eq:landes_approx:de}\\
    (\eta_{\sigma, \ell}(v))(0) &= T_k \zeta_\ell, \label{eq:landes_approx:init} \\
    \|\eta_{\sigma, \ell}(v)\|_{L^\infty(\OmT)} &\le k \label{eq:landes_approx:linfty_bdd}
  \end{align}
  for all $\varphi \in L^2((0, T); \sob12)$ and $\sigma, \ell \in \N$,
  and
  \begin{align}\label{eq:landes:approx:conv}
    (\eta_{\sigma, \ell}(v)) \ra T_k v \qquad \text{in $L^2((0, T); \sob12)$ as $\sigma \ra \infty$ for all $\ell \in \N$}.
  \end{align}
\end{lemma}
\begin{proof}
  As direct consequences of \eqref{eq:landes_approx:def_eta} and the assumed regularity of $v$ and $\zeta_\ell$,
  we obtain \eqref{eq:landes_approx:init} and that $\eta_{\sigma, \ell}(v)$ belongs to $L^2((0, T); \sob12)$.
  Moreover, an application of the generalized fundamental theorem of calculus (cf.\ \cite[Problem~23.5b]{ZeidlerNonlinearFunctionalAnalysis1990a}) reveals that
  \begin{align*}
        (\eta_{\sigma, \ell}(v))_t(\cdot, t)
    &=  - \sigma \ure^{-\sigma t} T_k \zeta_\ell
        - \sigma^2 \int_0^t \ure^{-\sigma(t-s)} T_k v(\cdot, s) \ds
        + \sigma T_k v(\cdot, t) \\
    &=  \sigma T_k v(\cdot, t)
        - \sigma (\eta_{\sigma, \ell}(v))(\cdot, t)
    \qquad \text{in $\sob12$ for a.e.\ $t \in (0, T)$ and all $\sigma, \ell \in \N$},
  \end{align*}
  which entails $(\eta_{\sigma, \ell}(v))_t \in L^2((0, T); \sob12)$ 
  and \eqref{eq:landes_approx:de} for all $\sigma, \ell \in \N$.
  The estimate \eqref{eq:landes_approx:linfty_bdd} follows from
  \begin{align*}
    \ure^{-\sigma t} k + \sigma \int_0^t \ure^{-\sigma(t-s)} k \ds = k
    \qquad \text{for all $t \in (0, T)$ and $\sigma \in \N$}
  \end{align*}
  and the monotony of the integral.
  In order to verify \eqref{eq:landes:approx:conv}, we set 
  \begin{align*}
    (L_\sigma \varphi)(x, t) \defs \sigma \int_0^t \ure^{-\sigma(t-s)} \varphi(x, s) \ds, \quad (x, t) \in \OmT,
  \end{align*}
  for $\varphi \in L^2(\OmT)$ and $\sigma \in \N$.
  By Hölder's inequality and Fubini's theorem, $L_\sigma$ is a bounded linear operator from $L^2(\OmT)$ to $L^2(\OmT)$ with operator norm at most $1$ for all $\sigma \in \N$.
  Now, for henceforth fixed $\varphi \in L^2(\OmT)$ and arbitrary $\delta > 0$,
  there exists $\tilde \varphi \in C^\infty(\Ombar \times [0, T])$ such that $\|\tilde \varphi - \varphi\|_{L^2(\OmT)} \le \frac{\delta}{4}$.
  As moreover
  \begin{align*}
        |(L_\sigma \tilde \varphi)(x, t) - \tilde \varphi(x, t)|
    \le \sigma \int_{-\infty}^t \ure^{-\sigma(t-s)} \mathds 1_{(0, t)}(s) |\tilde \varphi(x, s) - \tilde \varphi(x, t)| \ds + \ure^{-\sigma t} |\tilde \varphi(x, t)|
    \ra 0
  \end{align*}
  as $\sigma \ra \infty$ for all $(x, t) \in \OmT$ due to continuity of $\tilde \varphi$ at $(x, t)$,
  Lebesgue's theorem asserts that there is $\sigma_0 \in \N$ with $\|L_\sigma \tilde\varphi - \tilde\varphi\|_{L^2(\OmT)} \le \frac{\delta}{2}$ for all $\sigma \ge \sigma_0$.
  Therefore, we can conclude
  \begin{align*}
          \|L_\sigma \varphi - \varphi\|_{L^2(\OmT)}
    &\le  \|L_\sigma \varphi - L_\sigma \tilde \varphi\|_{L^2(\OmT)}
          + \|L_\sigma \tilde \varphi - \tilde\varphi\|_{L^2(\OmT)}
          + \|\tilde \varphi - \varphi\|_{L^2(\OmT)} \\
    &\le  \frac{\delta}{4} + \frac{\delta}{2} + \frac{\delta}{4}
    =     \delta
    \qquad \text{for all $\sigma \ge \sigma_0$}.
  \end{align*}
  That is, $L_\sigma \varphi \ra \varphi$ in $L^2(\OmT)$ as $\sigma \ra \infty$.
  If $\varphi \in L^2((0, T); \sob12)$, a straightforward approximation argument shows $\nabla (L_\sigma \varphi) = L_\sigma \nabla \varphi$ in $L^2(\OmT)$,
  so that then also $\nabla (L_\sigma \varphi) \ra \nabla \varphi$ in $L^2(\OmT)$ as $\sigma \ra \infty$.
  Since additionally
  \begin{align*}
        \intntom | \ure^{-\sigma t} (\zeta_\ell(x) + \nabla \zeta_\ell(x)) |^2 \dx \dt
    \le 2 \left( \intnt \ure^{-2\sigma t} \dt \right) \left(\intom |\zeta_\ell(x)|^2 \dx + \intom |\nabla \zeta_\ell(x)|^2 \dx \right)
    \ra 0
  \end{align*}
  as $\sigma \ra \infty$ for all $\ell \in \N$,
  we obtain \eqref{eq:landes:approx:conv}.
\end{proof}

In order to verify \eqref{eq:strong_grad_conv:v_grad_conv}, we test the solutions to the approximate problems with
\begin{align}\label{eq:stronger_grad_conv:varphi}
  \varphi_{j, \sigma, \ell, h} \defs T_{2k}(\vej - T_h \vej + T_k \vej - \eta_{\sigma, \ell}(v)), 
\end{align}
where $\eta_{\sigma, \ell}(v)$ is as in \eqref{eq:landes_approx:def_eta}, $T_k s = \max\{\min\{s, k\},-k\}$ for $s \in \R$ and $k \in \N$, and $j, \sigma, \ell, h \in \N$.
As a preparation for this testing procedure, we state the following
\begin{lemma}\label{lm:vet_varphi_lower_bdd}
  Suppose the hypotheses of Theorem~\ref{th:strong_grad_conv} hold,
  let $v$ and $(\eps_j)_{j \in \N}$ be as given by Lemma~\ref{lm:limit_ve},
  and let $T > 0$, $k \in \N$,
  $\zeta_\ell$ be as in \eqref{eq:landes_approx:def_zeta},
  $\eta_{\sigma, \ell}(v)$ be as in \eqref{eq:landes_approx:def_eta}
  and $\varphi_{j, \sigma, \ell, h}$ be as in \eqref{eq:stronger_grad_conv:varphi} for $j, \sigma, \ell, h \in \N$.
  Then
  \begin{align}\label{eq:vet_varphi_lower_bdd:statement}
    \liminf_{h \ra \infty} \liminf_{\ell \ra \infty} \liminf_{\sigma \ra \infty} \liminf_{j \ra \infty}
    \intntom \vejt \varphi_{j, \sigma, \ell, h} \ge 0.
  \end{align}
\end{lemma}
\begin{proof}
  By considering each case separately, we see that
  \begin{align*}
        T_{2k}(s - T_h s + T_k s - s_0)
    &=  \begin{cases}
          s - s_0,                   & |s| \le k, \\
          k \sign s - s_0,           & k \lt |s| \le h, \\
          s - (h - k) \sign s - s_0, & h \lt |s|, |s - s_0| \le h + k, \\
          2k \sign (s-s_0),          & h \lt |s|, |s - s_0| \gt h + k \\
        \end{cases} \\ 
    &=  T_{h+k}(s - s_0) - T_{h-k}(s - T_k s)
  \end{align*}
  holds for all $s, s_0 \in \R$ with $|s_0| \le k$ and $h > k$.
  Therefore, setting
  \begin{align*}
    G_{1, h}(s) \defs \int_0^s T_{h+k}(\sigma) \dsigma, \quad
    G_{2, h}(s) \defs \int_0^s T_{h-k}(\sigma- T_k \sigma) \dsigma
    \quad \text{and} \quad
    H_{s_0, h}(s) \defs G_{1, h}(s-s_0) - G_{2, h}(s)
  \end{align*}
  for $s, s_0 \in \R$ with $|s_0| \le k$ and $h > k$,
  and recalling \eqref{eq:landes_approx:def_eta} and \eqref{eq:stronger_grad_conv:varphi},
  we have
  \begin{align}\label{eq:vet_varphi_lower_bdd:first}
    &\pe  \intntom \vejt \varphi_{j, \sigma, \ell, h}
          - \intntom (\eta_{\sigma, \ell}(v))_t T_{h+k}(\vej - \eta_{\sigma, \ell}(v)) \notag \\ 
    &=    \intntom (\vej - \eta_{\sigma, \ell}(v))_t G_{1, h}'(\vej - \eta_{\sigma, \ell}(v))
          - \intntom \vejt G_{2, h}'(\vej) \notag \\
    &=    \intom H_{(\eta_{\sigma, \ell}(v))(\cdot, T), h}(\vej(\cdot, T))
          - \intom H_{\zeta_\ell, h}(\vnej)
    \qquad \text{for all $j, \sigma, \ell, h \in \N$}.
  \end{align}
  Since
  \begin{align*}
    H'_{s_0, h}(s) \sign (s-s_0) \ge 0, \quad 
    H_{s_0, h}(s_0) = 0
    \quad \text{and thus} \quad
    H_{s_0, h}(s) \ge 0
  \end{align*}
  for $s, s_0 \in \R$ with $|s_0| \le k$ and $h > k$,
  the first term on the right-hand side in \eqref{eq:vet_varphi_lower_bdd:first} is nonnegative.
  Regarding the second one,
  we note that since
  \begin{align*}
    |H_{s_0, h}(s)| \le (h+k) |s-s_0| + (h-k) |s|
    \quad \text{and} \quad
    |H_{T_k s, h}(s)| \le 2k |s| \mathds 1_{\{|s| \gt h\}}
  \end{align*}
  for $s, s_0 \in \R$ with $|s_0| \le k$ and $h > k$,
  and because of \eqref{eq:limit_ve:pw} and \eqref{eq:landes_approx:def_zeta},
  we may apply Lebesgue's theorem thrice to obtain
  \begin{align*}
        - \lim_{h \ra \infty} \lim_{\ell \ra \infty} \lim_{j \ra \infty}  \intom H_{\zeta_\ell, h}(\vnej)
    =   - \lim_{h \ra \infty} \intom H_{T_k v_0, h}(v_0)
    \ge - 2k \lim_{h \ra \infty} \int_{\{|v_0| \gt h\}} |v_0|
    =   0.
  \end{align*}
  When inserted into \eqref{eq:vet_varphi_lower_bdd:first}, this implies
  \begin{align}\label{eq:vet_varphi_lower_bdd:second}
    &\pe  \liminf_{h \ra \infty} \liminf_{\ell \ra \infty} \liminf_{j \ra \infty} 
          \intntom \vejt \varphi_{j, \sigma, \ell, h} \notag \\
    &\ge  \liminf_{h \ra \infty} \liminf_{\ell \ra \infty} \liminf_{j \ra \infty}
          \intntom (\eta_{\sigma, \ell}(v))_t T_{h+k}(\vej - \eta_{\sigma, \ell}(v))
    \qquad \text{for all $\sigma \in \N$}.
  \end{align}
  By \eqref{eq:limit_ve:pw} and \eqref{eq:landes_approx:de}, we have
  \begin{align*}
    &\pe  \lim_{j \ra \infty} \intntom (\eta_{\sigma, \ell}(v))_t T_{h+k}(\vej - \eta_{\sigma, \ell}(v)) \\
    &=    \sigma \intntom (T_k v - \eta_{\sigma, \ell}(v)) T_{h+k}(v-\eta_{\sigma, \ell}(v)) \\  
    &=    \sigma \int_{\{|v| \le k\}}  (T_k v - \eta_{\sigma, \ell}(v)) T_{h+k}(T_k v - \eta_{\sigma, \ell}(v))
          + \sigma \int_{\{|v| \gt k\}} (k \sign v - \eta_{\sigma, \ell}(v)) T_{h+k}(v-\eta_{\sigma, \ell}(v)) \\
    &\ge  \sigma \int_{\{|v| \le k\}}  (T_k v - \eta_{\sigma, \ell}(v))^2
          + \sigma \int_{\{|v| \gt k\}} (k \sign v - \eta_{\sigma, \ell}(v)) T_{h+k}(0)
    \ge   0
    \qquad \text{for all $\sigma, \ell \in \N$ and $h \ge k$},
  \end{align*}
  which, when combined with \eqref{eq:vet_varphi_lower_bdd:second}, allows us to conclude \eqref{eq:vet_varphi_lower_bdd:statement}.
\end{proof}

\begin{lemma}\label{lm:tk_v_strong_conv}
  Under the same assumptions as in Lemma~\ref{lm:vet_varphi_lower_bdd}, it holds that
  \begin{align}\label{eq:tk_v_strong_conv:statement}
    \nabla T_k \vej \ra \nabla T_k v
    \qquad \text{in $L^2(\OmT)$ as $j \ra \infty$}.
  \end{align}
\end{lemma}
\begin{proof}
  As already alluded to, the first step in this proof is to test the first equation in \eqref{prob:ve} with $\varphi_{j, \sigma, l, h} \in L^2((0, T); \sob12)$
  (defined in \eqref{eq:stronger_grad_conv:varphi}),
  which yields
  \begin{align}\label{eq:tk_v_strong_conv:test}
      \intntom \vejt \varphi_{j, \sigma, l, h}
    = - \intntom \nabla \ve \cdot \nabla \varphi_{j, \sigma, l, h}
      + \intntom \fej \varphi_{j, \sigma, l, h}
    \qquad \text{for all $j, \sigma, \ell, h \in \N$}.
  \end{align}
  We next show that the second term on the right-hand side vanishes as (in this order) $j \ra \infty$, $\sigma \ra \infty$ and $h \ra \infty$.
  Condition \eqref{eq:strong_grad_conv:f_conv} asserts $\fej \rh f$ in $L^1(\OmT)$ for $j \ra \infty$,
  \eqref{eq:limit_ve:pw} implies $\varphi_{j, \sigma, \ell, h} \ra T_{2k}(v - T_h v + T_k v - \eta_{\sigma, \ell}(v))$ a.e.\ in $\OmT$ for $j \ra \infty$
  and \eqref{eq:stronger_grad_conv:varphi} entails boundedness of $(\varphi_{j, \sigma, \ell, h})_{j \in \N}$ for all $\sigma, \ell, h \in \N$.
  Therefore, by combining the Egorov, Lebesgue and Dunford--Pettis theorems we obtain
  \begin{align*}
        \lim_{j \ra \infty} \intntom \fej \varphi_{j, \sigma, \ell, h}
    &=  \intntom f T_{2k}(v - T_h v + T_k v - \eta_{\sigma, \ell}(v))
    \qquad \text{for all $\sigma, \ell, h \in \N$}.
  \end{align*}
  According to \eqref{eq:landes:approx:conv} and since $T_h v \ra v$ a.e.\ in $\Omega \times (0, T)$ as $h \ra \infty$,
  two further applications of Lebesgue's theorem allow us to conclude
  \begin{align}\label{eq:tk_v_strong:conv:conv_f_varphi}
        \lim_{h \ra \infty} \lim_{\sigma \ra \infty} \lim_{j \ra \infty} \intntom \fej \varphi_{j, \sigma, \ell, h}
    &=  \lim_{h \ra \infty} \intntom f T_{2k}(v - T_h v)
    =   0
    \qquad \text{for all $\ell \in \N$}.
  \end{align}
  (We remark that it is not necessary to pass to subsequences here.)
  As to the first term on the right-hand side in \eqref{eq:tk_v_strong_conv:test}, we estimate
  \begin{align*}
    &\pe  \intntom \nabla \ve \cdot \nabla \varphi_{j, \sigma, l, h}
     =    \intntom \nabla \vej \cdot \nabla T_{2k}(\vej - T_h \vej + T_k \vej - \eta_{\sigma, \ell}(v)) \\
    &=    \int_{\{|\vej| \le k\}} \nabla \vej \cdot \nabla (T_k \vej - \eta_{\sigma, \ell}(v)) \\
    &\pe  - \int_{\{k \lt |\vej| \le h\}} \nabla \vej \cdot \nabla \eta_{\sigma, \ell}(v)
          + \int_{\{h \lt |\vej|\} \cap \{|\vej - h + k - \eta_{\sigma, \ell}(v)| \le 2k\}} \nabla \vej \cdot \nabla (\vej - \eta_{\sigma, \ell}(v)) \\
    &\ge  \intntom \nabla T_k \vej \cdot \nabla (T_k \vej - T_k v)
          + \intntom \nabla T_k \vej \cdot \nabla (T_k v - \eta_{\sigma, \ell}(v)) \\
    &\pe  - \int_{\{|\vej| \gt k\}} |\nabla T_{h + 4k}(\vej)| |\nabla \eta_{\sigma, \ell}(v)| \\
    &\ge  \intntom |\nabla (T_k \vej - T_k v)|^2
          + \intntom \nabla T_k v \cdot \nabla (T_k \vej - T_k v)
          + \intntom \nabla T_k \vej \cdot \nabla (T_k v - \eta_{\sigma, \ell}(v)) \\
    &\pe  - \int_{\{|\vej| \gt k\}} |\nabla T_{h + 4k}(\vej)| |\nabla T_k v|
          - \int_{\{|\vej| \gt k\}} |\nabla T_{h + 4k}(\vej)| |\nabla (T_k v - \eta_{\sigma, \ell}(v))|
  \end{align*}
  for all $j, \sigma, \ell, h \in \N$ with $h > k$.
  Since these integral terms satisfy
  \begin{alignat*}{2}
    &\lim_{j \ra \infty} \intntom \nabla T_k v \cdot \nabla (T_k \vej - T_k v) = 0 
    &&\qquad \text{for all $\sigma, \ell, h \in \N$}
    \intertext{by \eqref{eq:limit_ve:truncations_weak},}
    &\lim_{\sigma \ra \infty} \lim_{j \ra \infty} \intntom \nabla T_k \vej \cdot \nabla (T_k v - \eta_{\sigma, \ell}(v)) = 0
    &&\qquad \text{for all $\ell, h \in \N$}
    \intertext{by \eqref{eq:limit_ve:truncations_weak} and \eqref{eq:landes:approx:conv},}
    &\lim_{j \ra \infty} \int_{\{|\vej| \gt k\}} |\nabla T_{h + 4k}(\vej)| |\nabla T_k v| = 0
    &&\qquad \text{for all $\sigma, \ell, h \in \N$}
    \intertext{by \eqref{eq:limit_ve:pw} and \eqref{eq:limit_ve:truncations_weak} and}
    &\lim_{\sigma \ra \infty} \lim_{j \ra \infty} \int_{\{|\vej| \gt k\}} |\nabla T_{h + 4k}(\vej)| |\nabla (T_k v - \eta_{\sigma, \ell})| = 0
    &&\qquad \text{for all $\ell, h \in \N$}
  \end{alignat*}
  by \eqref{eq:limit_ve:pw}, \eqref{eq:limit_ve:truncations_weak} and \eqref{eq:landes:approx:conv},
  we can conclude
  \begin{align*}
          \limsup_{j \ra \infty} \intntom |\nabla (T_k \vej - T_k v)|^2
     \le  \limsup_{\sigma \ra \infty} \limsup_{j \ra \infty} \intntom \nabla \vej \cdot \nabla \varphi_{j, \sigma, \ell, h}
    \qquad \text{for all $\ell, h \in \N$}.
  \end{align*}  
  In combination with \eqref{eq:tk_v_strong_conv:test}, \eqref{eq:tk_v_strong:conv:conv_f_varphi} and \eqref{eq:vet_varphi_lower_bdd:statement}, this implies
  \begin{align*}
          \limsup_{j \ra \infty} \intntom |\nabla (T_k \vej - T_k v)|^2
    &\le  \limsup_{h \ra \infty} \limsup_{\ell \ra \infty} \limsup_{\sigma \ra \infty} \limsup_{j \ra \infty} \intntom \nabla \vej \cdot \nabla \varphi_{j, \sigma, \ell, h} \\
    &\le  - \liminf_{h \ra \infty} \liminf_{\ell \ra \infty} \liminf_{\sigma \ra \infty} \liminf_{j \ra \infty} \intntom \vejt \varphi_{j, \sigma, \ell, h}
    \le   0
  \end{align*}
  and hence also $\limsup_{j \ra \infty} \intntom |\nabla (T_k \vej - T_k v)|^2 = 0$, which is equivalent to \eqref{eq:tk_v_strong_conv:statement}.
\end{proof}

This finishes the most intricate part of the proof of Theorem~\ref{sec:intro:part1}, the remaining statements now follow rather quickly.
Indeed, as a straightforward consequence of Lemma~\ref{lm:tk_v_strong_conv} and the bounds obtained in Lemma~\ref{lm:nabla_z_llambda},
we next obtain convergence of $(\nabla \vej)_{j \in \N}$ in $L_{\loc}^\lambda(\Ombarinf)$ for all $\lambda \in [1, \tfrac{n+2}{n+1})$.
\begin{lemma}\label{lm:v_llambda_strong_conv}
  Suppose the hypotheses of Theorem~\ref{th:strong_grad_conv} are satisfied and let $v$ and $(\eps_j)_{j \in \N}$ be as given by Lemma~\ref{lm:limit_ve}.
  Then \eqref{eq:strong_grad_conv:v_pointwise_grad_conv} holds; that is,
  \begin{align*}
    \nabla \vej &\ra \nabla v
    \qquad \text{in $L_{\loc}^\lambda(\Ombarinf)$ for all $\lambda \in [1, \tfrac{n+2}{n+1})$ and a.e.\ in $(\Omega \times (0, \infty))$ as $j \ra \infty$}.
  \end{align*}
\end{lemma}
\begin{proof}
  Making use of a diagonalization argument and \eqref{eq:tk_v_strong_conv:statement}, 
  we see that there are a subsequence and a null set $N_1 \subset \Omega \times (0, \infty)$
  such that $\mathds 1_{\{|\vej| \le k\}} \nabla \vej \ra \mathds 1_{\{|v| \le k\}} \nabla v$ pointwise in $(\Omega \times (0, \infty)) \setminus N_1$ as $j \ra \infty$ for all $k \in \N$.
  As $v \in L_{\loc}^1(\Ombarinf)$ by \eqref{eq:limit_ve:pw}, $N_2 \defs \{|v| = \infty\}$ is a null set,
  so that we can conclude $\nabla \vej \ra \nabla v$ pointwise in $(\Omega \times (0, \infty)) \setminus (N_1 \cup N_2)$ as $j \ra \infty$.
  When combined with Lemma~\ref{lm:nabla_z_llambda} and Vitali's theorem, this yields the statement.
\end{proof}

In order to complete the proof of Theorem~\ref{th:strong_grad_conv}, we need to finally verify \eqref{eq:strong_grad_conv:v_weighted_grad_conv}.
Similarly as for \eqref{eq:strong_grad_conv:v_pointwise_grad_conv}, this again follows from Lemma~\ref{lm:tk_v_strong_conv} and the a priori estimates collected in Section~\ref{sec:apriori},
albeit in a less straightforward manner.
However, instead of directly proving \eqref{eq:strong_grad_conv:v_weighted_grad_conv}, we first state the following quite general lemma which will also be used in Section~\ref{sec:stronger_grad_conv} below.
\begin{lemma}\label{lm:strong_conv_product}
  Assume the hypotheses of Theorem~\ref{th:strong_grad_conv} hold and let $v$ and $(\eps_j)_{j \in \N}$ be as given by Lemma~\ref{lm:limit_ve}.
  Let $T \gt 0$, set $X_T \defs \bigcup_{j \in \N} \vej(\Ombar \times [0, T])$ and suppose that there are $\psi_1, \psi_2 \in C^0(X_T; [0, \infty))$ with 
  \begin{align}
    &\lim_{X_T \ni s \ra \pm \infty} \psi_1(s) = \infty, \label{eq:strong_conv_prod:psi_1_unbdd}\\
    &\sup_{s \in X_T} \psi_2(s) \mathds 1_{\{\psi_1(s) \le M\}} \lt \infty \qquad \text{for all $M \gt 0$ and} \label{eq:strong_conv_prod:psi_2_pw_bdd}\\
    &\lim_{M \ra \infty} \sup_{j \in \N} \int_{\{\psi_1(\vej) > M\}} \psi_2^2(\vej) |\nabla \vej|^2 = 0. \label{eq:strong_conv_prod:psi_2_l2_bdd}
  \end{align}
  Then 
  \begin{align}\label{eq:strong_conv_product:statement}
     \psi_2(\vej) \nabla \vej \ra \psi_2(v) \nabla v
     \qquad \text{in $L^2(\OmT)$ as $j \ra \infty$}.
  \end{align}
\end{lemma}
\begin{proof}
  For arbitrary $\eta > 0$, \eqref{eq:strong_conv_prod:psi_2_l2_bdd} first allows us to fix $M > 0$ such that
  \begin{align}\label{eq:strong_conv_product:eta}
    \int_{\{\psi_1(\vej) > M\}} \psi_2^2(\vej) |\nabla \vej|^2 \le \eta \qquad \text{for all $j \in \N$},
  \end{align}
  whenceupon we make use of \eqref{eq:strong_conv_prod:psi_1_unbdd} to obtain $k \in \N$ such that
  \begin{align}\label{eq:strong_conv_product:k}
    |\vej| \le k \quad \text{in $\{\psi_1(\vej) \le 2M\}$}
    \qquad \text{for all $j \in \N$}.
  \end{align}
  Therefore 
  \begin{align*}
          \intntom \psi_2^2(\vej) |\nabla \vej|^2
    &\le  \int_{\{\psi_1(\vej) \le 2M\}} \psi_2^2(\vej) |\nabla \vej|^2
          + \int_{\{\psi_1(\vej) > M\}} \psi_2^2(\vej) |\nabla \vej|^2 \\
    &\le  \sup_{s \in X_T} \left( \psi_2^2(s) \mathds 1_{\{\psi_1(s) \le 2M\}} \right) \int_{\{|\vej| \le k\}} |\nabla \vej|^2
          + \eta
    \qquad \text{holds for all $j \in \N$},
  \end{align*}
  so that \eqref{eq:strong_conv_prod:psi_2_pw_bdd} and Lemma~\ref{lm:nabla_tk_z} assert boundedness of $(\intntom \psi_2^2(\vej) |\nabla \vej|^2)_{j \in \N}$.
  Since moreover $\psi_2(\vej) \nabla \vej \ra \psi_2(v) \nabla v$ a.e.\ in $\Omega \times (0, T)$ as $j \ra \infty$ by \eqref{eq:limit_ve:pw} and Lemma~\ref{lm:v_llambda_strong_conv},
  we conclude $\psi_2(v) \nabla v \in L^2(\OmT)$ and
  \begin{align}\label{eq:strong_conv_prod:weak}
     \psi_2(\vej) \nabla \vej \rh \psi_2(v) \nabla v
     \qquad \text{in $L^2(\OmT)$ as $j \ra \infty$}.
  \end{align}
  We emphasize that we do not need to switch to a subsequence here.

  Let now $(\xi_1, \xi_2)$ be a smooth partition of unity subordinate to the open cover $(U_1, U_2) \defs ([0, 2M), (M, \infty))$ of $[0, \infty)$;
  that is, $\xi_i \in C^\infty(U_i)$ with $0 \le \xi_i \le 1$ and $\supp \xi_i \subset U_i$ for $i \in \{1, 2\}$ and $\xi_1 + \xi_2 = 1$ on $[0, \infty)$.
  By \eqref{eq:limit_ve:pw}, \eqref{eq:tk_v_strong_conv:statement}, \eqref{eq:strong_conv_prod:psi_2_pw_bdd} and Lebesgue's theorem, we then obtain
  \begin{align*}
          \psi_2(\vej) \xi_1^{\frac12}(\psi_1(\vej)) \mathds 1_{\{|\vej| \le k\}} \nabla \vej
    &\ra  \psi_2(v) \xi_1^{\frac12}(\psi_1(v)) \mathds 1_{\{|v| \le k\}} \nabla v
    \qquad \text{in $L^2(\OmT)$ as $j \ra \infty$}.
  \end{align*}
  According to \eqref{eq:strong_conv_product:k}, this is equivalent to
  \begin{align*}
          \psi_2(\vej) \xi_1^{\frac12}(\psi_1(\vej)) \nabla \vej
    &\ra  \psi_2(v) \xi_1^{\frac12}(\psi_1(v)) \nabla v
    \qquad \text{in $L^2(\OmT)$ as $j \ra \infty$}.
  \end{align*}
  Also making use of \eqref{eq:strong_conv_product:eta}, we then obtain
  \begin{align*}
    &\pe  \limsup_{j \ra \infty} \intntom \psi_2^2(\vej) |\nabla \vej|^2 \\
    &\le  \limsup_{j \ra \infty} \intntom \psi_2^2(\vej) \xi_1(\psi_1(\vej)) |\nabla \vej|^2
          + \limsup_{j \ra \infty} \intntom \psi_2^2(\vej) \xi_2(\psi_1(\vej)) |\nabla \vej|^2 \\
    &\le  \intntom \psi_2^2(v) \xi_1(\psi_1(v)) |\nabla v|^2
          + \limsup_{j \ra \infty} \int_{\{\psi_1(\vej) \gt M\}} \psi_2^2(\vej) |\nabla \vej|^2 \\
    &\le  \intntom \psi_2^2(v) |\nabla v|^2
          + \eta.
  \end{align*}
  Since $\eta$ was chosen arbitrarily, we conclude
  \begin{align*}
    \limsup_{j \ra \infty} \intntom \psi_2^2(\vej) |\nabla \vej|^2 \le  \intntom \psi_2^2(v) |\nabla v|^2.
  \end{align*}
  Combined with \eqref{eq:strong_conv_prod:weak} and the weak lower semicontinuity of the norm,
  this implies
  \begin{align*}
    \lim_{j \ra \infty} \intntom \psi_2^2(\vej) |\nabla \vej|^2 = \intntom \psi_2^2(v) |\nabla v|^2.
  \end{align*}
  Again making use of \eqref{eq:strong_conv_prod:weak}, we arrive at \eqref{eq:strong_conv_product:statement}.
\end{proof}

Let us note sufficient conditions (which may be easier to verify in certain applications) for \eqref{eq:strong_conv_prod:psi_2_pw_bdd} and \eqref{eq:strong_conv_prod:psi_2_l2_bdd}.
\begin{lemma}\label{lm:strong_conv_product_extra}
  Assume the hypotheses of Theorem~\ref{th:strong_grad_conv} hold and let $v$ and $(\eps_j)_{j \in \N}$ be as given by Lemma~\ref{lm:limit_ve}.
  Let $T \gt 0$, let $X_T$ be as in Lemma~\ref{lm:strong_conv_product} and let $\psi_1, \psi_2 \in C^0(X_T; [0, \infty))$.
  \begin{enumerate}
    \item
      If
      \begin{align}\label{eq:conv_product_extra:psi_2_c0}
        \psi_2 \in C^0(\R),
      \end{align}
      then \eqref{eq:strong_conv_prod:psi_2_pw_bdd} follows from \eqref{eq:strong_conv_prod:psi_1_unbdd}.

    \item
      Suppose there exist a nonnegative function $\theta \in C^0([0, \infty))$ with
      \begin{align}\label{eq:conv_product_extra:theta_growth}
        \lim_{s \ra \infty} \theta(s) = \infty
      \end{align}
      and $C > 0$ such that
      \begin{align}\label{eq:conv_product_extra:bound}
        \intntom \theta(\psi_1(\vej)) \psi_2^2(\vej) |\nabla \vej|^2 \le C
        \qquad \text{for all $j \in \N$}.
      \end{align}
      Then \eqref{eq:strong_conv_prod:psi_2_l2_bdd} is fulfilled.
  \end{enumerate}
\end{lemma}
\begin{proof}
  \begin{enumerate}
    \item
      For $M > 0$, by \eqref{eq:strong_conv_prod:psi_1_unbdd} there is $k \in \N$ such that $|\vej| \le k$ in $\{\psi_1(\vej) \le M\}$.
      Since $\psi_2 \in C^0([-k, k])$ by \eqref{eq:conv_product_extra:psi_2_c0}, this implies \eqref{eq:strong_conv_prod:psi_2_pw_bdd}.

    \item
      Without loss of generality, we may assume $\theta > 0$ in $[0, \infty)$.
      Then
      \begin{align*}
              \sup_{j \in \N} \int_{\{\psi_1(\vej) > M\}} \psi_2^2(\vej) |\nabla \vej|^2
        &\le  \sup_{s \gt M} \left(\frac{1}{\theta(s)}\right) \sup_{j \in \N} \intntom \theta(\psi_1(\vej)) \psi_2^2(\vej) |\nabla \vej|^2 \\
        &\le  \frac{C}{\inf_{s \gt M} \theta(s)} 
        \ra 0
        \qquad \text{as $M \ra \infty$},
      \end{align*}
      that is, \eqref{eq:strong_conv_prod:psi_2_l2_bdd} holds.
      \qedhere
  \end{enumerate}
\end{proof}

Lemma~\ref{lm:strong_conv_product_extra} and Lemma~\ref{lm:strong_conv_product} now indeed allow us to quickly prove \eqref{eq:strong_grad_conv:v_weighted_grad_conv}.
\begin{lemma}\label{lm:v_r_nabla_v_strong_conv}
  Suppose the hypotheses of Theorem~\ref{th:strong_grad_conv} are satisfied and let $v$ and $(\eps_j)_{j \in \N}$ be as given by Lemma~\ref{lm:limit_ve}.
  Then \eqref{eq:strong_grad_conv:v_weighted_grad_conv} holds, that is
  \begin{align}\label{eq:v_r_nabla_v_strong_conv:statement}
    \mathds (|\vej|+1)^{-r} \nabla \vej &\ra (|v|+1)^{-r} \nabla v
    \qquad \text{in $L_{\loc}^2(\Ombarinf)$  for all $r > \tfrac12$ as $j \ra \infty$}.
  \end{align}
\end{lemma}
\begin{proof}
  We let $r > \frac12$ and $\eta \defs \frac{2r-1}{2} > 0$.
  Setting $\psi_1(s) \defs |s|$, $\psi_2(s) \defs (|s|+1)^{-r}$ and $\theta(s) \defs (|s|+1)^{\eta}$ for $s \in \R$,
  we see that \eqref{eq:strong_conv_prod:psi_1_unbdd}, \eqref{eq:conv_product_extra:psi_2_c0} and \eqref{eq:conv_product_extra:theta_growth} are fulfilled.
  As moreover Lemma~\ref{lm:alpha_nabla_z} (applied to $\alpha \defs 2r - 1 - \eta > 0$) asserts that \eqref{eq:conv_product_extra:bound} holds for some $C > 0$,
  Lemma~\ref{lm:strong_conv_product_extra} and Lemma~\ref{lm:strong_conv_product} imply \eqref{eq:v_r_nabla_v_strong_conv:statement}.
\end{proof}

\subsection{Proof of Theorem~\ref{th:strong_grad_conv}}\label{sec:conv_l1:proof}
To conclude this section, we note that the lemmata above already contain all the conclusions from Theorem~\ref{th:strong_grad_conv}.
\begin{proof}[Proof of Theorem~\ref{th:strong_grad_conv}]
  The existence of a subsequence $(\eps_j)_{j \in \N}$ of $(\eps_j')_{j \in \N}$ and a function $v$ satisfying \eqref{eq:strong_grad_conv_gen:v_reg}
  such that \eqref{eq:strong_grad_conv:v_l1_conv}--\eqref{eq:strong_grad_conv:v_weighted_grad_conv} hold
  follow from Lemma~\ref{lm:limit_ve}, Lemma~\ref{lm:v_llambda_strong_conv}, Lemma~\ref{lm:tk_v_strong_conv} and Lemma~\ref{lm:v_r_nabla_v_strong_conv}.
  Moreover, Lemma~\ref{lm:v_weak_sol} and Lemma~\ref{lm:unique_weak} state that $v$ is the unique weak solution of \eqref{prob:v_limit}.
\end{proof}

\section{Stronger convergence properties of the data imply stronger convergence properties of the solutions: proof of Theorem~\ref{th:stronger_grad_conv}}\label{sec:stronger_grad_conv}
This section is devoted to the question whether (and if so, how) stronger convergence properties of $(\vnej)_{j \in \N}$ and $(\fej)_{j \in \N}$
than those required by Theorem~\ref{th:strong_grad_conv} allow us to obtain stronger convergence properties for $(\vej)_{j \in \N}$ than those asserted
by \eqref{eq:strong_grad_conv:v_l1_conv}--\eqref{eq:strong_grad_conv:v_grad_conv}.

The eventual goal of this section is then to prove Theorem~\ref{th:stronger_grad_conv},
which gives certain affirmative answers to this question under quite general assumptions.
However, for the sake of exposition,
let us first consider a special case, namely that (the hypotheses of Theorem~\ref{th:strong_grad_conv} hold, that additionally) $\vnej \ge 0$, $\fej \ge 0$ and hence $\vej \ge 0$ for $j \in \N$ and that
\begin{alignat}{2}
  \vne &\ra v_0
  &&\qquad \text{in $\leb p$}, \label{eq:futher_conv:intro_conv_ve} \\
  \ve^{p-1} \fe &\ra v^{p-1} f
  &&\qquad \text{in $L_{\loc}^1(\Ombarinf)$} \label{eq:futher_conv:intro_conv_fe}
\end{alignat}
as $\eps = \eps_j \sea 0$ for some $p > 1$.
Since testing the first equation in \eqref{prob:ve} with $(\ve+1)^{p-1}$ yields
\begin{align}\label{eq:further_conv:test_ve_p}
    \frac1p \intom (\vej(\cdot, T)+1)^p
    + (p-1) \intntom (\vej+1)^{p-2} |\nabla \vej|^2
  = \frac1p \intom (\vnej+1)^p
    + \intntom (\vej+1)^{p-1} \fej
\end{align}
for all $T > 0$ and $j \in \N$,
these assumptions make Lemma~\ref{lm:strong_conv_product_extra} and Lemma~\ref{lm:strong_conv_product}
(for $\psi_1(s) = |s|$, $\psi_2(s) = (|s|+1)^\frac{q-2}{2}$ and $\theta(s) = (|s|+1)^{p-q}$, $s \in \R$, $q < p$) applicable.
That is, we directly obtain
\begin{align}\label{eq:further_conv:conv_tilde_p}
  (\vej+1)^\frac{q-2}{2} \nabla \vej \ra (v+1)^\frac{q-2}{2} \nabla v
  \qquad \text{in $L_{\loc}^2(\Ombarinf)$ as $j \ra \infty$}
\end{align}
for all $q < p$.
On the other hand, \eqref{eq:further_conv:test_ve_p} (with $p$ replaced by $q$) also suggests that
if $q > p$ and at least one of the families $(\vnej)_{j \in \N}$ or $(\vej^{q-1} \fej)_{j \in \N}$ is unbounded in $\leb{q}$ or $L^1(\OmT)$ for some $T > 0$, respectively,
(which is consistent with \eqref{eq:futher_conv:intro_conv_ve} and \eqref{eq:futher_conv:intro_conv_fe}),
then the left-hand side in \eqref{eq:further_conv:test_ve_p} is unbounded and thus one probably cannot expect \eqref{eq:further_conv:conv_tilde_p} to hold.
This leaves the question whether \eqref{eq:further_conv:conv_tilde_p} continues to hold for the in some sense critical case $q = p$.
Theorem~\ref{th:stronger_grad_conv} is able to give an affirmative answer.

The first key idea for its proof
is to make use of Vitali's theorem and the de la Vall\'ee Poussin theorem,
which assert that \eqref{eq:futher_conv:intro_conv_ve} and \eqref{eq:futher_conv:intro_conv_fe} imply slightly stronger bounds
for $(\vnej)_{j \in \N}$ and $(\vej^{p-1} \fej)_{j \in \N}$ than those made use of above.
If these translate into stronger bounds for $(\vej)_{j \in \N}$,
we will be in a place to again conclude \eqref{eq:further_conv:conv_tilde_p} from Lemma~\ref{lm:strong_conv_product_extra} and Lemma~\ref{lm:strong_conv_product}, this time even for $q = p$.

As we will see in the proof of Lemma~\ref{lm:theta_v_nabla_v} below, however,
the ``usual'' formulation of the de la Vall\'ee Poussin theorem turns out to be insufficient for our purposes.
Instead we will employ a version thereof recently proved by Lankeit \cite{LankeitImmediateSmoothingGlobal2021},
which states that the function $\Phi$ given by the de la Vall\'ee Poussin theorem can be assumed to be
such that $\Phi''$ (exists and) is pointwise bounded from above by a given function $g$ \emph{not} belonging to $L^1((1, \infty))$.
In particular, for $g(s) \defs \frac1s$, $s > 0$, we have the following
\begin{lemma}\label{lm:delavalleepoussin}
  For $i \in \{1, 2\}$, let $G_i \subset \R^{m_i}$, $m_i \in \N$, be a measurable set with $|G_i| \lt \infty$.
  Given uniformly integrable families $\mc F_i \subset L^1(G_i)$, $i \in \{1, 2\}$,
  we can find a function $\Phi \in C^2([0, \infty))$ and $\newgc{phi_uniform_int} > 0$ with the following properties:
  \begin{alignat}{2}
     & \Phi \ge 1, \Phi' \ge 0, \Phi'' \ge 0                         && \qquad \text{in $[0, \infty)$}, \label{eq:delavalleepoussin:phi_ge_0} \\
     & \frac{\Phi(s)}{s} \ra \infty \text{ and } \Phi'(s) \ra \infty && \qquad \text{as $s \ra \infty$}, \label{eq:delavalleepoussin:phi_superlin} \\
     & s \Phi''(s) \le 1                                             && \qquad \text{for all $s \ge 0$}, \label{eq:delavalleepoussin:s_phi'_le_1} \\
     & \int_{G_i} \Phi(|z|) \le \gc{phi_uniform_int}                 && \qquad \text{for all $z \in \mc F_i$ and all $i \in \{1, 2\}$}. \label{eq:delavalleepoussin:z_uniform_int}
  \end{alignat}
\end{lemma}
\begin{proof}
  Since $\mc F_1$ and $\mc F_2$ are uniformly integrable, \cite[Lemma~3.1]{LankeitImmediateSmoothingGlobal2021} provides us with functions $\Psi_1, \Psi_2 \in C^2([0, \infty))$
  and $\newlc{psi_1_uniform_int}, \newlc{psi_2_uniform_int} > 0$
  such that \eqref{eq:delavalleepoussin:phi_ge_0} as well as \eqref{eq:delavalleepoussin:phi_superlin} hold for $\Phi$ replaced by $\Psi_i$
  and that $\int_{G_i} \Psi_i(z) \le c_i$ for all $z \in \mc F_i$ and all $i \in \{1, 2\}$.
  Since $\Psi_1'' \notin L^1((0, \infty))$ (else $\Psi_1'$ would be bounded, contradicting \eqref{eq:delavalleepoussin:phi_superlin}) and also $(1, \infty) \ni s \mapsto \frac1s \notin L^1((1, \infty))$,
  we may apply \cite[Lemma~3.3]{LankeitImmediateSmoothingGlobal2021} to obtain $h_1 \in C^0([0, \infty))$ with $h_1 \notin L^1((0, \infty))$ such that
  \begin{align*}
    0 \le h_1(s) \le \frac1s
    \quad \text{and} \quad
    \int_0^s h_1(\sigma) \dsigma \le \int_0^s \Psi_1''(\sigma) \dsigma
    \qquad \text{for all $s \ge 0$}.
  \end{align*}

  By another application of \cite[Lemma~3.3]{LankeitImmediateSmoothingGlobal2021}, there is $h_2 \in C^0([0, \infty))$ with $h_2 \notin L^1((0, \infty))$ such that
  \begin{align*}
    0 \le h_2(s) \le h_1(s)
    \quad \text{and} \quad
    \int_0^s h_2(\sigma) \dsigma \le \int_0^s \Psi_2''(\sigma) \dsigma
    \qquad \text{for all $s \ge 0$}.
  \end{align*}
  Combining these estimates and applying the fundamental lemma of calculus yields
  \begin{align}\label{eq:delavalleepoussin:h2_combined}
    0 \le h_2(s) \le \frac1s
    \quad \text{and} \quad
    \int_0^s h_2(\sigma) \dsigma \le \min\left\{ \Psi_1'(s) - \Psi_1'(0), \Psi_2'(s) - \Psi_2'(0) \right\} 
    \qquad \text{for all $s \ge 0$}.
  \end{align}

  We now claim that
  \begin{align*}
    \Phi \colon [0, \infty) \ra \R, \quad s \mapsto 1 + \int_0^s \left( \min\{\Psi_1'(0), \Psi_2'(0)\} + \int_0^\sigma h_2(\tau) \dtau\right) \dsigma
  \end{align*}
  satisfies all desired properties.
  Indeed, the inclusion $\Phi \in C^2([0, \infty))$ and \eqref{eq:delavalleepoussin:phi_ge_0} follow from continuity and nonnegativity of $h_2$, $\Psi_1'$ and $\Psi_2'$,
  \eqref{eq:delavalleepoussin:phi_superlin} holds since $\Phi'' = h_2$ does not belong to $L^1((0, \infty))$
  and \eqref{eq:delavalleepoussin:s_phi'_le_1} follows from the first part of \eqref{eq:delavalleepoussin:h2_combined}.
  Finally, the second part of \eqref{eq:delavalleepoussin:h2_combined} entails $\Phi' \le \min\{\Psi_1', \Psi_2'\}$ and thus also $\Phi \le \min\{\Psi_1, \Psi_2\}$ in $[0, \infty)$.
  Therefore, $\int_{G_i} \Phi(|z|) \le \int_{G_i} \Psi_i(|z|) \le c_i$ for all $z \in \mc F_i$ and all $i \in \{1, 2\}$,
  implying \eqref{eq:delavalleepoussin:z_uniform_int} for $\gc{phi_uniform_int} \defs \max\{\lc{psi_1_uniform_int}, \lc{psi_2_uniform_int}\}$.
\end{proof}

The importance of \eqref{eq:delavalleepoussin:s_phi'_le_1}, the main improvement compared to the ``usual'' version of the de la Vall\'ee Poussin theorem, for our purposes is that this bound for $\Phi''$ is a crucial ingredient in the proof of the following Young-type inequality.
\begin{lemma}\label{lm:phi_young}
  Suppose $\Phi \in C^2([0, \infty))$ fulfills \eqref{eq:delavalleepoussin:phi_ge_0}--\eqref{eq:delavalleepoussin:s_phi'_le_1}.
  Then there exists $\newgc{phi_young} \gt 0$ such that
  \begin{align}\label{phi_young:ineq}
   a \Phi'(b) \le \Phi(a) + \gc{phi_young} \Phi(b)
   \qquad \text{for all $a, b \ge 0$}.
  \end{align}
\end{lemma}
\begin{proof}
  Since
  \begin{align*}
        1
    \le \frac{\Phi'(s) + s \Phi''(s)}{\Phi'(s)}
    \le 1 + \frac{1}{\Phi'(s)}
    \ra 1
    \qquad \text{as $s \ra \infty$}
  \end{align*}
  by \eqref{eq:delavalleepoussin:phi_ge_0}--\eqref{eq:delavalleepoussin:s_phi'_le_1},
  L'Hôpital's rule asserts that
  \begin{align}\label{eq:phi_young:lopital}
      \lim_{s \ra \infty} \frac{s \Phi'(s)}{\Phi(s)}
    = \lim_{s \ra \infty} \frac{\Phi'(s) + s \Phi''(s)}{\Phi'(s)}
    = 1.
  \end{align}
  Because $\Phi \in C^1([0, \infty))$ is positive and hence $[0, \infty) \ni s \mapsto \frac{s \Phi'(s)}{\Phi(s)}$ is continuous on $[0, \infty)$, and due to \eqref{eq:phi_young:lopital},
  there is $\newlc{c_1} \gt 1$ such that
  \begin{align}\label{eq:phi_young:s_phi'_le_phi}
    s \Phi'(s) \le \lc{c_1} \Phi(s)
    \qquad \text{for all $s \ge 0$}.
  \end{align}

  Moreover, convexity of $\Phi$ implies
  \begin{align*}
    \begin{cases}
      \Phi'(b) \le \Phi'(s) \text{ for all $s \in [b, a]$} & \text{if }a \ge b, \\
      \Phi'(b) \ge \Phi'(s) \text{ for all $s \in [a, b]$} & \text{if }a \le b,
    \end{cases}
  \end{align*}
  that is, $(a-b)\Phi'(b) \le \int_b^a \Phi'(s) \ds$ for all $a, b \ge 0$.
  Recalling  \eqref{eq:phi_young:s_phi'_le_phi} and setting $\gc{phi_young} \defs c_1-1 > 0$, we conclude that
  \begin{align*}
          a \Phi'(b)
    &=    (a - b) \Phi'(b) + b \Phi'(b) \\
    &\le  \int_b^a \Phi'(s) \ds + b \Phi'(b) \\
    &\le  \Phi(a) - \Phi(b) + c_1 \Phi(b) 
     =    \Phi(a) + \gc{phi_young} \Phi(b)
  \end{align*}
  for all $a, b \ge 0$.
\end{proof}

With these preparations at hand, we now see that under the conditions of Theorem~\ref{th:stronger_grad_conv},
the convergence properties \eqref{eq:stronger_grad_conv:conv_v0} and \eqref{eq:stronger_grad_conv:conv_v_f}
indeed can be used to derive stronger bounds for $(\vej)_{j \in \N}$
than those implied by mere boundedness of $(\psi(\vnej))_{j \in \N}$ and $(\psi'(\vej) \fej)_{j \in \N}$ in $\leb1$ and $L^1(\OmT)$, $T > 0$.
\begin{lemma}\label{lm:theta_v_nabla_v}
  Suppose the hypotheses of Theorem~\ref{th:stronger_grad_conv} hold and let $T \gt 0$.
  Then
  \begin{align}\label{eq:theta_v_nabla_v:uniform_int}
    (\psi(\vej))_{j \in \N} \quad \text{is uniformly integrable in } \OmT.
  \end{align}
  Moreover, there exist a nonnegative function $\theta \in C^0([0, \infty))$ with
  \begin{align}\label{eq:theta_v_nabla_v:theta_growth}
    \lim_{s \ra \infty} \theta(s) = \infty
  \end{align}
  and $\newgc{theta_nabla_v} \gt 0$ such that
  \begin{align}\label{eq:theta_v_nabla_v:bound}
    \intntom \theta(\psi(\vej)) \psi''(\vej) |\nabla \vej|^2 \le \gc{theta_nabla_v}.
    \qquad \text{for all $j \in \N$}.
  \end{align}
\end{lemma}
\begin{proof}
  For $\eps \in (0, 1)$, we set $\we \defs \psi(\ve)$.
  Then
  \begin{align*}
    \wet = \psi'(\ve) \vet, \quad
    \nabla \we = \psi'(\ve) \nabla \ve
    \quad \text{and} \quad
    \Delta \we = \psi'(\ve) \Delta \ve + \psi''(\ve) |\nabla \ve|^2 \qquad \text{in $(0, T)$}
  \end{align*}
  and thus $\we \in C^0(\Ombar \times [0, T]) \cap C^{2, 1}(\Ombar \times (0, T])$ is a classical solution of the problem
  \begin{align}\label{eq:theta_v_nabla_v:w_eq}
    \begin{cases}
      \wet = \Delta \we - \psi''(\ve) |\nabla \ve|^2 + \psi'(\ve) \fe & \text{in $\Omega \times (0, T)$}, \\
      \partial_\nu \we = 0                                            & \text{on $\partial \Omega \times (0, T)$}, \\
      \we(\cdot, 0) = \psi(\vne)                                      & \text{in $\Omega$}
    \end{cases}
  \end{align}
  for all $\eps \in (0, 1)$.

  Due to \eqref{eq:stronger_grad_conv:conv_v0} and \eqref{eq:stronger_grad_conv:conv_v_f}, Vitali's theorem asserts that
  $\mc F_1 \defs (\psi(\vnej))_{j \in \N}$ and $\mc F_2 \defs (\psi'(\vej) \fej)_{j \in \N}$ are uniformly integrable in $G_1 \defs \Omega$ and $G_2 \defs \OmT$, respectively.
  Thus, Lemma~\ref{lm:delavalleepoussin} provides us with $\Phi \in C^2([0, \infty))$ and $\gc{phi_uniform_int} > 0$ such that \eqref{eq:delavalleepoussin:phi_ge_0}--\eqref{eq:delavalleepoussin:z_uniform_int} hold.
  This in turn makes Lemma~\ref{lm:phi_young} applicable; that is, \eqref{phi_young:ineq} holds for some $\gc{phi_young} > 0$.
   
  By testing the differential equation in \eqref{eq:theta_v_nabla_v:w_eq} with $\Phi'(\we)$ and applying \eqref{phi_young:ineq}, we obtain
  \begin{align*}
    &\pe  \ddt \intom \Phi(\wej)
          + \intom \Phi''(\wej) |\nabla \wej|^2
          + \intom \Phi'(\wej) \psi''(\vej) |\nabla \vej|^2 \\
    &=    \intom \Phi'(\wej)\psi'(\vej) \fej
     \le  \intom \Phi(|\psi'(\vej) \fej|) + \gc{phi_young} \intom \Phi(\wej)
    \qquad \text{in $(0, T)$ for all $j \in \N$},
  \end{align*}
  which thanks to $\Phi'' \ge 0$ and \eqref{eq:delavalleepoussin:z_uniform_int} upon an integration in time results in
  \begin{align}\label{eq:theta_v_nabla_v:first_est}
    &\pe  \intom \Phi(\wej(\cdot, t))
          + \intnstom \Phi'(\wej) \psi''(\vej) |\nabla \vej|^2 \notag \\
    &\le  \intom \Phi(\wej(\cdot, 0))
          + \intnstom \Phi(\psi'(\vej) \fej)
          + \gc{phi_young} \intnstom \Phi(\wej) \notag \\
    &\le  2 \gc{phi_uniform_int}
          + \gc{phi_young} \intnstom \Phi(\wej)
    \qquad \text{for all $t \in (0, T)$ and $j \in \N$}.
  \end{align}
  By Gronwall's lemma and as $\Phi', \psi'' \ge 0$, this first implies
  \begin{align}\label{eq:theta_v_nabla_v:Phi_w_bdd}
          \intom \Phi(\wej(\cdot, t))
    &\le  2\gc{phi_uniform_int} \ure^{\gc{phi_young} T}
    \qquad \text{for all $t \in (0, T)$ and $j \in \N$}
  \end{align}
  and hence \eqref{eq:theta_v_nabla_v:uniform_int} due to the de la Vall\'ee Poussin theorem.
  By inserting \eqref{eq:theta_v_nabla_v:Phi_w_bdd} into \eqref{eq:theta_v_nabla_v:first_est}, we further obtain
  \begin{align*}
          \intntom \Phi'(\wej) \psi''(\vej) |\nabla \vej|^2
    &\le  2\gc{phi_uniform_int} (1 + \gc{phi_young}T \ure^{\gc{phi_young} T})
          \sfed \gc{theta_nabla_v}
    \qquad \text{for all $j \in \N$}.
  \end{align*}
  Upon setting $\theta(s) \defs \Phi'(s)$ for $s \in X_T$ and recalling \eqref{eq:delavalleepoussin:phi_superlin},
  we can thus conclude that also \eqref{eq:theta_v_nabla_v:theta_growth} and \eqref{eq:theta_v_nabla_v:bound} hold.
\end{proof}

As already alluded to, the derived bounds make Lemma~\ref{lm:strong_conv_product_extra} and Lemma~\ref{lm:strong_conv_product} applicable,
yielding the desired convergence of $((\psi''(\vej))^\frac12 \nabla \vej)_{j \in \N}$ in $L_{\loc}^2(\Ombarinf)$.
\begin{proof}[Proof of Theorem~\ref{th:stronger_grad_conv}]
  Letting $T > 0$ and setting $\psi_1 \defs \psi$ and $\psi_2 \defs (\psi'')^\frac12$,
  we see that \eqref{eq:strong_conv_prod:psi_1_unbdd} and \eqref{eq:strong_conv_prod:psi_2_pw_bdd}
  directly follow from \eqref{eq:stronger_grad_conv:psi_ra_infty} and \eqref{eq:stronger_grad_conv:psi''_bdd},
  while \eqref{eq:phi_young:s_phi'_le_phi} and \eqref{eq:theta_v_nabla_v:theta_growth} imply \eqref{eq:conv_product_extra:theta_growth} and \eqref{eq:conv_product_extra:bound}
  (for $\theta$ given by Lemma~\ref{lm:theta_v_nabla_v}).
  Therefore, we may apply Lemma~\ref{lm:strong_conv_product_extra} and Lemma~\ref{lm:strong_conv_product} to obtain \eqref{eq:stronger_grad_conv:conv_psi''_nabla_v}.
  Moreover, \eqref{eq:stronger_grad_conv:conv_psi_v} follows from \eqref{eq:strong_grad_conv:v_l1_conv}, \eqref{eq:theta_v_nabla_v:uniform_int} and Vitali's theorem.
\end{proof}

\section{Application I: Keller--Segel systems with superlinear dampening}\label{sec:a1}
As a first application of Theorem~\ref{th:strong_grad_conv}, we now construct global generalized solutions to
the Keller--Segel system with superlinear dampening,
\begin{align}\label{prob:ks_damp}
  \begin{cases}
    u_t = \Delta u - \nabla \cdot (u \nabla v) + g(u) & \text{in $\Omega \times (0, \infty)$}, \\
    v_t = \Delta v - v + u                            & \text{in $\Omega \times (0, \infty)$}, \\
    \partial_\nu u = \partial_\nu v= 0                & \text{on $\partial \Omega \times (0, \infty)$}, \\
    u(\cdot, 0) = u_0, v(\cdot, 0) = v_0              & \text{in $\Omega$},
  \end{cases}
\end{align}
in smooth bounded domains $\Omega \subset \R^n$, $n \in \N$, for dampening terms $g \in C^1([0, \infty))$ merely fulfilling
\begin{align}\label{eq:a1:cond_g}
  g(0) \ge 0
  \quad \text{and} \quad
  \frac{g(s)}{s} \ra -\infty \text{ as } s \ra \infty.
\end{align}
For $g \equiv 0$, such systems have been introduced by Keller and Segel in \cite{KellerSegelInitiationSlimeMold1970} to model the aggregation behavior of slime mold.
The key feature of this system is that the organisms under consideration, whose density is denoted by $u$, are attracted by a chemical, whose density is denoted by $v$, which they produce themselves.
This effect, called chemotaxis, corresponds to the term $-\nabla \cdot (u \nabla v)$ in the first equation in \eqref{prob:ks_damp}
and counters the stabilization properties associated with the heat equation.
In certain situations (see \cite{HerreroVelazquezBlowupMechanismChemotaxis1997, MizoguchiWinklerBlowupTwodimensionalParabolic, WinklerFinitetimeBlowupHigherdimensional2013}
or the survey \cite{LankeitWinklerFacingLowRegularity2019}, for instance)
this then leads to the probably most drastic form of pattern formation, namely finite-time blow-up.

As the body of mathematical literature regarding \eqref{prob:ks_damp} and related systems is huge,
instead of giving a detailed review thereof, we confine ourselves with referencing to the survey \cite{BellomoEtAlMathematicalTheoryKeller2015} for an overview
and just state a few results regarding the global existence of solutions to \eqref{prob:ks_damp} with nontrivial $g$.
Indeed, the introduction of such zeroth order terms, which inter alia model intrinsic logistic-type growth \cite{HillenPainterUserGuidePDE2009},
in the first equation is one of the most popular modifications of the classical Keller--Segel system.
For the probably most natural choice for $g$, $g(u) = \lambda u - \mu u^2$ for fixed $\lambda, \mu > 0$,
global classical solutions of \eqref{prob:ks_damp} are known to exist for all reasonably smooth initial data both
for planar domains (and arbitrary positive~$\mu$) \cite{OsakiEtAlExponentialAttractorChemotaxisgrowth2002}
and in the higher dimensional setting if $\mu$ is large enough \cite{WinklerBoundednessHigherdimensionalParabolic2010, XiangChemotacticAggregationLogistic2018}.
On the other hand, it has recently been shown in \cite{FuestApproachingOptimalityBlowup2021} that solutions to a simplified version of \eqref{prob:ks_damp} may blow up in finite time
for certain initial data provided $n \ge 5$ and $\mu$ is sufficiently small
(cf.\ also the precedents \cite{WinklerBlowupHigherdimensionalChemotaxis2011}, \cite{WinklerFinitetimeBlowupLowdimensional2018} and \cite{BlackEtAlRelaxedParameterConditions2021}).

These results indicate that global classical solutions to \eqref{prob:ks_damp} might fail to exist for certain dampening terms $g$,
which in turn motivates the introduction of weaker solution concepts.
For the parabolic--parabolic system \eqref{prob:ks_damp}, this has been first done in \cite{LankeitEventualSmoothnessAsymptotics2015},
where global weak solutions were obtained for quadratically growing $-g$.
Regarding weaker dampening terms, global generalized solutions have been constructed in \cite{WinklerRoleSuperlinearDamping2019} under certain growth assumptions on $-g$.
These conditions have then first been slightly relaxed in \cite{YanFuestWhenKellerSegel2021}
and then further reduced to \eqref{eq:a1:cond_g} in \cite{WinklerSolutionsParabolicKellertoappear}.
In view of solutions to (parabolic--elliptic analogues of) \eqref{prob:ks_damp} with $g \equiv 0$ collapsing to persistent Dirac-type singularities in finite time
\cite{BilerRadiallySymmetricSolutions2008},
the latter requirement is conjectured to be optimal.

In the present section, we show that Theorem~\ref{th:strong_grad_conv} allows for a simplified proof of a slightly stronger (cf.\ Remark~\ref{rm:a1:sol_concept_discussion} below)
version of the result in \cite{WinklerSolutionsParabolicKellertoappear}. That is, we prove
\begin{theorem}\label{th:a1}
  Let $\Omega \subset \R^n$, $n \ge 2$, be a smooth, bounded domain and suppose that $g \in C^1([0, \infty))$ complies with \eqref{eq:a1:cond_g}.
  For any given nonnegative $u_0, v_0 \in \leb1$, there then exists a global generalized solution $(u, v)$ in the sense of Definition~\ref{def:a1:sol_concept} below.
\end{theorem}

\subsection{Solution concept}\label{sec:a1:sol_concept}
The concept of generalized solutions of \eqref{prob:ks_damp} considered in this section is the following.
\begin{definition}\label{def:a1:sol_concept}
  Let $\Omega \subset \R^n$, $n \in \N$, be a smooth, bounded domain, $u_0, v_0 \in \leb1$ nonnegative and $g \in C^0([0, \infty))$.
  We call a tuple of nonnegative functions $(u, v) \in (L_{\loc}^1(\Ombarinf))^2$ with
  \begin{align*}
    g(u) \in L_{\loc}^1(\Ombarinf)
    \quad \text{and} \quad
    \nabla v \in L_{\loc}^1(\Ombarinf)
  \end{align*}
  as well as
  \begin{align*}
    \mathds 1_{\{u, v \le M\}} \nabla u \in L_{\loc}^2(\Ombarinf)
    \quad \text{and} \quad 
    \mathds 1_{\{u \le M\}} \nabla v \in L_{\loc}^2(\Ombarinf)
    \qquad \text{for all $M > 0$}
  \end{align*}
  a \emph{global generalized solution} of \eqref{prob:ks_damp}, if
  \begin{itemize}
    \item
      for all $\phi \in C^\infty([0, \infty)^2)$ with $D \phi \in C_c^\infty([0, \infty)^2)$ and $\phi_{uu} \le 0$ in $[0, \infty)^2$,
      $(u, v)$ is a weak $\phi$-supersolution of \eqref{prob:ks_damp} in the sense that
      \begin{align}\label{eq:a1:sol_concept:phi_supersol}
        &\pe  - \intninfom \phi(u, v) \varphi_t
              - \intom \phi(u_0, v_0) \varphi(\cdot, 0) \notag \\
        &\ge  - \intninfom \phi_{uu}(u, v) |\nabla u|^2 \varphi
              - \intninfom (\phi_{vv}(u, v) - u \phi_{uv}(u, v)) |\nabla v|^2 \varphi \notag \\
        &\pe  - \intninfom (2 \phi_{uv}(u, v) - u \phi_{uu}(u, v)) \varphi \nabla u \cdot \nabla v \notag \\
        &\pe  - \intninfom \phi_u(u, v) \nabla u \cdot \nabla \varphi
              - \intninfom (\phi_v(u, v) - u \phi_u(u, v)) \nabla v \cdot \nabla \varphi \notag \\
        &\pe  + \intninfom g(u) \phi_u(u, v) \varphi
              - \intninfom v \phi_v(u, v) \varphi
              + \intninfom u \phi_v(u, v) \varphi
      \end{align}
      for all $0 \le \varphi \in C_c^\infty(\Ombarinf)$,

    \item
      there is a null set $N \subset (0, \infty)$ such that
      \begin{align}\label{eq:a1:sol_concept:u_mass}
        \intom u(\cdot, t) \le \intom u_0 - \intnstom g(u)
        \qquad \text{for all $t \in (0, \infty) \setminus N$ and}
      \end{align}

    \item
      $v$ is a weak solution of the second subproblem in \eqref{prob:ks_damp} in the sense that
      \begin{align}\label{eq:a1:sol_concept:v_weak_sol}
          - \intninfom v \varphi_t
          - \intom v_0 \varphi(\cdot, 0)
        = - \intninfom \nabla v \cdot \nabla \varphi
          - \intninfom v \varphi
          + \intninfom u \varphi
      \end{align}
      for all $\varphi \in C_c^\infty(\Ombarinf)$.
  \end{itemize}
\end{definition}

The solution concept defined in Definition~\ref{def:a1:sol_concept} is consistent with that of classical solutions in the following sense.
\begin{lemma}\label{lm:a1:sol_concept_sensible}
  Let $\Omega \subset \R^n$, $n \in \N$, be a smooth, bounded domain, $u_0, v_0 \in \con0$ nonnegative and $g \in C^0([0, \infty))$.
  Suppose $(u, v) \in C^0(\Ombarinf) \times C^{2, 1}(\Ombar \times (0, \infty))$ is a generalized solution of \eqref{prob:ks_damp} in the sense of Definition~\ref{def:a1:sol_concept}.
  Then $(u, v)$ is also a classical solution of \eqref{prob:ks_damp}.
\end{lemma}
\begin{proof}
  Due to \eqref{eq:a1:sol_concept:v_weak_sol}, a standard reasoning shows that $v$ is a classical solution of the second subproblem in \eqref{prob:ks_damp}.
  For fixed $T > 0$, $M \defs \max\{\|u\|_{C^0(\OmbarT)}, \|v\|_{C^0(\OmbarT)}\}$ is finite.
  We choose any $\phi \in C^\infty([0, \infty)^2)$ with $\phi_{uu} \le 0$ in $\Ombarinf$ and
  \begin{align*}
    \phi(s_1, s_2) =
    \begin{cases}
      s_1,   & \text{if } s_1, s_2 \le M, \\
      M+1, & \text{if } s_1, s_2 \ge M+1
    \end{cases}
    \qquad \text{for $s_1, s_2 \ge 0$}.
  \end{align*}
  Then $\supp D\phi \subset [0, M+1]^2$ so that $\phi$ is admissible in \eqref{eq:a1:sol_concept:phi_supersol} and we obtain
  \begin{align}\label{eq:a1:sol_concept_consistent:u_supersol}
          - \intninfom u \varphi_t
          - \intom u_0 \varphi(\cdot, 0)
    &\ge  - \intninfom \nabla u \cdot \nabla \varphi
          + \intninfom u \nabla v \cdot \nabla \varphi
          + \intninfom g(u) \varphi
  \end{align}
  for all $0 \le \varphi \in C_c^\infty(\Ombar \times [0, T))$.
  From now on, we can argue as in \cite[Lemma~2.1]{WinklerLargedataGlobalGeneralized2015}:
  By choosing $\varphi$ supported away from the tempo-spatial boundary, near the spatial boundary and near the temporal origin in \eqref{eq:a1:sol_concept_consistent:u_supersol}, respectively,
  and employing density arguments, we see that
  \begin{align*}
    \begin{cases}
      u_t \ge \Delta u - \nabla \cdot (u \nabla v) + g & \text{in $\Omega \times (0, T)$}, \\
      \nu \cdot (\nabla u - u \nabla v) \ge 0          & \text{in $\partial \Omega \times (0, T)$}, \\
      u(\cdot, 0) \ge u_0                              & \text{in $\Omega$}
    \end{cases}
  \end{align*}
  holds.
  When combined with \eqref{eq:a1:sol_concept:u_mass}, this rapidly shows that also $u$ is a classical solution of the respective subproblem in \eqref{prob:ks_damp}.
\end{proof}

\begin{remark}\label{rm:a1:sol_concept_discussion}
  The solution concept proposed in \cite{WinklerSolutionsParabolicKellertoappear}, where also global generalized solvability of \eqref{prob:ks_damp} is considered,
  differs from the one introduced in Definition~\ref{def:a1:sol_concept} in that there \eqref{eq:a1:sol_concept:phi_supersol}
  is not required to hold for all $\phi \in C^\infty([0, \infty))$ with $D \phi \in C_c^\infty([0, \infty)^2)$ and $\phi_{uu} \le 0$ in $[0, \infty)^2$,
  but merely for the choice $\phi(u, v) = \tilde \phi(u, v) \defs -(u+1)^{-p} \ure^{-\kappa v}$, $u, v \ge 0$, for certain $p, \kappa > 0$.
  Requiring \eqref{eq:a1:sol_concept:phi_supersol} for a range of functions is arguably both more natural and more general than focusing on a single such function.
  (Let us remark that while $\tilde \phi$ is not directly admissible in \eqref{eq:a1:sol_concept:phi_supersol} as $\supp \tilde \phi$ is not compact,
  for sufficiently regular $(u, v)$ the class of admissible $\phi$ can be widened by approximation arguments.)
  A concrete example where the solution concept introduced in Definition~\ref{def:a1:sol_concept} seems to be advantageous is the verification of consistency with classical solvability,
  which is relatively straightforward for our concept (see Lemma~\ref{lm:a1:sol_concept_sensible})
  and more involved if only specific functions $\phi$ of both $u$ and $v$ are admissible
  (cf.\ \cite[Lemma~2.5]{LankeitWinklerGeneralizedSolutionConcept2017} and \cite[Remark after Definition~2.2]{WinklerSolutionsParabolicKellertoappear}).

  On the other hand, the advantage of choosing the specific function $\tilde \phi$ is that the first three terms on the right-hand side in \eqref{eq:a1:sol_concept:phi_supersol} can then be written as
  \begin{align}\label{a1:winkler_sol_concept_psi}
    \intninfom |\nabla \psi_1(u, v) + \psi_2(u, v) \nabla v|^2 \varphi + \intninfom |\psi_3(u, v)| |\nabla v|^2 \varphi
  \end{align}
  for certain $\psi_1, \psi_2, \psi_3$ (cf.\ \cite[(3.13)]{WinklerSolutionsParabolicKellertoappear}).
  Now assuming appropriate a priori estimates for solutions to approximative problems implying weak convergences of weighted gradients,
  one can make use of the weak lower semicontinuity of the norm to relate the terms in \eqref{a1:winkler_sol_concept_psi} to their counterparts in the approximative problems.
  To that end, the signs of the terms in \eqref{a1:winkler_sol_concept_psi} are evidently important and hence this approach fails for general choices of $\phi$.

  As we will see in Lemma~\ref{lm:a1:eps_sea_0} below, we will be able to apply Theorem~\ref{th:strong_grad_conv} to a sequence of approximative problems
  and thus obtain strong convergence of weighted gradients of the second solution component.
  Therefore, there is no need to cleverly combine the first three terms on the right-hand side in \eqref{eq:a1:sol_concept:phi_supersol}:
  The expressions involving $\nabla v$ converge and the sign condition on $\phi_{uu}$ allows us to favorably use the weak lower semicontinuity just for the first term (see Lemma~\ref{lm:a1:u_v_phi_supersol} below).
\end{remark}

\subsection{Existence of a global generalized solution: proof of Theorem~\ref{th:a1}}
In this subsection, we prove Theorem~\ref{th:a1}.
To that end, we fix a smooth, bounded domain $\Omega \subset \R^n$, $n \ge 2$, $g \in C^1([0, \infty))$ complying with \eqref{eq:a1:cond_g} and $u_0, v_0 \in \leb1$.
As in \cite{WinklerSolutionsParabolicKellertoappear}, we aim to obtain a solution to \eqref{prob:ks_damp} as a limit of solutions to regularized problems
and thus fix $(\une, \vne)_{\eps \in (0, 1)} \subset \con0 \times \sob1\infty$ with
\begin{align*}
  (\une, \vne) &\ra (u_0, v_0) \qquad \text{in $\leb1$ and a.e.\ in $\Omega$ as $\eps \sea 0$.}
\end{align*}
and, for each $\eps \in (0, 1)$, functions
\begin{align}
  0 \le \ue &\in C^0(\Ombarinf) \times C^{2, 1}(\Ombar \times (0, \infty)), \label{eq:a2:reg_ue}\\
  0 \le \ve &\in \bigcap_{q > n} C^0([0, \infty); \sob1q) \times C^{2, 1}(\Ombar \times (0, \infty)) \label{eq:a2:reg_ve}
\end{align}
solving
\begin{align}\label{prob:ks_damp_eps}
  \begin{cases}
    \uet = \Delta \ue - \nabla \cdot (\ue \nabla \ve) + g(\ue) & \text{in $\Omega \times (0, \infty)$}, \\
    \vet = \Delta \ve - \ve + \frac{\ue}{1 + \eps \ue}         & \text{in $\Omega \times (0, \infty)$}, \\
    \partial_\nu \ue = \partial_\nu \ve = 0                    & \text{on $\partial \Omega \times (0, \infty)$}, \\
    \ue(\cdot, 0) = \une, \ve(\cdot, 0) = \vne                 & \text{in $\Omega$}
  \end{cases}
\end{align}
classically.
(The existence of these functions has been established in \cite[Lemma~2.1]{WinklerSolutionsParabolicKellertoappear}.)

The crucial assumption \eqref{eq:a1:cond_g} allows us to derive the following a priori estimate.
\begin{lemma}\label{lm:a1:u_uniform_int}
  Let $T \gt 0$. Then there exists $\newgc{g_l1} > 0$ such that
  \begin{align}\label{eq:a1:u_uniform_int:g_l1}
    \intntom |g(\ue)| \le \gc{g_l1}
    \qquad \text{for all $\eps \in (0, 1)$}
  \end{align}
  and
  \begin{align}\label{eq:a1:u_uniform_int:u_uniform_int}
    (\ue)_{\eps \in (0, 1)} \quad \text{is uniformly integrable in $\Omega \times (0, T)$}.
  \end{align}
\end{lemma}
\begin{proof}
  As $(\une)_{\eps \in (0, 1)}$ is bounded in $\leb1$ and $\sup \{\,s \ge 0 : g(s) \ge 0\,\}$ is well-defined and finite by \eqref{eq:a1:cond_g},
  integrating the first equation in \eqref{prob:ks_damp} eventually reveals that \eqref{eq:a1:u_uniform_int:g_l1} holds for some $\gc{g_l1} > 0$;
  we refer to \cite[Lemma~4.1]{WinklerSolutionsParabolicKellertoappear} for details.
  Thanks to \eqref{eq:a1:cond_g}, the de la Vall\'ee Poussin theorem then allows us
  to infer \eqref{eq:a1:u_uniform_int:u_uniform_int} from \eqref{eq:a1:u_uniform_int:g_l1}.
\end{proof}

Lemma~\ref{lm:a1:u_uniform_int} already makes Theorem~\ref{th:strong_grad_conv} applicable, allowing us to obtain a solution candidate $(u, v)$ in the following manner.
\begin{lemma}\label{lm:a1:eps_sea_0}
  There exist nonnegative $u, v \in L_{\loc}^1(\Ombarinf)$ with $\nabla v \in L_{\loc}^1(\Ombarinf)$ and $\mathds 1_{\{v \le k\}} \nabla v \in L_{\loc}^2(\Ombarinf)$ and a null sequence $(\eps_j)_{j \in \N} \subset (0, 1)$ such that
  \begin{alignat}{2}
    \ue &\rh u
    &&\qquad \text{in $L_{\loc}^1(\Ombarinf)$}, \label{eq:a1:eps_sea_0:u_weakly_l1}\\
    \ve &\ra v
    &&\qquad \text{in $L_{\loc}^1(\Ombarinf)$ and a.e.\ in $\Omega \times (0, \infty)$}, \label{eq:a1:eps_sea_0:v_l1}\\
    \nabla \ve &\ra \nabla v
    &&\qquad \text{in $L_{\loc}^1(\Ombarinf)$ and a.e.\ in $\Omega \times (0, \infty)$}, \label{eq:a1:eps_sea_0:nabla_v_l1}\\
    \mathds 1_{\{\ve \le k\}} \nabla \ve &\ra \mathds 1_{\{v \le k\}} \nabla v
    &&\qquad \text{in $L_{\loc}^2(\Ombarinf)$} \label{eq:a1:eps_sea_0:nabla_tk_v_l2}
  \end{alignat}
  as $\eps = \eps_j \sea 0$.
  In addition, \eqref{eq:a1:sol_concept:v_weak_sol} holds for all $\varphi \in C_c^\infty(\Ombarinf)$.
\end{lemma}
\begin{proof}
  By \eqref{eq:a1:u_uniform_int:u_uniform_int}, the Dunford--Pettis theorem and a diagonalization argument,
  there exists $u \in L_{\loc}^1(\Ombarinf)$ and a null sequence $(\eps_j)_{j \in \N} \subset (0, 1)$ such that \eqref{eq:a1:eps_sea_0:u_weakly_l1} holds.
  Therefore, we may apply Theorem~\ref{th:strong_grad_conv} (with $\fe = \ue$),
  which asserts that \eqref{eq:a1:sol_concept:v_weak_sol} and \eqref{eq:a1:eps_sea_0:v_l1}--\eqref{eq:a1:eps_sea_0:nabla_tk_v_l2} hold for some function $v$ of the desired regularity,
  provided we switch to a suitable subsequence of $(\eps_j)_{j \in \N}$.
  Due to the (norm) closure of the convex hull of $(\uej)_{j \in \N}$ being weakly closed, \eqref{eq:a2:reg_ue} implies $u \ge 0$,
  while $v \ge 0$ follows from \eqref{eq:a1:eps_sea_0:v_l1} and \eqref{eq:a2:reg_ve}.
\end{proof}

In order to show that the pair $(u, v)$ constructed in Lemma~\ref{lm:a1:eps_sea_0} is a weak $\phi$-supersolution for appropriate $\phi$
in the sense of Definition~\ref{def:a1:sol_concept},
we need to obtain stronger convergence properties of the first solution component than those asserted by Lemma~\ref{lm:a1:eps_sea_0}
and thus derive appropriate a priori estimates also including the space and time derivatives of $\ue$.

In contrast to other parts of the proof,
Theorem~\ref{th:strong_grad_conv} evidently cannot be used to shorten or simplify the derivation of these estimates.
Instead, we proceed by citing a key result from \cite{WinklerSolutionsParabolicKellertoappear}.
\begin{lemma}\label{lm:a1:nabla_u_v}
  There exists $p > 0$ and $\kappa > 0$ such that for all $T \gt 0$, we can find $\newgc{nabla_u_v} > 0$ with the property that 
  \begin{align}\label{eq:a1:nabla_u_v:statement}
    \intntom \left| \nabla \left( (\ue + 1)^{-p} \ure^{-\kappa \ve} \right) \right|^2 \le \gc{nabla_u_v}
    \qquad \text{for all $\eps \in (0, 1)$}.
  \end{align}
\end{lemma}
\begin{proof}
  This is contained in \cite[Lemma~6.3]{WinklerSolutionsParabolicKellertoappear}.
\end{proof}

Next, we make use of Lemma~\ref{lm:a1:nabla_u_v} to indeed obtain a priori estimates for the gradient of the first solution component.
\begin{lemma}\label{lm:a1:nabla_u_bdd}
  Let $T \gt 0$ and $h > 0$. Then there exists $\newgc{nabla_u} > 0$ such that
  \begin{align*}
    \intntom \mathds 1_{\{\ue \le h, \ve \le h\}} |\nabla \ue|^2 \le \gc{nabla_u}
    \qquad \text{for all $\eps \in (0, 1)$}.
  \end{align*}
\end{lemma}
\begin{proof}
  We fix $p, \kappa, \gc{nabla_u_v} > 0$ as given by Lemma~\ref{lm:a1:nabla_u_v}.
  Moreover, according to \eqref{eq:a1:eps_sea_0:nabla_tk_v_l2}, there is $\newlc{nabla_v_l2} > 0$ such that
  \begin{align*}
    \intntom \mathds 1_{\{\ve \le h\}} |\nabla \ve|^2 \le \lc{nabla_v_l2}
    \qquad \text{for all $\eps \in (0, 1)$.}
  \end{align*}
  By also making use of \eqref{eq:a1:nabla_u_v:statement}, we then obtain
  \begin{align*}
    &\pe  \|\mathds 1_{\{\ue \le h, \ve \le h\}} \nabla \ue\|_{L^2(\OmT)} \\
    &\le  \frac{(h+1)^{p+1} \ure^{\kappa h}}{p} \|p (\ue + 1)^{-p-1} \ure^{-\kappa \ve} \nabla \ue\|_{L^2(\{\ue, \ve \le h\})} \\
    &\le  \frac{(h+1)^{p+1} \ure^{\kappa h}}{p} \left( \|\nabla( (\ue+1)^{-p} \ure^{-\kappa \ve})\|_{L^2(\OmT)} + \|\kappa (\ue+1)^{-p} \ure^{-\kappa \ve} \nabla \ve\|_{L^2(\{\ve \le h\})} \right) \\
    &\le  \frac{(h+1)^{p+1} \ure^{\kappa h}}{p} \left( \gc{nabla_u_v}^{1/2} + \kappa \lc{nabla_v_l2}^{1/2}\right)
    \sfed \newlc{C}
    \qquad \text{for all $\eps \in (0, 1)$},
  \end{align*}
  implying the statement for $\gc{nabla_u} \defs \lc{C}^2$.
\end{proof}

Before proceeding to also infer estimates for the time derivatives from Lemma~\ref{lm:a1:nabla_u_bdd}
and the estimates implicitly contained in \eqref{eq:a1:eps_sea_0:u_weakly_l1}--\eqref{eq:a1:eps_sea_0:nabla_tk_v_l2},
we state the following identity, which will also be made use of in the proof of Lemma~\ref{lm:a1:u_v_phi_supersol} below.
\begin{lemma}\label{lm:a1:ue_ve_phi_sol}
  Let $\varphi \in \con\infty$ and
  $\phi \in C^\infty([0, \infty)^2)$.
  Then
  \begin{align}\label{eq:a1:ue_ve_phi_sol:eq}
            \intom [\phi(\ue, \ve)]_t \varphi
    &=    - \intom \phi_{uu}(\ue, \ve) |\nabla \ue|^2 \varphi
          - \intom (\phi_{vv}(\ue, \ve) - \ue \phi_{uv}(\ue, \ve)) |\nabla \ve|^2 \varphi \notag \\
    &\pe  - \intom (2 \phi_{uv}(\ue, \ve) - \ue \phi_{uu}(\ue, \ve)) \varphi \nabla \ue \cdot \nabla \ve \notag \\
    &\pe  - \intom \phi_u(\ue, \ve) \nabla \ue \cdot \nabla \varphi
          - \intom (\phi_v(\ue, \ve) - \ue \phi_u(\ue, \ve)) \nabla \ve \cdot \nabla \varphi \notag \\
    &\pe  + \intom g(\ue) \phi_u(\ue, \ve) \varphi
          - \intom \ve \phi_v(\ue, \ve) \varphi
          + \intom \frac{\ue}{1+\eps\ue} \phi_v(\ue, \ve) \varphi
  \end{align}
  in $(0, \infty)$ for all $\eps \in (0, 1)$.
\end{lemma}
\begin{proof}
  Testing the first and second equations in \eqref{prob:ks_damp_eps} with $\phi_u(\ue, \ve) \varphi$ and $\phi_v(\ue, \ve) \varphi$, respectively,
  and integrating by parts shows that \eqref{eq:a1:ue_ve_phi_sol:eq} holds in $(0, \infty)$ for all $\eps \in (0, 1)$.
\end{proof}

\begin{lemma}\label{lm:a3:phi_t}
 Let $T \gt 0$ and $\phi \in C_c^\infty([0, \infty))$. 
 Then there exists $\newgc{phi_t} > 0$ such that
 \begin{align*}
   \|(\phi(\ue, \ve))_t\|_{L^1((0, T); \dual{\sob{n}{2}})} \le \gc{phi_t}
   \qquad \text{for all $\eps \in (0, 1)$}.
 \end{align*}
\end{lemma}
\begin{proof}
  Since $\supp \phi$ is compact, there exist $h > 0$ and $\newlc{phi_bdd} > 0$ such that $\supp \phi \subset [0, h]^2$
  and
  \begin{align*}
    \max\{\phi_u, \phi_v, \phi_{uu}, \phi_{uv}, \phi_{vv}\}(s_1, s_2) \le \lc{phi_bdd} \qquad \text{for all $s_1, s_2 \ge 0$}.
  \end{align*}
  According to Lemma~\ref{lm:a1:ue_ve_phi_sol} and Hölder's inequality, we can estimate
  \begin{align*}
          \left| \intom [\phi(\ue, \ve)]_t \varphi \right|
    &\le  \left( (3+h)\lc{phi_bdd} \int_{\{\ue, \ve \le h\}} |\nabla \ue|^2 + (3+2h)\lc{phi_bdd} \int_{\{\ve \le h\}} |\nabla \ve|^2 \right) \|\varphi\|_{\leb\infty} \\
    &\pe  + \left( \lc{phi_bdd} \left( 1+ \int_{\{\ue, \ve \le h\}} |\nabla \ue|^2 \right) + (1+h)\lc{phi_bdd} \left( 1 + \int_{\{\ve \le h\}} |\nabla \ve|^2 \right) \right) \|\nabla \varphi\|_{\leb2} \\
    &\pe  + \lc{phi_bdd}\left( \intom |g(\ue)| + \intom \ve + \intom \ue \right) \|\varphi\|_{\leb\infty}
  \end{align*}
  in $(0, T)$ for all $\eps \in (0, 1)$ and all $\varphi \in \con\infty$.
  Since $\sob n2 \embed \leb\infty$ and $\con\infty$ is dense in $\sob n2$,
  the statement follows upon an integration in time and recalling Lemma~\ref{lm:a1:nabla_u_bdd}, \eqref{eq:a1:eps_sea_0:nabla_tk_v_l2}, \eqref{eq:a1:u_uniform_int:g_l1}, \eqref{eq:a1:eps_sea_0:v_l1} and \eqref{eq:a1:eps_sea_0:u_weakly_l1}.
\end{proof}

With these preparations at hand, we are now able to prove stronger convergence properties going considerably beyond \eqref{eq:a1:eps_sea_0:u_weakly_l1}.
\begin{lemma}\label{lm:a1:u_eps_sea_0}
  Let $u$ and $(\eps_j)_{j \in \N}$ be as given by Lemma~\ref{lm:a1:eps_sea_0}.
  Then there exists a subsequence of $(\eps_j)_{j \in \N}$, which we do not relabel, such that
  \begin{alignat}{2}
    \ue &\ra u
    &&\qquad \text{in $L_{\loc}^1(\Ombarinf)$ and a.e.\ in $\Omega \times (0, \infty)$}, \label{eq:a1:u_eps_sea_0:u_pw_l1}\\
    \ue(\cdot, t) &\ra u(\cdot, t)
    &&\qquad \text{in $\leb1$ for a.e.\ $t \in (0, \infty)$}, \label{eq:a1:u_eps_sea_0:u_l1_t}\\
    \mathds 1_{\{\ue \le h\}} \nabla \ue &\rh \mathds 1_{\{u \le h\}} \nabla u
    &&\qquad \text{in $L_{\loc}^2(\Ombarinf)$ for $h \in \N$} \label{eq:a1:u_eps_sea_0:nabla_u}
  \end{alignat}
  as $\eps = \eps_j \sea 0$.
\end{lemma}
\begin{proof}
  For each $T > 0$ and $\phi \in C_c^\infty([0, \infty)^2)$,
  the sequence $(\phi(\uej, \vej))_{j \in \N}$ is bounded in $L^2((0, T); \sob12)$ by Lemma~\ref{lm:a1:nabla_u_bdd} and \eqref{eq:a1:eps_sea_0:nabla_v_l1},
  and the corresponding time derivatives are bounded in $L^2((0, T); \dual{\sob n2})$ by Lemma~\ref{lm:a3:phi_t}.
  Therefore, the Aubin--Lions lemma asserts the existence of a subsequence of $(\phi(\uej, \vej))_{j \in \N}$ converging in $L^2(\OmT)$,
  from which we can extract another subsequence converging a.e.\ in $\Omega \times (0, T)$.
  By applying this to functions $\phi$ with $\phi(s_1, s_2) = s_1$ for all $s_1, s_2 \le M$ and increasing $M \in \N$ and making use of a diagonalization argument,
  we obtain a subsequence of $(\eps_j)_{j \in \N}$, not relabeled, such that $\uej \ra z$ a.e.\ in $\Omega \times (0, \infty))$ as $j \ra \infty$ for some $z \colon \Omega \times (0, \infty) \ra \R$.
  Thanks to \eqref{eq:a1:u_uniform_int:u_uniform_int}, Vitali's theorem further asserts $z \in L_{\loc}^1(\Ombarinf)$ and that this convergence also takes place in $L_{\loc}^1(\Ombarinf)$.
  As weak limits are unique, \eqref{eq:a1:eps_sea_0:u_weakly_l1} implies $z = u$ and thus \eqref{eq:a1:u_eps_sea_0:u_pw_l1}.
  Upon switching to a further subsequence, this entails \eqref{eq:a1:u_eps_sea_0:u_l1_t}.

  In order to finally verify \eqref{eq:a1:u_eps_sea_0:nabla_u}, we set $T_h(s) \defs \min\{s, h\}$ for $s \ge 0$ and $h \in \N$.
  By Lemma~\ref{lm:a1:nabla_u_bdd}, \eqref{eq:a1:u_eps_sea_0:u_pw_l1}, continuity of $T_h$ and a diagonalization argument,
  we obtain a final subsequence of $(\eps_j)_{j \in \N}$, again not relabeled, such that
  $\nabla(T_h \uej) \rh \nabla(T_h u)$ in $L_{\loc}^1(\Ombarinf)$ as $j \ra \infty$ for all $h \in \N$;
  that is, \eqref{eq:a1:u_eps_sea_0:nabla_u} holds.
\end{proof}

Lemma~\ref{lm:a1:eps_sea_0} and Lemma~\ref{lm:a1:u_eps_sea_0} now allow us to conclude that the pair $(u, v)$ obtained in Lemma~\ref{lm:a1:eps_sea_0}
is indeed a $\phi$-supersolution (for appropriate $\phi$) in the sense of Definition~\ref{def:a1:sol_concept}.
\begin{lemma}\label{lm:a1:u_v_phi_supersol}
  Let $(u, v)$ be as constructed in Lemma~\ref{lm:a1:eps_sea_0} 
  and $0 \le \varphi \in C_c^\infty(\Ombarinf)$ and $\phi \in C^\infty([0, \infty)^2)$ with $D\phi \in C_c^\infty([0, \infty)^2)$ and $\phi_{uu} \le 0$ in $[0, \infty)^2$.
  Then \eqref{eq:a1:sol_concept:phi_supersol} is fulfilled and there is a null set $N$ such that \eqref{eq:a1:sol_concept:u_mass} holds.
\end{lemma}
\begin{proof}
  Integrating \eqref{eq:a1:ue_ve_phi_sol:eq} shows that \eqref{eq:a1:sol_concept:phi_supersol} (with equality) would follow
  if we could pass to the limit in each term in (an integrated version of) \eqref{eq:a1:ue_ve_phi_sol:eq}.
  For all but the first term on the right-hand side therein, the convergences asserted by Lemma~\ref{lm:a1:eps_sea_0} and Lemma~\ref{lm:a1:u_eps_sea_0} indeed allow us to do that.
  For instance, if $(\eps_j)_{j \in \N}$ is as in Lemma~\ref{lm:a1:u_eps_sea_0} and $h \in \N$ is so large that $\supp D\varphi \subset [0, h]^2$, then
  \begin{align*}
    \mathds 1_{\{\vej \le h\}} \nabla \vej \ra \mathds 1_{\{v \le h\}} \nabla v
    \quad \text{and} \quad
    \mathds 1_{\{\uej, \vej \le h\}} \nabla \uej \rh \mathds 1_{\{u, v \le h\}} \nabla u
    \qquad \text{in $L_{\loc}^2(\Ombarinf)$}
  \end{align*}
  as $j \ra \infty$ by \eqref{eq:a1:eps_sea_0:nabla_v_l1} and \eqref{eq:a1:u_eps_sea_0:nabla_u},
  which due to boundedness of $(\uej \phi_{uu}(\uej,\vej) \varphi)_{j \in \N}$ implies
  \begin{align}\label{eq:a1:u_v_phi_supersol:conv_nabla_ue_nabla_ve}
    \lim_{j \ra \infty} \intninfom \uej \phi_{uu}(\uej,\vej) \varphi \nabla \uej \cdot \nabla \vej
    = \intninfom u \phi_{uu}(u, v) \varphi \nabla u \cdot \nabla v.
  \end{align}
  Regarding the first term on the right-hand side in \eqref{eq:a1:ue_ve_phi_sol:eq}, we make use of \eqref{eq:a1:u_eps_sea_0:u_pw_l1}, \eqref{eq:a1:eps_sea_0:v_l1}, \eqref{eq:a1:u_eps_sea_0:nabla_u},
  nonpositivity of $\phi_{uu}$ and the weak lower semicontinuity of the norm to obtain
  \begin{align*}
        \liminf_{j \ra \infty} \intninfom (-\phi_{uu}(\uej, \vej)) |\nabla \uej|^2 \varphi
    \ge \intninfom (-\phi_{uu}(u, v)) |\nabla u|^2 \varphi.
  \end{align*}
  In combination, this shows \eqref{eq:a1:sol_concept:phi_supersol}.

  Finally, since integrating the first equation in \eqref{prob:ks_damp_eps} implies that an analogue of \eqref{eq:a1:sol_concept:u_mass} holds (with equality) for $\eps \in (0, 1)$,
  a consequence of \eqref{eq:a1:u_eps_sea_0:u_l1_t}, \eqref{eq:a1:u_eps_sea_0:u_pw_l1} and Fatou's lemma is \eqref{eq:a1:sol_concept:u_mass}.
\end{proof}
Let us emphasize that \eqref{eq:a1:eps_sea_0:nabla_tk_v_l2}, which goes back to one of the main statements of Theorem~\ref{th:strong_grad_conv},
played a key role in the proof of Lemma~\ref{lm:a1:u_v_phi_supersol}.
That is, if we had only guaranteed weak convergence of weighted gradients of the second solution component instead,
we would not have been able to conclude \eqref{eq:a1:u_v_phi_supersol:conv_nabla_ue_nabla_ve}, for instance.

We have thereby shown that the pair $(u, v)$ constructed in Lemma~\ref{lm:a1:eps_sea_0} is a global generalized solution of \eqref{prob:ks_damp}.
\begin{proof}[Proof of Theorem~\ref{th:a1}]
  Let $(u, v)$ be as given by Lemma~\ref{lm:a1:eps_sea_0}.
  That \eqref{eq:a1:sol_concept:phi_supersol} and \eqref{eq:a1:sol_concept:u_mass} as well as \eqref{eq:a1:sol_concept:v_weak_sol} hold
  has been shown in Lemma~\ref{lm:a1:eps_sea_0} and Lemma~\ref{lm:a1:u_v_phi_supersol}, respectively,
  and the desired regularity of $(u, v)$ is included in Lemma~\ref{lm:a1:eps_sea_0} and Lemma~\ref{lm:a1:u_eps_sea_0}.
  Thus, $(u, v)$ is indeed a global generalized solution of \eqref{prob:ks_damp} in the sense of Definition~\ref{def:a1:sol_concept}.
\end{proof}

\section{Application II: Keller--Segel systems with logarithmic sensitivity}\label{sec:a2}
In the present section, we show that Theorem~\ref{th:strong_grad_conv} can also form a key ingredient in proving a global existence result for the Keller--Segel system with logarithmic sensitivity,
\begin{align}\label{prob:ks_log_sens}
  \begin{cases}
    u_t = \Delta u - \chi \nabla \cdot \big( \frac uv \nabla v \big) & \text{in $\Omega \times (0, \infty)$}, \\
    v_t = \Delta v - v + u                                           & \text{in $\Omega \times (0, \infty)$}, \\
    \partial_\nu u = \partial_\nu v = 0                              & \text{on $\partial \Omega \times (0, \infty)$}, \\
    u(\cdot, 0) = u_0, v(\cdot, 0) = v_0                             & \text{in $\Omega$},
  \end{cases}
\end{align}
where $\Omega \subset \R^n$ is a smooth, bounded domain and $\chi > 0$, $u_0 \ge 0$ and $v_0 > 0$ are given.

For a brief introduction to chemotaxis systems, we refer to the beginning of Section~\ref{sec:a1} (and for a more thorough one for instance to the survey \cite{BellomoEtAlMathematicalTheoryKeller2015}).
Here, let us just note that as a key difference compared to the system studied in the previous section,
the organism with density $u$ is now attracted by higher gradients of the logarithm of the signal substance with concentration $v$ rather than by higher gradients of the substance itself.
That such logarithmic chemotactic sensitives model perception in accordance with the Weber--Fechner law and are thus sensible modelling choices in certain cases
has already been argued in an early work by Keller and Segel \cite{KellerSegelTravelingBandsChemotactic1971}.
For further biological motivation, see the survey~\cite{HillenPainterUserGuidePDE2009}.

Accordingly, questions of global existence of solutions to \eqref{prob:ks_log_sens} and variants thereof have received quite some attention,
although due to the apparent lack of suitable energy functionals
the answers are yet not as conclusive as for the classical Keller--Segel system (\eqref{prob:ks_log_sens} with $-\chi \nabla \cdot(\frac uv \nabla v)$ replaced by $-\chi \nabla \cdot(u \nabla v)$).
 
Global classical solutions of \eqref{prob:ks_log_sens} have been constructed under various conditions on $\chi$ depending on the space dimension $n$:
in \cite{TaoEtAlLargetimeBehaviorParabolicparabolic2013} for $n = 1$ and $\chi < \infty$,
in \cite{AhnEtAlGlobalWellposednessLogarithmic2021} for ($2 \le n \le 8$ and) $\chi \le \frac{4}{n}$
and in \cite{WinklerGlobalSolutionsFully2011} for ($n \ge 9$ and) $\chi < (\frac2n)^\frac12$.
For related results, in particular for a history of the development of these requirements,
see for instance \cite[introduction]{AhnEtAlGlobalWellposednessLogarithmic2021}.
While to the best of our knowledge, these conditions on $\chi$ are the best known so far,
the question whether any of these is optimal, that is, if global classical solution may fail to exist for sufficiently large $\chi$, is entirely open.
The only result for the fully parabolic system in this direction seems to consist of a finding stating
that solutions may grow arbitrary large if a certain parameter is chosen sufficiently small~\cite{WinklerUnlimitedGrowthLogarithmic2022}.
The picture is slightly more complete for the parabolic--elliptic simplification of \eqref{prob:ks_log_sens} considered in $n$-dimensional balls.
Corresponding radially symmetric classical solutions always exist globally in time if $\chi < \frac{n}{(n-2)_+}$
and blow up in finite time if $n \ge 3$, $\chi > \frac{2n}{n-2}$ and the initial second moment is sufficiently small~\cite{NagaiSenbaGlobalExistenceBlowup1998}.

Similarly as for the system considered in Section~\ref{sec:a1},
the challenges for obtaining global classical solutions outlined above make it worthwhile to also aim for global existence results with respect to weaker solution concepts.
In some sense, the knowledge in this direction is already quite satisfactory:
Global generalized solutions are known to exist for arbitrary $\chi > 0$
both in radially symmetric settings \cite{StinnerWinklerGlobalWeakSolutions2011} and for general domains \cite{ZhigunGeneralisedSupersolutionsMass2018}.
However, the solution concepts proposed in these findings do not guarantee conservation of mass of the first solution component,
which may be seen as an important feature of \eqref{prob:ks_log_sens} (which is evidently fulfilled by classical solutions).
If one restricts the range of admissible $\chi$, this issue can be overcome:
For $\chi < (\frac{n+2}{3n-4})^\frac12$, global weak solutions have been constructed in \cite{WinklerGlobalSolutionsFully2011},
while in \cite{LankeitWinklerGeneralizedSolutionConcept2017} global generalized solutions have been obtained for
\begin{align}\label{eq:a2:cond_chi}
  \chi \lt
  \begin{cases}
    \infty, & n = 2, \\
    \sqrt8, & n = 3, \\
    \frac{n}{n-2}, & n \ge 4.
  \end{cases}
\end{align}
In both cases, the mass of the first solution component is conserved.

The aim of the present section is to show how Theorem~\ref{th:strong_grad_conv} can be used to shorten the proof in \cite{LankeitWinklerGeneralizedSolutionConcept2017}
and arrive at a more natural solution concept (cf.\ the discussion in Remark~\ref{rm:a2:sol_concept_discussion} below).
The main result of this section is
\begin{theorem}\label{th:a2}
  Let $\Omega \subset \R^n$, $n \ge 2$, be a smooth, bounded domain and suppose that $\chi > 0$ fulfills \eqref{eq:a2:cond_chi}.
  Then given any $0 \not\equiv u_0 \in \con0$ and $v_0 \in \sob1\infty$ with $u_0 \ge 0$ and $v_0 > 0$ in $\Ombar$,
  there exists a global generalized solution $(u, v)$ of \eqref{prob:ks_log_sens_eps} in the sense of Definition~\ref{def:a2:sol_concept} below,
  which additionally satisfies
  \begin{align}\label{eq:a2:mass_conservation}
    \intom u(\cdot, t) = \intom u_0
    \qquad \text{for a.e.\ $t > 0$}.
  \end{align}
\end{theorem}

\subsection{Solution concept}\label{sec:a2:sol_concept}
We propose the following notion of generalized solvability for the system \eqref{prob:ks_log_sens}.
\begin{definition}\label{def:a2:sol_concept}
  Let $\Omega \subset \R^n$, $n \in \N$, be a smooth bounded domain, $\chi > 0$, $0 \le u_0 \in \leb1$ and $0 \le v_0 \in \leb1$.
  We call a tuple of nonnegative functions $(u, v) \in (L_{\loc}^1(\Ombarinf))^2$ with
   \begin{align*}
     \frac{\nabla u}{u+1} \in L_{\loc}^2(\Ombarinf), \quad
     \nabla v \in L_{\loc}^1(\Ombarinf)
     \quad \text{and} \quad
     \frac{\nabla v}{v} \in L_{\loc}^2(\Ombarinf)
   \end{align*}
  a global generalized solution of \eqref{prob:ks_log_sens}, if
  \begin{itemize}
    \item
      $u$ is a weak $\ln$-supersolution of the first subproblem in \eqref{prob:ks_log_sens} in the sense that
      \begin{align}\label{eq:a2:sol_concept:u_ln_supersol}
        &\pe  - \intninfom \ln(u+1) \varphi_t
              - \intom \ln(u_0+1) \varphi(\cdot, 0) \notag \\
        &\ge  \intninfom \frac{|\nabla u|^2}{(u+1)^2} \varphi
              - \chi \intninfom \frac{u \nabla u \cdot \nabla v}{(u+1)^2 v} \varphi \notag \\
        &\pe  - \intninfom \frac{\nabla u \cdot \nabla \varphi}{u+1} 
              + \chi \intninfom \frac{u \nabla v \cdot \nabla \varphi}{(u+1)v}
        \qquad \text{for all $0 \le \varphi \in C_c^\infty(\Ombarinf)$},
      \end{align}

    \item
      there is a null set $N \subset (0, \infty)$ such that
      \begin{align}\label{eq:a2:sol_concept:u_mass}
        \intom u(\cdot, t) \le \intom u_0
        \qquad \text{for all $t \in (0, \infty) \setminus N$ and}
      \end{align}

    \item
      $v$ is a weak solution of the second subproblem in \eqref{prob:ks_log_sens} in the sense that
      \begin{align}\label{eq:a2:sol_concept:v_weak_sol}
          - \intninfom v \varphi_t
          - \intom v_0 \varphi(\cdot, 0)
        = - \intninfom \nabla v \cdot \nabla \varphi
          - \intninfom v \varphi
          + \intninfom u \varphi
      \end{align}
      holds for all $\varphi \in C_c^\infty(\Ombarinf)$.
  \end{itemize}
\end{definition}

Consistency with classical solvability is of course a minimal requirement for a sensible solution concept,
which we have verified in Lemma~\ref{lm:a1:sol_concept_sensible} for the one proposed in Definition~\ref{def:a1:sol_concept} 
and which is also fulfilled for the solution concept introduced in Definition~\ref{def:a2:sol_concept}.
\begin{lemma}\label{lm:a2:sol_concept_consistent}
  Suppose $u, v \in C^0(\Ombar \times [0, \infty)) \cap C^{2, 1}(\Ombar \times (0, \infty))$ with $u \ge 0$ and $v \gt 0$ in $\Ombar \times [0, \infty)$ are such that
  $(u, v)$ is a generalized solution of \eqref{prob:ks_log_sens} in the sense of Definition~\ref{def:a2:sol_concept}.
  Then $(u, v)$ also solves \eqref{prob:ks_log_sens} classically.
\end{lemma}
\begin{proof}
  This can be proved as in \cite[Lemma~2.1]{WinklerLargedataGlobalGeneralized2015}.
\end{proof}

\begin{remark}\label{rm:a2:sol_concept_discussion}
  As already mentioned above, also in \cite{LankeitWinklerGeneralizedSolutionConcept2017} global solutions for \eqref{prob:ks_log_sens} are constructed,
  albeit with a different solution concept (cf.\ \cite[Definition~2.4]{LankeitWinklerGeneralizedSolutionConcept2017}).
  The main difference compared to Definition~\ref{def:a2:sol_concept} is that instead of \eqref{eq:a2:sol_concept:u_ln_supersol},
  $u^p v^q$ is required to be a weak supersolution of the corresponding problem for some $p, q \in (0, 1)$.
  In order to be able to show that this solution concept is consistent with classical solvability (that is, that an analogue of Lemma~\ref{lm:a2:sol_concept_consistent} holds),
  in particular regarding the boundary conditions for $u$,
  it is then additionally required that $u^p v^q$ is positive a.e.\ on $\partial \Omega \times (0, \infty)$.

  The solution concept introduced in Definition~\ref{def:a2:sol_concept} has two advantages.
  First, while \eqref{eq:a2:sol_concept:u_ln_supersol} requires that a nonlinear function of the solution is a supersolution to a certain equation,
  this function only depends on a single solution component.
  This is arguably more natural and has the practical advantage of leading to less convoluted expressions.
  Second, the previous point significantly simplifies the verification of the consistency with the concept of classical solutions.
  In particular, we do not need to ask for positivity  of $u$ and $v$, neither in the interior nor on the boundary of $\Omega$.
  Thus, testing procedures such as the one employed in \cite[Lemma~8.6]{LankeitWinklerGeneralizedSolutionConcept2017} verifying these requirements for certain solution candidates become superfluous.
\end{remark}

\subsection{Existence of a global generalized solution: proof of Theorem~\ref{th:a2}}
For the remainder of this section, we fix a smooth, bounded domain $\Omega \subset \R^n$, $n \ge 2$, $\chi > 0$ as well as
$0 \not\equiv u_0 \in \con0$ and $v_0 \in \sob1\infty$ with $u_0 \ge 0$ and $v_0 > 0$ in $\Ombar$.

\begin{lemma}\label{lm:a2:local_ex}
  Let $\eps \in (0, 1)$.
  Then there exists a global classical solution
  \begin{align*}
    (u, v) \in (C^0(\Ombarinf) \cap C^{2, 1}(\Ombar \times (0, \infty)))^2
  \end{align*}
  of
  \begin{align}\label{prob:ks_log_sens_eps}
    \begin{cases}
      \uet = \Delta \ue - \chi \nabla \cdot \big( \frac{\ue}{(1+\eps\ue)\ve} \nabla \ve) & \text{in $\Omega \times (0, \infty)$}, \\
      \vet = \Delta \ve - \ve + \ue                                                      & \text{in $\Omega \times (0, \infty)$}, \\
      \partial_\nu \ue = \partial_\nu \ve = 0                                            & \text{on $\partial \Omega \times (0, \infty)$}, \\
      \ue(\cdot, 0) = u_0, \ve(\cdot, 0) = v_0                                           & \text{in $\Omega$}.
    \end{cases}
  \end{align}
  Moreover, $\ue \ge 0$ and $\ve > 0$ in $\Ombar \times [0, \infty)$.
\end{lemma}
\begin{proof}
  This has been shown in \cite[Lemma~3.1]{LankeitWinklerGeneralizedSolutionConcept2017}.
\end{proof}
For $\eps \in (0, 1)$, we henceforth fix a solution $(\ue, \ve)$ of \eqref{prob:ks_log_sens_eps} given by Lemma~\ref{lm:a2:local_ex}.
As a first, yet very basic, $\eps$-independent estimate, we note that  $\ve$ can be estimated from below uniformly in $\eps$ on bounded time intervals.
\begin{lemma}\label{lm:a2:first_apriori}
  Let $T > 0$. Then there is $\newgc{v_inf} > 0$ such that
  \begin{align*}
    \ve \ge \gc{v_inf}
    \qquad \text{in $\Omega \times (0, T)$ for all $\eps \in (0, 1)$}.
  \end{align*}
\end{lemma}
\begin{proof}
  We fix $\eps \in (0, 1)$.
  Since $\ue \ge 0$ by Lemma~\ref{lm:a2:local_ex}, $\ul v_\eps \defs \min_{x \in \Omega} v_0(x) \ure^{-t}$ is a subsolution of the second subproblem in \eqref{prob:ks_log_sens_eps}.
  Therefore, the comparison principle asserts $\ve \ge \ul v_\eps \ge \min_{x \in \Omega} v_0(x) \ure^{-T} \sfed \gc{v_inf}$ in $\Omega \times (0, T)$.
  As $v_0$ is assumed to be continuous and positive in $\Ombar$, $\gc{v_inf}$ is positive.
\end{proof}

As in Section~\ref{sec:a1}, our aim is now to apply Theorem~\ref{th:strong_grad_conv} to $f_\eps \defs \ue$,
which due to the Dunford--Pettis theorem essentially requires uniform integrability of $(\ue)_{\eps \in (0, 1)}$ in $\Omega \times (0, T)$ for all $T > 0$.
That in turn is equivalent to bounding $\intntom \Phi_T(\ue)$ for some superlinearly growing $\Phi_T$ and all $T > 0$ by the de la Vall\'ee Poussin theorem.
Unlike the system \eqref{prob:ks_damp} considered in the previous section, however, \eqref{prob:ks_log_sens} lacks a convenient dampening term in the first equation,
making it more challenging to verify uniform integrability of $(\ue)_{\eps \in (0, 1)}$.

Fortunately, such a~priori estimates have already been derived in \cite{LankeitWinklerGeneralizedSolutionConcept2017}.
The main idea is to consider the time evolution of $\intom \ue^p \ve^q$,
a functional well adapted to the logarithmic taxis sensitivity in the first equation of \eqref{prob:ks_log_sens}.
For $p \in (0, \min\{1, \frac{1}{\chi^2}\})$ and certain $q > 0$, this then eventually leads to $\eps$-independent boundedness of $\intntom \ue^{p+1} \ve^{q-1}$.
When combined with certain estimates for the second solution component and assuming that \eqref{eq:a2:cond_chi} holds,
this bound implies a $\eps$-independent $L^r$ bound for $\ue$ for certain $r > 1$.

Let us again emphasize that
while the theorems proven in the first part of the present paper greatly simplify and shorten existence proofs for global generalized solutions of systems such as \eqref{prob:ks_log_sens},
they evidently cannot be used to obtain the required a~priori estimates for the first solution component.
That is, if the following lemma were not already contained in \cite{LankeitWinklerGeneralizedSolutionConcept2017}, its proof would still require substantial work.
\begin{lemma}\label{lm:a2:u_bdd}
  Assume \eqref{eq:a2:cond_chi}.
  Then there is $r > 1$ such that for all $T \gt 0$, we can find $\newgc{u_bdd} > 0$ with
  \begin{align}\label{eq:a2:u_bdd:bdd}
    \intntom \frac{|\nabla \ue|^2}{(\ue+1)^2} + \intntom \ue^r \le \gc{u_bdd}
    \qquad \text{for all $\eps \in (0, 1)$}.
  \end{align}
\end{lemma}
\begin{proof}
  This immediately follows from \cite[Lemma~4.3]{LankeitWinklerGeneralizedSolutionConcept2017} and \cite[Lemma~5.2]{LankeitWinklerGeneralizedSolutionConcept2017}.
\end{proof}

As already discussed, Lemma~\ref{lm:a2:u_bdd} now allows us to apply Theorem~\ref{th:strong_grad_conv} and thus to obtain solution candidates for \eqref{prob:ks_log_sens}.
\begin{lemma}\label{lm:a2:eps_sea_0_first}
  Assume \eqref{eq:a2:cond_chi}.
  There exist functions
  \begin{align*}
    0 \le u \in L_{\loc}^1(\Ombarinf)
  \end{align*}
  and
  \begin{align}\label{eq:a2:eps_sea_0_first:reg_v}
    0 \le v \in L_{\loc}^1([0, \infty); \sob11)
    \quad \text{with} \quad
    \frac{\nabla v}{v} \in L_{\loc}^2(\Ombarinf) 
  \end{align}
  and a null sequence $(\eps_j)_{j \in \N} \subset (0, 1)$ such that
  \begin{alignat}{2}
    \ue                    &\rh u                  &&\qquad \text{in $L_{\loc}^1(\Ombarinf)$}, \label{eq:a2:eps_sea_0_first:u_weak_l1} \\
    \ve                    &\ra v                  &&\qquad \text{in $L_{\loc}^1(\Ombarinf)$ and a.e.\ in $\Omega \times (0, \infty)$}, \label{eq:a2:eps_sea_0_first:v_pointwise} \\
    \frac{\nabla \ve}{\ve} &\ra \frac{\nabla v}{v} &&\qquad \text{in $L_{\loc}^2(\Ombarinf)$} \label{eq:a2:eps_sea_0_first:nabla_ln_v}
  \end{alignat}
  as $\eps = \eps_j \sea 0$.
  Moreover, $v$ fulfills \eqref{eq:a2:sol_concept:v_weak_sol}.
\end{lemma}
\begin{proof}
  According to the de la Vall\'ee Poussin theorem and \eqref{eq:a2:u_bdd:bdd}, $(\ue)_{\eps \in (0, 1)}$ is uniformly integrable in $\Omega \times (0, T)$ for all $T > 0$.
  Thus, by the Dunford--Pettis theorem and a diagonalization argument,
  we obtain a null sequence $(\eps_j)_{j \in \N} \subset (0, 1)$ and $u \in L_{\loc}^1(\Ombarinf)$ such that \eqref{eq:a2:eps_sea_0_first:u_weak_l1} holds.
  As the (norm) closure of the convex hull of $(\uej)_{j \in \N}$ is weakly closed and $\ue \ge 0$ for $\eps \in (0, 1)$ by Lemma~\ref{lm:a2:local_ex}, $u$ is also nonnegative.
 
  Setting $\fe \defs \ue$ and $\vne \defs v_0$ for $\eps \in (0, 1)$ and $f \defs u$,
  \eqref{eq:strong_grad_conv:f_conv} holds according to \eqref{eq:a2:eps_sea_0_first:u_weak_l1}, while \eqref{eq:strong_grad_conv:v0_conv} is a trivial statement.
  Therefore, Theorem~\ref{th:strong_grad_conv} provides us with a subsequence of $(\eps_j)_{j \in \N}$, which we do not relabel, and a function $v \colon \Ominf \ra \R$
  such that \eqref{eq:strong_grad_conv_gen:v_reg}--\eqref{eq:strong_grad_conv:weak_sol} hold.
  This already entails \eqref{eq:a2:sol_concept:v_weak_sol}, \eqref{eq:a2:eps_sea_0_first:reg_v} as well as \eqref{eq:a2:eps_sea_0_first:v_pointwise}
  and therefore, as $\ve  \ge 0$ for $\eps \in (0, 1)$ by Lemma~\ref{lm:a2:local_ex}, also $v \ge 0$.
   
  In order to finally verify \eqref{eq:a2:eps_sea_0_first:nabla_ln_v}, we first fix $T > 0$.
  As $(\vej^{-1}(\vej+1))_{j \in \N}$ is bounded in $\Omega \times (0, T)$ by Lemma~\ref{lm:a2:first_apriori} and converges a.e.\ in $\Omega \times (0, T)$ to $v^{-1} (1+v)$ by \eqref{eq:a2:eps_sea_0_first:v_pointwise}
  and since $(\vej+1)^{-1} \nabla \vej \ra (v+1)^{-1} \nabla v$  in $L^2(\Omega \times (0, T))$ as $j \ra \infty$ by \eqref{eq:strong_grad_conv:v_weighted_grad_conv},
  Lebesgue's theorem asserts that indeed $\vej^{-1} \nabla \vej \ra v^{-1} \nabla v$ in $L^2(\Omega \times (0, T))$ as $j \ra \infty$.
\end{proof}

Next, with the additional bound implied by \eqref{eq:a2:eps_sea_0_first:nabla_ln_v} at hand, we are able to estimate the time derivative of (a function of) the first solution component.
\begin{lemma}\label{lm:a2:ln_u_t}
  Assume \eqref{eq:a2:cond_chi}.
  Let $T > 0$ and $(\eps_j)_{j \in \N}$ be as given by Lemma~\ref{lm:a2:eps_sea_0_first}.
  Then there exists $\newgc{ln_u_t} > 0$ such that
  \begin{align*}
    \|(\ln(\uej+1))_t\|_{L^1((0, T); \dual{\sob{n}{2}})} \le \gc{ln_u_t}
    \qquad \text{for all $j \in \N$}.
  \end{align*}
\end{lemma}
\begin{proof}
  As testing the first equation in \eqref{prob:ks_log_sens_eps} with $\frac{\varphi}{\ue+1}$, integrating by parts and applying Hölder's inequality reveal that
  \begin{align*}
          \left| \intom (\ln(\ue+1))_t \varphi \right|
    &\le  \left| \intom \frac{|\nabla \ue|^2}{(\ue+1)^2} \varphi \right|
          + \chi \left| \intom \frac{\ue \nabla \ue \cdot \nabla \ve}{(1+\eps\ue) (\ue+1)^2 \ve} \varphi \right| \\
    &\pe  + \left| \intom \frac{\nabla \ue \cdot \nabla \varphi}{\ue+1} \right|
          + \chi \left| \intom \frac{\ue \nabla \ve \cdot \nabla \varphi}{(1+\eps\ue)(\ue+1)\ve} \right| \\
    &\le  \|\varphi\|_{\leb\infty} \left(
            \intom \frac{|\nabla \ue|^2}{(\ue+1)^2}
            + \chi \intom \frac{|\nabla \ue \cdot \nabla \ve|}{(\ue+1) \ve}
          \right) \\
    &\pe  + \|\varphi\|_{\sob12} \left(
            \intom \frac{|\nabla \ue|^2}{(\ue+1)^2}
            + \chi \intom \frac{|\nabla \ve|^2}{\ve^2}
            + 1 + \chi
          \right)
  \end{align*}
  holds in $(0, T)$ for all $\varphi \in \con\infty$ and all $\eps \in (0, 1)$,
  we obtain the statement due to the embedding $\sob n2 \embed \leb\infty$ and the bounds implied by \eqref{eq:a2:u_bdd:bdd} and \eqref{eq:a2:eps_sea_0_first:nabla_ln_v}.
\end{proof}

Combining the a priori estimates for the first solution component derived above, we now can obtain stronger convergence properties than asserted in \eqref{eq:a2:eps_sea_0_first:u_weak_l1}.
\begin{lemma}\label{lm:a2:eps_sea_0_u_second}
  Assume \eqref{eq:a2:cond_chi} and let $u$ and $(\eps_j)_{j \in \N}$ be as given by Lemma~\ref{lm:a2:eps_sea_0_first}.
  Then there exists a subsequence of $(\eps_j)_{j \in \N}$, which we do not relabel, such that
  \begin{alignat}{2}
    \ue                  & \ra u                &&\qquad \text{in $L_{\loc}^1(\Ombarinf)$ and a.e.\ in $\Omega \times (0, \infty)$}, \label{eq:a2:eps_sea_0_u:u_l1} \\
    \ue(\cdot, t)        & \ra u(\cdot, t)      &&\qquad \text{in $\leb1$ for a.e.\ $t \in (0, \infty)$}, \label{eq:a2:eps_sea_0:ut_l1} \\
    \nabla \ln(\ue+1)    & \rh \nabla \ln(u+1)  &&\qquad \text{in $L_{\loc}^2(\Ombarinf)$}, \label{eq:a2:eps_sea_0_u:grad_u}
  \end{alignat}
  as $\eps = \eps_j \sea 0$.
  Moreover, there is a null set $N \subset (0, \infty)$ such that \eqref{eq:a2:mass_conservation} holds.
\end{lemma}
\begin{proof}
  Thanks to Lemma~\ref{lm:a2:u_bdd} and Lemma~\ref{lm:a2:ln_u_t},
  the Aubin--Lions lemma and a diagonalization argument show the existence of a function $z \in L_{\loc}^2(\Ombarinf)$ and a subsequence of $(\eps_j)_{j \in \N} \subset (0, 1)$, which we do not relabel,
  such that $(\ln(\uej+1))_{j \in \N}$ converges strongly  in $L_{\loc}^2(\Ombarinf)$ and weakly in $L_{\loc}^2([0, \infty); \sob12)$ to $z$.
  Upon switching to a subsequence, we may also assume that $\ln(\uej+1) \ra z$ and thus $\uej \ra \ure^z-1$ a.e.\ in $\Omega \times (0, \infty)$ as $j \ra \infty$.
  Combining \eqref{eq:a2:eps_sea_0_first:u_weak_l1} with \eqref{eq:a2:u_bdd:bdd} and Vitali's theorem, we see that $u = \ure^{z}-1$ and \eqref{eq:a2:eps_sea_0_u:u_l1} hold,
  whenceupon we can also conclude \eqref{eq:a2:eps_sea_0_u:grad_u}.
  Possibly after a final subsequence extraction, we then infer \eqref{eq:a2:eps_sea_0:ut_l1} from \eqref{eq:a2:eps_sea_0_u:u_l1}.
  As integrating the first equation in \eqref{eq:a2:sol_concept:u_mass} shows that the mass of the first solution component is conserved on the approximate level,
  the latter implies \eqref{eq:a2:mass_conservation}.
\end{proof}

These convergence properties now allow us to verify that $u$ is a weak $\ln$-supersolution of the first subproblem in \eqref{prob:ks_log_sens}.
\begin{lemma}\label{lm:a2:u_ln_supersol}
  Assume \eqref{eq:a2:cond_chi} and let $(u, v)$ be as given by Lemma~\ref{lm:a2:eps_sea_0_first}.
  Then \eqref{eq:a2:sol_concept:u_ln_supersol} holds; that is, $u$ is a weak $\ln$-supersolution of the first subproblem in \eqref{prob:ks_log_sens}.
\end{lemma}
\begin{proof}
  We fix $0 \le \varphi \in C_c^\infty(\Ombarinf)$.
  By testing the first equation in \eqref{prob:ks_log_sens_eps} with $\frac{\varphi}{\ue+1}$ and integrating by parts,
  we see that
  \begin{align}\label{eq:a2:u_ln_supersol:ue_ln_supersol}
    &\pe  - \intninfom \ln(\ue+1) \varphi_t
          - \intom \ln(\une+1) \varphi(\cdot, 0) \notag \\
    &=     \intninfom \frac{|\nabla \ue|^2}{(\ue+1)^2} \varphi
          - \intninfom \frac{\nabla \ue \cdot \nabla \varphi}{\ue+1} \notag \\
    &\pe  - \chi \intninfom \frac{\ue \nabla \ue \cdot \nabla \ve}{(1+\eps\ue) (\ue+1)^2 \ve} \varphi
          + \chi \intninfom \frac{\ue \nabla \ve \cdot \nabla \varphi}{(1+\eps\ue)(\ue+1)\ve}
    \qquad \text{for all $\eps \in (0, 1)$}.
  \end{align}
  With $(\eps_j)_{j \in \N}$ as given by Lemma~\ref{lm:a2:eps_sea_0_u_second},
  we have by \eqref{eq:a2:eps_sea_0_u:grad_u} and the weak lower semicontinuity of the norm,
  \begin{align*}
        \liminf_{j \ra \infty} \intninfom \frac{|\nabla \uej|^2}{(\uej+1)^2} \varphi
    \ge \intninfom \frac{|\nabla u|^2}{(u+1)^2} \varphi.
  \end{align*}
  As thanks to Lemma~\ref{lm:a2:eps_sea_0_first} and Lemma~\ref{lm:a2:eps_sea_0_u_second}
  (and especially the fact that \eqref{eq:a2:eps_sea_0_first:nabla_ln_v} asserts strong $L^2$ convergence of $(\vej^{-1} \nabla \vej)_{j \in \N}$)
  the remaining terms in \eqref{eq:a2:u_ln_supersol:ue_ln_supersol} converge to their counterparts without $\eps$ as $\eps = \eps_j \sea 0$,
  we can conclude that \eqref{eq:a2:sol_concept:u_ln_supersol} is indeed fulfilled.
\end{proof}

We finally note that all statements in Theorem~\ref{th:a2} are already entailed by the above lemmata.
\begin{proof}[Proof of Theorem~\ref{th:a2}]
  We claim that the pair $(u, v)$ constructed in Lemma~\ref{lm:a2:eps_sea_0_first} is indeed a global generalized solution of \eqref{prob:ks_log_sens} in the sense of Definition~\ref{def:a2:sol_concept}.
  Indeed, nonnegativity of both solution components, the required regularity properties and \eqref{eq:a2:mass_conservation}, which entails \eqref{eq:a2:sol_concept:u_mass},
  have been verified in Lemma~\ref{lm:a2:eps_sea_0_first} and Lemma~\ref{lm:a2:eps_sea_0_u_second}.
  Moreover, \eqref{eq:a2:sol_concept:u_ln_supersol} has been shown in Lemma~\ref{lm:a2:u_ln_supersol}.
\end{proof}

\section{Application III: chemotaxis systems with logarithmic sensitivity, signal consumption and superlinear dampening}\label{sec:a3}
As a final application of the theorems proven in the first part of the present paper, we now consider the system
\begin{align}\label{prob:abs_log_sens}
  \begin{cases}
    u_t = \Delta u - \chi \nabla \cdot (\frac{u}{v} \nabla v) + g(u) & \text{in $\Omega \times (0, \infty)$}, \\
    v_t = \Delta v - uv                                              & \text{in $\Omega \times (0, \infty)$}, \\
    \partial_\nu u = \partial_\nu v= 0                               & \text{on $\partial \Omega \times (0, \infty)$}, \\
    u(\cdot, 0) = u_0, v(\cdot, 0) = v_0                             & \text{in $\Omega$},
  \end{cases}
\end{align}
where $\chi > 0$ and $g \colon [0, \infty] \ra \R$ is a given function,
which (for $g \equiv 0$) has been proposed by Keller and Segel to model travelling bands of \emph{E.\ coli} bacteria \cite{KellerSegelTravelingBandsChemotactic1971}.
These organisms, with density $u$, are assumed to partially direct their movement towards higher concentration of (the logarithm of) a signal substance such as oxygen with density $v$.
In contrast to the chemotaxis systems analyzed in Section~\ref{sec:a1} and Section~\ref{sec:a2}, the chemoattractant is now supposed to be consumed rather than produced.
Apart from survey \cite{BellomoEtAlMathematicalTheoryKeller2015},
we also refer to the introductions of Section~\ref{sec:a1} and Section~\ref{sec:a2} for a brief biological motivation for chemotaxis systems and logarithmic density functions.

At first glance, one might suspect that it is easier to obtain global solutions for chemotaxis systems with signal consumption than for similar systems with signal production.
After all, an $L^\infty$ bound for the second solution component directly follows from the maximum principle.
However, as (also) $\nabla v$ enters the first equation, this bound does not immediately translate into estimates for the first solution component.
Moreover, the consumption term $-uv$ in the second equation might force $v$ to become very small (also in finite time);
that is, unlike in Section~\ref{sec:a2} we cannot easily obtain a (local in time) upper bound for the factor $\frac1v$ appearing in the first equation.

As a consequence, the knowledge on global solvability of the system \eqref{prob:abs_log_sens} is yet rather limited.
Positive results in this direction all require smallness conditions (on the parameters or the space dimension) or symmetry assumptions, propose weak or generalized solution concepts
or are limited to certain simplified systems.
More concretely, for \eqref{prob:abs_log_sens} with $g \equiv 0$,
global bounded solutions in $\R^2$ and $\R^3$ have been constructed in \cite{WangEtAlAsymptoticDynamicsSingular2016} under certain smallness conditions,
global generalized solutions have been obtained in \cite{WinklerTwodimensionalKellerSegel2016} for two-dimensional domains
(see also \cite{WangGlobalLargedataGeneralized2016} and \cite{LiuLargetimeBehaviorTwodimensional2021} for corresponding results for systems additionally coupled with a fluid equation)
and the existence of global renormalized solutions has been shown in \cite{WinklerRenormalizedRadialLargedata2018} for the radially symmetric setting.
Moreover, global (locally) bounded solutions exist if one modifies the system to account for
nonlinear diffusion~\cite{LankeitLocallyBoundedGlobal2017},
saturated taxis sensitivity~\cite{LiuGlobalClassicalSolution2018}
or weaker consumption terms~\cite{LankeitViglialoroGlobalExistenceBoundedness2020}.

As already discussed in Section~\ref{sec:a1}, a natural way to amend chemotaxis systems, both from a biological viewpoint and as an attempt to make global existence results achievable,
is to introduce a dampening term in the first equation; that is to chose a function $g$ with $g(s) \ra -\infty$ for $s \ra \infty$.
For the typical logistic source term, $g(s) \defs \kappa s - \mu s^2$ for $s \ge 0$ and $\kappa, \mu > 0$,
the corresponding system has been studied by Lankeit and Lankeit;
in \cite{LankeitLankeitClassicalSolutionsLogistic2019} they have obtained global classical solutions provided $\chi$ is sufficiently small and $\mu$ is sufficiently large
and in \cite{LankeitLankeitGlobalGeneralizedSolvability2019} they have constructed global generalized solutions for arbitrary positive parameters.

In this section, we improve on the latter result and show that for global solvability in a generalized sense, it suffices to require that
\begin{align}\label{a3:cond_g}
  g \in C^1([0, \infty))
  \quad \text{fulfills} \quad
  g(0) \ge 0
  \quad \text{and} \quad
  \frac{g(s)}{s} \ra -\infty \text{ as } s \ra \infty.
\end{align}
Our main result is the following
\begin{theorem}\label{th:a3}
  Suppose $\Omega \subset \R^n$, $n \in \N$, is a smooth bounded domain, $\chi > 0$, that \eqref{a3:cond_g} is fulfilled
  and that $u_0 \in \leb1$ and $v_0 \in \leb\infty$ fulfill $u_0 \ge 0$ and $v_0 \ge \infv$ a.e.\ in $\Omega$ for some constant $\infv > 0$.
  Then there exists a global generalized solution $(u, v)$ of \eqref{prob:abs_log_sens} in the sense of Definition~\ref{def:a3:sol_concept} below.
\end{theorem}

\subsection{Solution concept}\label{sec:a3:sol_concept}
The solution concept we propose in the present section is very similar to the one we presented in Definition~\ref{def:a2:sol_concept}.
Again, we require that the first solution component is a weak $\ln$-supersolution fulfilling a certain mass inequality and that the second solution component is a weak solution of the corresponding subproblems.
\begin{definition}\label{def:a3:sol_concept}
  Let $\Omega \subset \R^n$, $n \in \N$, be a smooth, bounded domain, $\chi > 0$, $g \in C^0([0, \infty))$ and $0 \le u_0 \in \leb1$ and $0 \le v_0 \in \leb\infty$.
  We call a tuple of nonnegative functions $(u, v) \in L_{\loc}^1(\Ombarinf) \times L^\infty(\Ominf)$ with 
  \begin{align*}
    \frac{\nabla u}{u+1} \in L_{\loc}^2(\Ombarinf)
     \quad \text{and} \quad
    g(u) \in  L_{\loc}^1(\Ombarinf), \quad
  \intertext{as well as}
    \nabla v \in L_{\loc}^1(\Ombarinf)
    \quad \text{and} \quad
    \frac{\nabla v}{v} \in L_{\loc}^2(\Ombarinf)
  \end{align*}
  a \emph{global generalized solution} of \eqref{prob:abs_log_sens}, if
  \begin{itemize}
    \item
      $u$ is a weak $\ln$-supersolution of the first subproblem in \eqref{prob:abs_log_sens} in the sense that
      \begin{align}\label{eq:a3:sol_concept:u_ln_supersol}
        &\pe  - \intninfom \ln(u+1) \varphi_t
              - \intom \ln(u_0+1) \varphi(\cdot, 0) \notag \\
        &\ge  \intninfom \frac{|\nabla u|^2}{(u+1)^2} \varphi
              - \intninfom \frac{\nabla u \cdot \nabla \varphi}{u+1} \notag \\
        &\pe  - \chi \intninfom \frac{u \nabla u \cdot \nabla v}{(u+1)^2 v} \varphi
              + \chi \intninfom \frac{u \nabla v \cdot \nabla \varphi}{(u+1)v} \notag \\
        &\pe  - \intninfom \frac{g(u)}{u+1} \varphi
        \qquad \text{for all $0 \le \varphi \in C_c^\infty(\Ombarinf)$},
      \end{align}

    \item
      there is a null set $N \subset (0, \infty)$ such that
      \begin{align}\label{eq:a3:sol_concept:u_mass}
        \intom u(\cdot, t) \le \intom u_0 - \intnstom g(u)
        \qquad \text{for all $t \in (0, \infty) \setminus N$ and}
      \end{align}

    \item
      $v$ is a weak solution of the second subproblem in \eqref{prob:abs_log_sens} in the sense that
      \begin{align}\label{eq:a3:sol_concept:v_weak_sol}
          - \intninfom v \varphi_t
          - \intom v_0 \varphi(\cdot, 0)
        = - \intninfom \nabla v \cdot \nabla \varphi
          - \intninfom uv \varphi
        \qquad \text{for all $\varphi \in C_c^\infty(\Ombarinf)$}.
      \end{align}
  \end{itemize}
\end{definition}

\begin{remark}
  Similarly as in Lemma~\ref{lm:a1:sol_concept_sensible} and Lemma~\ref{lm:a2:sol_concept_consistent}, one can verify that the solution concept defined above is consistent with classical solvability
  in the sense that if $(u, v)$ is a sufficiently smooth global generalized solution of \eqref{prob:abs_log_sens} in the sense of Definition~\ref{a3:cond_g},
  then $(u, v)$ solves \eqref{prob:abs_log_sens} also in the classical sense.
\end{remark}

\begin{remark}
  Let us compare Definition~\ref{def:a3:sol_concept} with the solution concepts
  proposed in \cite{WinklerTwodimensionalKellerSegel2016}, \cite{WinklerRenormalizedRadialLargedata2018} and \cite{LankeitLankeitGlobalGeneralizedSolvability2019},
  where systems resembling \eqref{prob:abs_log_sens} have been investigated.
  First, in \cite{WinklerTwodimensionalKellerSegel2016}, where \eqref{prob:abs_log_sens} is considered in two-dimensional domains for $g \equiv 0$, essentially the same definition is used.
  Second, \cite{WinklerRenormalizedRadialLargedata2018} discusses \eqref{prob:abs_log_sens} in the radially symmetric setting for $g \equiv 0$
  and replaces \eqref{eq:a3:sol_concept:u_ln_supersol} and \eqref{eq:a3:sol_concept:u_mass} with a renormalized solution formulation;
  that is, it is required that certain variants of \eqref{eq:a3:sol_concept:u_ln_supersol} hold with equality.
  Third, \cite{LankeitLankeitGlobalGeneralizedSolvability2019} analyzes \eqref{prob:abs_log_sens} for quadratically growing $-g$
  and proposes a solution concept similar to the one defined above, the main difference being that instead of \eqref{eq:a3:sol_concept:u_mass},
  $u$ is required to be a weak subsolution of the first subproblem in \eqref{prob:abs_log_sens}.

  That is, in contrast to the solution concepts introduced in Section~\ref{sec:a1} and Section~\ref{sec:a2},
  the solutions constructed below are only shown to satisfy a weaker definition than those proposed for related systems in the literature.
  Unlike in the previous two sections, however,
  we now do not show how Theorem~\ref{th:strong_grad_conv} can be used to rather quickly construct global solutions in a stronger sense than already established for the \emph{same} systems,
  but make use of Theorem~\ref{th:strong_grad_conv} (and Theorem~\ref{th:stronger_grad_conv}) to extend the class of systems for which \emph{some} global generalized solutions can be obtained.
  In particular, our findings are not limited to two-dimensional domains, radially symmetric settings or quadratically growing $-g$.
\end{remark}

\subsection{Existence of a global generalized solution: proof of Theorem~\ref{th:a3}}
We henceforth fix a smooth, bounded domain $\Omega \subset \R^n$, $n \in \N$, $\chi > 0$, $g$ as in \eqref{a3:cond_g},
$u_0 \in \leb1$ and $v_0 \in \leb\infty$ with $u_0 \ge 0$ and $v_0 \ge \infv$ a.e.\ in $\Omega$ for some constant $\infv > 0$
and $(\une)_{\eps \in (0, 1)}, (\vne)_{\eps \in (0, 1)} \subset \con\infty$ satisfying 
\begin{align}\label{eq:a3:une_vne_reg}
  0 \le \une
  \quad \text{and} \quad
  \infv \le \vne \le \|v_0\|_{\leb\infty}
  \qquad \text{for all $\eps \in (0, 1)$}
\end{align}
as well as
\begin{align}\label{eq:a3:une_vne_conv}
  \une \ra u_0
  \quad \text{and} \quad
  \vne \ra v_0
  \qquad \text{in $\leb1$ and a.e.\ in $\Omega \times (0, \infty)$ as $\eps \sea 0$}.
\end{align}

The existence of global classical solutions to certain approximate problems, limits of which will eventually be seen to be generalized solutions of \eqref{prob:ks_log_sens}, is stated in the following
\begin{lemma}\label{lm:a3:ex_ue_ve}
  For each $\eps \in (0, 1)$, there exists a tuple of functions
  \begin{align*}
    (\ue, \ve) \in (C^0(\Ombarinf) \times C^{2, 1}(\Ombar \times (0, \infty)))^2
  \end{align*}
  with $\ue \ge 0$ and $\ve \gt 0$ in $\Ombarinf$,
  solving
  \begin{align}\label{prob:abs_log_sens_eps}
    \begin{cases}
      \uet = \Delta \ue - \chi \nabla \cdot \big( \frac{\ue}{(1+\eps\ue)\ve} \nabla \ve) + g(\ue) & \text{in $\Omega \times (0, \infty)$}, \\
      \vet = \Delta \ve - \frac{\ue \ve}{(1+\eps\ue)(1+\eps\ve)} & \text{in $\Omega \times (0, \infty)$}, \\
      \partial_\nu \ue = \partial_\nu \ve = 0 & \text{on $\partial \Omega \times (0, \infty)$}, \\
      \ue(\cdot, 0) = \une, \ve(\cdot, 0) = \vne & \text{in $\Omega$}
    \end{cases}
  \end{align}
  classically.
\end{lemma}
\begin{proof}
  For $g(s) = \kappa s - \mu s^2$, $s \ge 0$ and $\kappa, \mu > 0$, this has been shown in \cite[Lemma~3.1]{LankeitLankeitGlobalGeneralizedSolvability2019}.
  The properties of $g$ made use of there are
  $g \in \con1$ (to guarantee local classical solvability),
  $g(0) \ge 0$ (to conclude nonnegative of $\ue$ by the maximum principle) and
  finiteness of $\sup_{s \ge 0} g(s)$ (to be able estimate $g$ from above by a constant).
  All of these are implied by \eqref{a3:cond_g}, so that the statement can be derived exactly as in \cite[Lemma~3.1]{LankeitLankeitGlobalGeneralizedSolvability2019}.
\end{proof}

In the sequel, we fix the solution $(\ue, \ve)$ given by Lemma~\ref{lm:a3:ex_ue_ve} for $\eps \in (0, 1)$.
We begin by collecting basic a priori estimates rapidly implied by the dampening term in the first and the consumption term in the second equation in \eqref{prob:abs_log_sens_eps}.
\begin{lemma}\label{lm:a3:basic_apriori}
  Let $T \gt 0$. Then there exist $\newgc{u_l1}$, $\newgc{g_u_bdd} > 0$ such that
   \begin{align}\label{eq:a3:basic_apriori:u_l1}
     \intom \ue(\cdot, t) \le \gc{u_l1}
     \qquad \text{for all $t \in (0, T]$ and $\eps \in (0, 1)$}.
   \end{align}
   and
   \begin{align}\label{eq:a3:basic_apriori:g_u}
     \intntom |g(\ue)| \le \gc{g_u_bdd}
     \qquad \text{for all $\eps \in (0, 1)$}.
   \end{align}
   Moreover,
   \begin{align}\label{eq:a3:basic_apriori:v_infty}
     \ve \le \|v_0\|_{\leb\infty}
     \qquad \text{in $\Ominf$ for all $\eps \in (0, 1)$.}
   \end{align}
\end{lemma}
\begin{proof}
  By \eqref{a3:cond_g}, $\newlc{g_pos} \defs \sup \{\,s \ge 0 : g(s) \ge 0\,\}$ is well-defined and finite.
  As then $g \le \newlc{g_max} - |g|$ in $[0, \infty)$ for $\lc{g_max} \defs 2\|g\|_{C^0([0, \lc{g_pos}])}$,
  integrating the first equation in \eqref{prob:abs_log_sens_eps} yields
  \begin{align*}
        \intom \ue(\cdot, T)
        - \intom \une
    =   \intntom g(\ue)
    \le |\Omega| T \lc{g_max} 
        - \intntom |g(\ue)|
    \qquad \text{for all $\eps \in (0, 1)$}.
  \end{align*}
  Thanks to \eqref{eq:a3:une_vne_conv}, this implies \eqref{eq:a3:basic_apriori:u_l1} and \eqref{eq:a3:basic_apriori:g_u} for certain $\gc{u_l1}, \gc{g_u_bdd} > 0$.
  Moreover, according to \eqref{eq:a3:une_vne_reg}, $\ol v_\eps \defs \|v_0\|_{\leb\infty}$ is a supersolution of the second subproblem in \eqref{prob:abs_log_sens_eps},
  which implies \eqref{eq:a3:basic_apriori:v_infty}.
\end{proof}

In order to obtain a priori estimates also for the gradients of both solution components,
we make use of the special structure in \eqref{prob:abs_log_sens} (or, more precisely, in \eqref{prob:abs_log_sens_eps})
and observe that
\begin{align*}
  [0, \infty) \ni t \mapsto - \left(\intom \ln(\ue(\cdot, t)+1) + \chi^2 \intom \ln \ve(\cdot, t) \right)
\end{align*}
defines a quasi-energy functional for each $\eps \in (0, 1)$.
Boundedness of the corresponding dissipative terms implies the following
\begin{lemma}\label{lm:a3:nabla_ln_u}
  Let $T \gt 0$. There exists $\newgc{grad_ln_u_bdd} > 0$ such that
  \begin{align*}
    \intntom \frac{|\nabla \ue|^2}{(\ue+1)^2} \le \gc{grad_ln_u_bdd}
    \qquad \text{for all $\eps \in (0, 1)$}.
  \end{align*}
\end{lemma}
\begin{proof}
  By testing the first and second equation in \eqref{prob:abs_log_sens_eps} with $-\frac{1}{\ue+1}$ and $-\frac{1}{\ve}$, respectively, we obtain
  \begin{align*}
      - \ddt \intom \ln(\ue+1)
    = - \intom \frac{|\nabla \ue|^2}{(\ue+1)^2}
      + \chi \intom \frac{\ue \nabla \ue \cdot \nabla \ve}{(1+\eps \ue)(\ue+1)^2 \ve}
      - \intom \frac{g(\ue)}{1+\ue}
  \end{align*}
  and
  \begin{align*}
      - \ddt \intom \ln(\ve)
    = - \intom \frac{|\nabla \ve|^2}{\ve^2}
      + \intom \frac{\ue \ve}{(1+\eps\ue)(1+\eps\ve)\ve} 
  \end{align*}
  in $(0, T)$ for all $\eps \in (0, 1)$.
  Therefore,
  \begin{align*}
        - \ddt \left(\intom \ln(\ue+1) + \chi^2 \intom \ln \ve \right)
    \le - \frac12 \intom \frac{|\nabla \ue|^2}{(\ue+1)^2}
        - \frac{\chi^2}{2} \intom \frac{|\nabla \ve|^2}{\ve}
        + \intom |g(\ue)|
        + \intom \ue
  \end{align*}
  in $(0, T)$ for all $\eps \in (0, 1)$
  and thus (since $\ln(s + 1) \le s$ for $s \ge 0$)
  \begin{align*}
    &\pe  \frac12 \intntom \frac{|\nabla \ue|^2}{(\ue+1)^2}
          + \frac{\chi^2}{2} \intntom \frac{|\nabla \ve|^2}{\ve} \\
    &\le  \intom \ln(\ue(\cdot, T)+1)
          + \intntom |g(\ue)|
          + \intntom \ue
          + \chi^2 \intom \ln \ve
          - \chi^2 \intom \ln \vne \\
    &\le  (T+1) \gc{u_l1}
          + \gc{g_u_bdd}
          + \chi^2 \intom v_0
          + \chi^2 |\Omega| |\ln \infv|.
    \qedhere
  \end{align*}
\end{proof}

As usual (cf.\ Lemma~\ref{lm:a3:phi_t} and Lemma~\ref{lm:a2:ln_u_t}, for instance),
from bounds of the above type we can infer a priori estimates for the time derivative of (a function of) the first solution component.
\begin{lemma}\label{lm:a3:ln_u_t}
  Let $T > 0$. There exists $\newgc{ln_u_t} > 0$ such that
  \begin{align*}
    \|\ln(\ue+1)\|_{L^1((0, T); \dual{\sob{n}{2}})} \le \gc{ln_u_t}.
  \end{align*}
\end{lemma}
\begin{proof}
  Since
  \begin{align*}
          \left| \intom (\ln(\ue + 1))_t \varphi \right|
    &\le  \intom \frac{|\nabla \ue|^2}{(\ue+1)^2} |\varphi|
          + \chi \intom \frac{\ue |\nabla \ue \cdot \nabla \ve|}{(1+\eps \ue)(\ue+1)^2 \ve} |\varphi|
          + \intom \frac{|g(\ue)|}{\ue+1} |\varphi| \\
    &\pe  + \intom \frac{|\nabla \ue \cdot \nabla \varphi|}{\ue+1}
          + \chi \intom \frac{\ue |\nabla v \cdot \nabla \varphi|}{(1+\eps \ue)(\ue+1) \ve} \\
    &\le  \|\varphi\|_{\leb\infty} \left(
            \frac{2+\chi^2}{2} \intom \frac{|\nabla \ue|^2}{(\ue+1)^2}
            + \frac{\chi^2}{2} \intom \frac{|\nabla \ve|^2}{\ve^2}
            + \intom |g(\ue)|
          \right) \\
    &\pe  + \|\varphi\|_{\sob12} \left(
            \intom \frac{|\nabla \ue|^2}{(\ue+1)^2}
            + \chi \intom \frac{|\nabla \ve|^2}{\ve^2}
            + 1 + \chi
          \right)
  \end{align*}
  for all $\eps \in (0, 1)$ and $\sob{n}{2} \embed \leb\infty$, the statement follows from Lemma~\ref{lm:a3:nabla_ln_u} and lemma~\ref{lm:a3:basic_apriori}.
\end{proof}

Combining the estimates above and making use of various compactness theorems allows us to obtain a limit function $u$ with the following properties.
\begin{lemma}\label{lm:a3:eps_sea_0_u}
  There exist $0 \le u \in L_{\loc}^1(\Ombarinf)$ with $g(u) \in L_{\loc}^1(\Ombarinf)$ and $\nabla \ln(u+1) \in L_{\loc}^2(\Ombarinf)$
  and a null sequence $(\eps_j)_{j \in \N} \subset (0, 1)$ such that
  \begin{alignat}{2}
    \ue                  & \ra u                &&\qquad \text{in $L_{\loc}^1(\Ombarinf)$ and a.e.\ in $\Omega \times (0, \infty)$}, \label{eq:a3:eps_sea_0_u:u_l1} \\
    \ue(\cdot, t)        & \ra u(\cdot, t)      &&\qquad \text{in $\leb1$ for a.e.\ $t \in (0, \infty)$}, \label{eq:a3:eps_sea_0:ut_l1} \\
    \frac{g(\ue)}{\ue+1} & \ra \frac{g(u)}{u+1} &&\qquad \text{in $L_{\loc}^1(\Ombarinf)$}, \label{eq:a3:eps_sea_0_u:g_u} \\
    \nabla \ln(\ue+1)    & \rh \nabla \ln(u+1)  &&\qquad \text{in $L_{\loc}^2(\Ombarinf)$} \label{eq:a3:eps_sea_0_u:grad_u}
  \end{alignat}
  as $\eps = \eps_j \sea 0$.
  Moreover, there is a null set $N \subset (0, \infty)$ such that \eqref{eq:a3:sol_concept:u_mass} holds.
\end{lemma}
\begin{proof}
  Making use of Lemma~\ref{lm:a3:nabla_ln_u} and Lemma~\ref{lm:a3:ln_u_t} instead of Lemma~\ref{lm:a2:u_bdd} and Lemma~\ref{lm:a2:ln_u_t},
  we obtain $0 \le u \in L_{\loc}^1(\Ombarinf)$ with $\nabla \ln(u+1) \in L_{\loc}^2(\Ombarinf)$
  and a null sequence $(\eps_j)_{j \in \N} \subset (0, 1)$ such that \eqref{eq:a3:eps_sea_0_u:u_l1}, \eqref{eq:a3:eps_sea_0:ut_l1} and \eqref{eq:a3:eps_sea_0_u:grad_u} hold
  by the same reasoning as in Lemma~\ref{lm:a2:eps_sea_0_u_second}.
  Moreover, \eqref{eq:a3:eps_sea_0_u:u_l1}, \eqref{eq:a3:basic_apriori:g_u} and Vitali's theorem imply \eqref{eq:a3:eps_sea_0_u:g_u}.

  Noting that \eqref{a3:cond_g} entails $\inf_{s \in [0, \infty)} (-g(s)) \gt -\infty$,
  we infer from \eqref{eq:a3:eps_sea_0_u:u_l1}, Fatou's lemma, an integration of the first equation in \eqref{prob:abs_log_sens_eps} and \eqref{eq:a3:eps_sea_0:ut_l1} that
  \begin{align*}
        - \intnstom g(u)
    \le \liminf_{j \ra \infty} \left(- \intnstom g(\uej)\right)
    =   \liminf_{j \ra \infty} \left( \intom \uej(\cdot, t) - \intom \unej \right)
    =   \intom u(\cdot, t) - \intom u_0
  \end{align*}
  for all $t \in (0, \infty)$ for some null set $N \subset (0, \infty)$.
  Thus, $g(u) \in L_{\loc}^1(\Ombarinf)$ and \eqref{eq:a3:sol_concept:u_mass} holds.
\end{proof}

Regarding convergence properties of the second solution component, we once again rely on Theorem~\ref{th:strong_grad_conv}.
In contrast to the systems considered in the previous two sections, however,
this theorem is not directly applicable since we first need to make sure that $f_\eps \defs -\ue \ve$ converges in $L_{\loc}^1(\Ombarinf)$.
Additionally, we also crucially rely on Theorem~\ref{th:stronger_grad_conv}.
\begin{lemma}\label{lm:a3:eps_sea_0_v}
  There exist a nonnegative function $v \in L^\infty(\Ominf)$ satisfying \eqref{eq:strong_grad_conv_gen:v_reg} as well as $\frac{\nabla v}{v} \in L^2(\Ombarinf)$
  and a subsequence of the null sequence given by Lemma~\ref{lm:a3:eps_sea_0_u}, which we still denote by $(\eps_j)_{j \in \N}$,
  such that \eqref{eq:strong_grad_conv:v_l1_conv}--\eqref{eq:strong_grad_conv:v_weighted_grad_conv} as well as \eqref{eq:a3:sol_concept:v_weak_sol} hold and that
  \begin{align}\label{eq:a3:eps_sea_0_v:nabla_ln_v}
    \frac{\nabla \vej}{\vej} &\ra \frac{\nabla v}{v} \qquad \text{in $L_{\loc}^2(\Ombarinf)$ as $j \ra \infty$}. 
  \end{align}
\end{lemma}
\begin{proof}
  We first apply Lemma~\ref{lm:limit_ve}
  (noting that $(f_\eps)_{\eps \in (0, 1)} \defs (-\frac{\ue \ve}{(1+\eps\ue)(1+\eps\ve)})_{\eps \in (0, 1)}$ is indeed bounded in $L_{\loc}^1(\Ombarinf)$
  according to \eqref{eq:a3:basic_apriori:u_l1} and \eqref{eq:a3:basic_apriori:v_infty})
  to obtain $v \in L_{\loc}^1(\Ombarinf)$ and a subsequence of $(\eps_j)_{j \in \N} \subset (0, 1)$, not relabeled, such that $\vej \ra v$ a.e.\ in $\Omega \times (0, \infty)$.
  As $(\ve)_{\eps \in (0, 1)}$ is bounded in $L^\infty(\Ominf)$ by \eqref{eq:a3:basic_apriori:v_infty},
  this implies $0 \le v \in L^\infty(\Ominf)$ and, when combined with \eqref{eq:a3:eps_sea_0_u:u_l1},
  also
  \begin{align*}
    \fej = -\frac{\uej \vej}{(1+\eps_j\uej)(1+\eps_j\vej)} \ra -uv \sfed f
    \qquad \text{in $L_{\loc}^1(\Ombarinf)$ as $j \ra \infty$}.
  \end{align*}
  Therefore, we can indeed apply Theorem~\ref{th:strong_grad_conv} (with $\kappa \defs 0$) to conclude that, upon switching to a subsequence,
  \eqref{eq:strong_grad_conv_gen:v_reg}--\eqref{eq:strong_grad_conv:v_weighted_grad_conv} and \eqref{eq:a3:sol_concept:v_weak_sol} hold.
  Since $\ve \ge 0$ by Lemma~\ref{lm:a3:ex_ue_ve}, nonnegativity of $v$ follows from \eqref{eq:strong_grad_conv:v_l1_conv}.

  Setting $\newlc{v0_inf} \defs \|v_0\|_{\leb\infty}$ and $\psi(s) = -\ln \frac{s}{\lc{v0_inf}}$ for $s \in [0, \lc{v0_inf}]$,
  we see that $\psi \in C^2((0, \lc{v0_inf}])$ is nonnegative and convex.
  Therefore, we can derive \eqref{eq:a3:eps_sea_0_v:nabla_ln_v} by applying Theorem~\ref{th:stronger_grad_conv} once we have verified that the assumptions of Theorem~\ref{th:stronger_grad_conv} are fulfilled.
  First, $X \defs \bigcup_{\eps \in (0, 1)} \ve(\Ombarinf) \subset (0, \lc{v0_inf}]$ (and thus $\psi \in C^2(X)$) follows from Lemma~\ref{lm:a3:ex_ue_ve} and \eqref{eq:a3:basic_apriori:v_infty},
  second, $\psi(\vne) \ra \psi(v_0)$ in $\leb1$ as $\eps \sea 0$ from \eqref{eq:a3:une_vne_reg} and \eqref{eq:a3:une_vne_conv}
  and third, $\psi'(\vej) \fej = -\uej(1+\eps_j\uej)^{-1}(1+\eps_j\vej)^{-1} \ra -u = \psi'(v) f$ in $L_{\loc}^1(\Ombarinf)$ as $j \ra \infty$ from \eqref{eq:a3:eps_sea_0_u:u_l1} and \eqref{eq:strong_grad_conv:v_l1_conv}.
\end{proof}

Due to convergence statements asserted by Lemma~\ref{lm:a3:eps_sea_0_u} and Lemma~\ref{lm:a3:eps_sea_0_v},
we are now able to show that $(u, v)$ also fulfills the remaining property in Definition~\ref{def:a3:sol_concept}.
\begin{lemma}\label{lm:a3:u_ln_supersol}
  Let $u$ and $v$ be as constructed in Lemma~\ref{lm:a3:eps_sea_0_u} and Lemma~\ref{lm:a3:eps_sea_0_v}, respectively.
  Then $u$ is a weak $\ln$-supersolution of the first subproblem in \eqref{prob:abs_log_sens}; that is, \eqref{eq:a3:sol_concept:u_ln_supersol} holds.
\end{lemma}
\begin{proof}
  We let $0 \le \varphi \in C_c^\infty(\Ombarinf)$ and $(\eps_j)_{j \in \N}$ be as given by Lemma~\ref{lm:a3:eps_sea_0_v}.
  We first note that by testing the first equation in \eqref{prob:abs_log_sens_eps} with $\frac{\varphi}{\ue+1}$,
  we obtain
  \begin{align}\label{eq:a3_u_ln_supersol:ue_ln_sol}
    &\pe  - \intninfom \ln(\ue+1) \varphi_t
          - \intom \ln(\une+1) \varphi(\cdot, 0) \notag \\
    &=    \intninfom \frac{|\nabla \ue|^2}{(\ue+1)^2} \varphi
          - \intninfom \frac{\nabla \ue \cdot \nabla \varphi}{\ue+1} \notag \\
    &\pe  - \chi \intninfom \frac{\ue \nabla \ue \cdot \nabla \ve}{(1+\eps \ue)(\ue+1)^2 \ve} \varphi
          + \chi \intninfom \frac{\ue \nabla \ve \cdot \nabla \varphi}{(1+\eps \ue)(\ue+1)\ve} \notag \\
    &\pe  - \intninfom \frac{g(\ue)}{\ue+1} \varphi
    \qquad \text{for all $\eps \in (0, 1)$}.
  \end{align}
  According to the weak lower semicontinuity of the norm and \eqref{eq:a3:eps_sea_0_u:grad_u}, we have
  \begin{align*}
        \lim_{j \ra \infty} \intninfom \frac{|\nabla \uej|^2}{(\uej+1)^2} \varphi
    \ge \intninfom \frac{|\nabla u|^2}{(u+1)^2} \varphi.
  \end{align*}
  Moreover, the convergence properties asserted by \eqref{eq:a3:une_vne_conv}, Lemma~\ref{lm:a3:eps_sea_0_u} and Lemma~\ref{lm:a3:eps_sea_0_v} imply
  that each of the remaining terms in \eqref{eq:a3_u_ln_supersol:ue_ln_sol} converges to its counterpart without $\eps$ as $\eps = \eps_j \sea 0$.
  For instance, 
  as the sequence $(\frac{\uej}{(1+\eps_j \uej)(\uej+1)})_{j \in \N}$ is bounded and convergent to $\frac{u}{1+u}$ a.e.\ in $\Omega \times (0, \infty)$ as $j \ra \infty$ by \eqref{eq:a3:eps_sea_0_u:u_l1}
  and due to \eqref{eq:a3:eps_sea_0_u:grad_u} and \eqref{eq:a3:eps_sea_0_v:nabla_ln_v}, we obtain
  \begin{align}\label{eq:a3_u_ln_supersol:lim}
        \lim_{j \ra \infty} \intninfom \frac{\uej \nabla \uej \cdot \nabla \vej}{(1+\eps_j \uej)(\uej+1)^2 \vej} \varphi
    =   \intninfom \frac{u \nabla u \cdot \nabla v}{(u+1)^2 v} \varphi.
  \end{align}
  (Let us remark that weak convergence in \eqref{eq:a3:eps_sea_0_v:nabla_ln_v}
  or strong convergence of merely $((\vej+1)^{-1}\nabla \vej)_{j \in \N}$, for instance, instead of $((\vej)^{-1}\nabla \vej)_{j \in \N}$ would be insufficient to conclude \eqref{eq:a3_u_ln_supersol:lim}.)
  In combination, this shows that indeed \eqref{eq:a3:sol_concept:u_ln_supersol} holds.
\end{proof}

Finally, combining these lemmata shows that the tuple $(u, v)$ constructed above is indeed a generalized solution of \eqref{prob:abs_log_sens} in the sense of Definition~\ref{def:a3:sol_concept}.
\begin{proof}[Proof of Theorem~\ref{th:a3}]
  Letting $u$ and $v$ be as given by Lemma~\ref{lm:a3:eps_sea_0_u} and Lemma~\ref{lm:a3:eps_sea_0_v}, respectively,
  we first note that nonnegativity and the desired regularity of $u$ and $v$ has already asserted by Lemma~\ref{lm:a3:eps_sea_0_u} and Lemma~\ref{lm:a3:eps_sea_0_v}.
  Moreover, the three required properties in Definition~\ref{def:a3:sol_concept} have been verified in Lemma~\ref{lm:a3:u_ln_supersol}, Lemma~\ref{lm:a3:eps_sea_0_u} and Lemma~\ref{lm:a3:eps_sea_0_v}.
\end{proof}

\end{document}